\documentclass[final,leqno]{siamltex}
\usepackage{array}
\usepackage{amssymb,amsmath,amsfonts}
\usepackage{graphicx}
\usepackage{color}
\usepackage{epstopdf}
\usepackage{enumerate}
\usepackage{pgfplots}
\usepackage{titlesec}
\usepackage{mathrsfs}
\usepackage{tikz}
\usepackage{booktabs}
\usepackage{subfigure}
\usepackage{colortbl}
\usepackage{mathrsfs}
\usepackage{multirow}
\usepackage[active]{srcltx}
\usepackage{comment}
\usepackage[mathscr]{euscript}
\usepackage{color}
\usepackage{colortbl}
\usepackage{placeins}
\usetikzlibrary{shapes.geometric}

\allowdisplaybreaks

\newcommand{\cD}{\mathcal{D}}

\newcommand{\bC}{\mathbb{C}}

\newcommand{\cM}{\mathcal{M}}
\newcommand{\cH}{\mathcal{H}}
\newcommand{\rA}{\mathscr{A}}
\newcommand{\rB}{\mathscr{B}}
\newcommand{\rT}{\mathscr{T}}
\newcommand{\Rn}{\widetilde{R}(t_n)}
\newcommand{\Rnn}{\widetilde{R}(t_{n+1})}
\newcommand{\Dt}{\tau}
\newcommand{\tomega}{\tilde{\omega}}
\newcommand{\sech}{{\rm sech}}
\newcommand{\asj}{\alpha_{s,j}}
\newcommand{\bsj}{\beta_{s,j}}
\newcommand{\Li}{L^\infty}
\definecolor{mygray}{gray}{0.95}

\newtheorem{remark}{Remark}

\newcolumntype{C}[1]{>{\vspace{0.5em}\begin{minipage}{#1}\centering\let\newline\\
\arraybackslash\hspace{0pt}}m{#1}<{\end{minipage}\vspace{0.5em}}}

\begin{document}
\graphicspath{{graphics/}}	
\title{Mass- and energy-conserved numerical schemes for nonlinear Schr\"odinger equations} 

\author{
Xiaobing Feng\thanks{Department of Mathematics, The University of
Tennessee, Knoxville, TN 37996, U.S.A.  ({\tt xfeng@math.utk.edu})}
\and
Hailiang Liu\thanks{Department of Mathematics, Iowa State University, Ames, IA 50011, U.S.A.
({\tt hliu@iastate.edu}).}
\and
Shu Ma\thanks{Department of Applied Mathematics, Northwestern Polytechnical University, 
Xian, Shaanxi, 710065, P. R. China.
({\tt mashu@mail.nwpu.edu.cn}). }
}

\maketitle
	
\begin{abstract}
In this paper, we propose a family of time-stepping schemes for approximating 
general nonlinear Schr\"odinger equations. 
The proposed schemes all satisfy both mass and energy conservation (in a modified form for the latter).
Truncation and dispersion error analyses are provided for {four proposed schemes}. 
Efficient fixed-point iterative solvers are also constructed to solve 
the resulting nonlinear discrete problems. As a byproduct, an efficient one-step implementation 
of the BDF schemes is obtained as well. 
Extensive numerical experiments are presented to demonstrate 
the convergence and the capability of capturing the blow-up time of the proposed schemes.
\end{abstract}
	
	

\pagestyle{myheadings}
\thispagestyle{plain}
\markboth{XIAOBING FENG, HAILIANG LIU and SHU MA}
{ENERGY-CONSERVED NUMERICAL METHODS FOR NLS EQUATIONS} 

\section{Introduction.}\label{sec-1}
In this paper we consider the following nonlinear 
Schr\"odinger (NLS) equation:   
\begin{alignat}{2}
\label{nls_eqn1}
i u_t &= -\Delta u +\lambda f(|u|^2) u &&\qquad \mbox{in } {\cD _T}: = \cD  \times (0,T),\\
\label{nls_eqn3}
u (0) &= u _0 &&\qquad\mbox{in }\cD, 
\end{alignat}
where $\cD\subset \mathbf{R}^d\, (d=1,2,3)$ is a bounded domain, $\lambda=\pm1$, $T>0$ and 
$i=\sqrt{-1}$ stands for the imaginary unit.  
$u=u(x,t): \cD_T \to \bC$ is  a complex-valued function. 
$f: \mathbf{R}_+\to \mathbf{R}_+$ is a given real-valued function,
which could be different in different applications, e.g., see
\cite{Pelinovsky1996Nonlinear, Sch1996Traveling} and the references therein. 
The best known $f$ is $f(s)=s$, which leads to the well-known nonlinear 
Schr\"odinger equation with cubic nonlinearity. To close the system, we 
also need to specify a boundary condition for $u$. In this paper we consider both  
homogeneous Dirichlet and periodic boundary condition (see Sections \ref{sec-7} and \ref{sec-8}), 
although most of our derivations and proofs are independent of the boundary condition. 
 
The above Schr\"odinger equation  
describes many physical phenomena in optics, mechanics, and plasma physics.
Mathematically, the NLS equation is a prototypical dispersive wave equation, 
its solutions exhibit some intriguing properties such as energy conservation, 
soliton wave, and possible blow-ups \cite{Bourgain1999Global, Tao2006Nonlinear}. 
In particular, the equation preserves both the mass and the Hamiltonian energy, 
that is, the following quantities are constants in time:
\begin{align}\label{energy1}
\cM(u)(t) &:= \|u(t)\|_{L^2}^2=\int_{\cD} |u(t)|^2\, dx, \\
\label{energy2}
\cH(u)(t) &:=\int_{\cD} \Bigl(|\nabla u(t)|^2+\lambda F(|u(t)|^2)\Bigr)\, dx,
\quad F(s)=\int_0^s f(\mu)\, d\mu. 
\end{align}
Here the dependence of $u$ on $x$ variable is suppressed for notational brevity. 
The case with positive $\lambda$ is called defocusing and with negative $\lambda$ is called focusing 
which allows for bright soliton solutions as well as breather solutions. 

Dispersion and nonlinearity can interact to produce permanent and localized wave forms 
in nonlinear dispersive wave equations such as the Korteweg-de Vries (KdV) equation
\cite{Lax2010Integrals, Miura1968Korteweg, Zabusky1965Interaction} 
and the cubic Schr\"odinger equation \cite{Calogero1988Spectral, Zwillinger1989Handbook}. 
A distinct feature of these equations is the infinite many conservation laws 
(conserved integrals as invariants), allowing for soliton solutions which emerge from 
collision unchanged over time. The quality of the numerical approximation hence hinges 
on how well the conserved integrals can be preserved at the discrete level. 
Numerical methods without this property may result in substantial phase and shape 
errors after long time integration.
Indeed for some wave equations the invariant preserving high order numerical methods 
have been shown more accurate than non-conservative methods after long-time numerical 
integration (see  \cite{Bona2013Conservative, Liu2016A}). 

For the nonlinear Schr\"odinger equation considered in this paper,  
a natural question is whether it is possible to design numerical schemes which conserve the mass  and 
energy  simultaneously.  A lot of effort has been made to preserve the mass by high order 
spatial discretization such as spectral methods \cite{Antoine2013Computational, bao2013numerical,bao2013Mathematical},   
and discontinuous Galerkin methods \cite{Lu2015Mass}.  
A modified numerical energy may be preserved approximately by the corresponding spatial discretization 
(see \cite{Liu2018On}). However, since those methods are based on a time-splitting technique \cite{bao2013Mathematical}, 
{except the Crank-Nickson scheme, they are only mass-conserved.} The objective of this work is to provide an attempt to address the above  question. Specifically, in this paper we develop and analyze a family of  
mass- and energy-conserved time-stepping schemes for approximating the cubic and general 
nonlinear Schr\"odinger equations. It should be noted that the energy conservation is also achieved for a 
modified energy, instead of the original energy. 

The rest of this paper is organized as follows. In Section \ref{sec-2} we present a general framework 
involving two sequences of time-stepping schemes, which is shown to preserve both mass and 
energy for arbitrary time-step sizes, for the cubic nonlinear Schr\"odinger equation. 
In Section \ref{sec-3} we present many specific examples of mass- and (modified) energy-conserved time 
stepping schemes which fit into the general framework, and derive the truncation errors for four of these schemes.
In Section \ref{sec-4} we present an efficient iterative algorithm to solve the resulting 
nonlinear equations. 
In Section \ref{sec-5} we extend the framework and examples to the general Schr\"odinger equations 
with arbitrary nonlinearity. 
In Section \ref{sec-6} we present a dispersion error analysis and derive the convergence rates 
for the dispersion errors.
In Section \ref{sec-7} we present numerical experiments to validate the theoretical results and 
to gauge the performance of the proposed schemes, especially the sharpness of the convergence rates. 
In Section 8 we present additional numerical experiments to demonstrate the capability of the proposed 
numerical schemes for resolving the blow-up
phenomenon. The paper is completed with some concluding remarks and comments 
given in Section \ref{sec-9}.
 
\section{Semi-discretization in time: a general framework.}\label{sec-2}
In this section we propose a family of energy-conserved time-stepping schemes 
for approximating the cubic nonlinear Schr\"odinger equation.

Let $\tau>0$ and $t=t_n=n\tau$ for $n=0,1,2,\cdots, N$ be a uniform mesh for $[0,T]$. 
Let $k$ be a positive integer. We propose the following general $k$-step time-stepping scheme 
for problem \eqref{nls_eqn1}--\eqref{nls_eqn3}: 
Seeking $\{R^n,u^n\}$ for $n=k,k+1,\cdots, N$ such that 
\begin{align}\label{scheme_1}
id_t R^{n+1} =&-\Delta R^{n+1/2} + \frac{\lambda}{2}\bigl(|R^n|^2+|R^{n+1}|^2\bigr)R^{n+1/2},\\
\label{scheme_2}
u^{n+1}=&
\begin{cases}
\beta_0^{-1} (R^{n+1}-\sum\limits_{j=1}^{k-1} \beta_{j} u^{n+1-j}) &\qquad k>1,  \\
\beta_0^{-1} R^{n+1} &\qquad k=1,
\end{cases}
\end{align}
where we use notation 
\begin{align}\label{dt}
d_t R^{n+1}=\frac{R^{n+1}-R^n}{\tau}, \qquad R^{n+1/2}=\frac{R^{n+1}+R^n}2.
\end{align}

Note that from \eqref{scheme_2} we see that 
$R^n$ is a linear combination of $u^n,u^{n-1},\cdots,u^{n-k+1}$ given as follows:
\begin{align}\label{rn}
R^n =\sum\limits_{j= 0}^{ k-1}  \beta_{j} u^{n-j}.
\end{align}

As expected, choosing the parameters $\{\beta_j\}$ is a delicate issue. The guideline we use (see the details in
the next section) is to choose them such that $d_t R^{n+1}$ has a specific order of accuracy for approximating  $u_t(t_{n+1})$.  In Table~\ref{table_beta} we list several sets of parameters $\{\beta_j\}$ to be used in  
scheme \eqref{scheme_1}--\eqref{scheme_2}, which results in various specific schemes for the nonlinear 
Schr\"odinger equation (again, see the details in the subsequent section). 

\begin{table}[ht]\small\centering
\caption{Parameters $\beta_j$ for some proposed schemes.}
\renewcommand\arraystretch{1.35}
\begin{tabular*}{\hsize}{@{}@{\extracolsep{\fill}}rlcccccc@{}}
\toprule
Schemes&$\beta_0$&$\beta_1$&$\beta_2$&$\beta_3$&$\beta_4$&$\beta_5$\\
\midrule
Crank-Nicolson&1&&&\\
Leapfrog&$\frac12$&$\frac12$&&&\\
M-BDF2&$\frac{3}{2}$&$-\frac{1}{2}$&                &                 &               &              \\
M-BDF3&$\frac{11}{6}$&$-\frac{7}{6}$&$\frac{1}{3}$&                 &               &              \\
M-BDF4&$\frac{25}{12}$& $-\frac{23}{12}$& $\frac{13}{12}$&   $-\frac{1}{4}$&               &              \\
M-BDF5&$\frac{137}{60}$&$-\frac{163}{60}$&$\frac{137}{60}$& $-\frac{21}{20}$&$\frac{1}{5}$  &              \\
M-BDF6&$\frac{147}{60}$&$-\frac{213}{60}$&$\frac{237}{60}$&$-\frac{163}{60}$&$\frac{31}{30}$&$-\frac{1}{6}$\\
4-Step Symmetric&$-\frac{1}{12}$&$\frac{7}{12}$&$\frac{7}{12}$&$-\frac{1}{12}$\\
\bottomrule
\end{tabular*}
\label{table_beta}
\end{table}

We also remark that scheme \eqref{scheme_1}--\eqref{scheme_2} produces two sequences, 
namely $\{R^n\}$ and $\{u^n\}$.
The first sequence can be regarded as the auxiliary quantities 
which are generated by solving nonlinear equation \eqref{scheme_1}, 
while the second sequence are obtained 
as linear combinations of the first one. 

To prove a key mass- and energy-conservation property of scheme \eqref{scheme_1}--\eqref{scheme_2},
we define the following discrete mass and energy
\begin{align}\label{dis_mh}
\cM^n:=\|R^n\|_{L^2}^2,\qquad 
\cH^n:= \frac 12 \|\nabla R^n\|_{L^2}^2 + \frac{\lambda}{4} \|R^n\|_{L^4}^4.
\end{align}

We start with establishing the following conservation results for the solution of scheme~\eqref{scheme_1}--\eqref{scheme_2}.

\begin{theorem} \label{Co:conservation} 
The solution to scheme \eqref{scheme_1} and \eqref{scheme_2} satisfies
$\cM^n = \cM^0$ and $\cH^n = \cH^0$ for all $n \ge 1$.
\end{theorem}

\begin{proof}
We multiply equation \eqref{scheme_1} by $\bar{R}^{n+1/2}$,
integrate it over $\cD$ and use the integration by parts to get, 
\begin{align}
\label{M_conv_1}
&\frac i{2\tau}\int_{\cD} \bigl(R^{n+1}-R^n\bigr)
\bigl(\bar{R}^{n+1}+\bar{R}^n\bigr)\, dx \\
&\qquad =\|\nabla R^{n+1/2}\|_{L^2}^2 
+ \frac \lambda 2 \int_{\cD}\bigl(|R^n|^2+|R^{n+1}|^2 \bigr)\bigl|R^{n+1/2}\bigr|^2 \, dx.
\end{align}
Taking the imaginary part of the resulting equation to get
\begin{align}\label{M_conv_2}
\frac{1}{2\tau}\Re \Bigl[\int_{\cD} \bigl(R^{n+1}-R^n\bigr)
\bigl(\bar{R}^{n+1}+\bar{R}^n\bigr)\, dx\Bigr]=0.
\end{align}
It follows from the identity $\Re \bigl[(a-b)(\bar{a}+\bar{b})\bigr] =|a|^2-|b|^2$ 
for $a,b \in \bC$ that
\begin{align}\label{M_conv_3}
\cM^{n+1} = \cM^n, \quad \forall\, n\ge 0.
\end{align}

To show the second conservation property, we multiply \eqref{scheme_1} by $d_t \bar{R}^{n+1}$,
integrate the equation over $\cD$ and use the integration by parts to get
\begin{align}\label{H_conv_1}
i\|d_tR^{n+1}\|_{L^2}^2
=&\frac 1{2\tau}\int_{\cD}\bigl(\nabla R^{n+1}-\nabla R^n\bigr)
\bigl(\nabla\bar{R}^{n+1}+\nabla \bar{R}^n\bigr)\, dx\\
&+ \frac \lambda{4\tau}\int_{\cD}\bigl(|R^n|^2+|R^{n+1}|^2\bigr)\bigl(R^{n+1}-R^n\bigr)
\bigl(\bar{R}^{n+1}+\bar{R}^n\bigr)\, dx. \notag
\end{align}
Taking the real part on the equation and applying identity 
$\Re\bigl[(|a|^2+|b|^2)(a-b)(\bar{a}+\bar{b})\bigr] =|a|^4-|b|^4$ for $a,b \in \bC$ yield 
\begin{align}\label{H_conv_2}
\cH^{n+1} = \cH^n, \quad \forall\, n\ge 0.
\end{align} 
The proof is completed.
\end{proof}

\section{Specific schemes and their truncation error analysis.}\label{sec-3}
In this section we propose a number of specific schemes by defining $R^n$ in terms of 
$u^n, u^{n-1},\cdots, u^{n-k+1}$. In other words, we shall specify the choice of 
parameters $\{\beta_j\}_{j=0}^{k-1}$ for each scheme. 

\subsection{\bf A modified Crank-Nicolson scheme ($k=1$).} 
By setting $R^n=u^n$, the modified Crank-Nicolson scheme is defined as
\begin{align}\label{C-N}
\frac i \tau (u^{n+1}-u^n) = -\frac 12 \Delta \bigl(u^{n+1}+u^n\bigr) 
+\frac \lambda 4 \Bigl(\bigl|u^n\bigr|^2 + \bigl|u^{n+1}\bigr|^2\Bigr)
\bigl(u^{n+1} + u^n\bigr).
\end{align}

The local truncation error (LTE) of the modified  Crank-Nicolson scheme 
\eqref{C-N} is defined by
\begin{align}\label{C-N:1}
TE^n = &i\frac 1\tau \bigl(u(t_{n+1})-u(t_{n})\bigr)
+ \frac 12\Delta\Bigl(u(t_{n+1}) + u(t_n)\Bigr)\\
&-\frac{\lambda}{4} \Bigl(\bigl|u(t_n)\bigr|^2 + \bigl|u(t_{n+1})\bigr|^2\Bigr)
\bigl(u(t_{n+1}) + u(t_n)\bigr)
\notag\\
=&:\rA_1+\rA_2+\rA_3,
\notag
\end{align}
where $u(t)$ is the true solution of the nonlinear Schr\"odinger equation as follows:
\begin{align}\label{pde}
i u_t(t) +\Delta u(t) -\lambda |u(t)|^2 u(t) = 0.
\end{align}
Here and below the dependence of $u$ on $x$ variable is suppressed for notational brevity.

The first lemma establishes the local truncation error for the modified Crank-Nicolson scheme.

\begin{lemma}
The local truncation error of the modified Crank-Nicolson scheme \eqref{C-N} ($k=1$) is $O(\tau^2)$.
\end{lemma}

\begin{proof}
By using a Taylor series expansion of $u(t_{n+1})$ and $u(t_n)$ about $u(t_{n+1/2})$,
the accurate order of derivative term $\rA_1$ and Laplace term $ \rA_2 $ are $O(\tau^2)$.
\begin{align}\label{C-N:2}
\rA_1
=&\frac{i}{\tau}u(t_{n+1/2}) + \frac{i}{2}u_t(t_{n+1/2})
+\frac{i\tau}{8}u_{tt}(t_{n+1/2}) + \frac{i\tau^2}{2^33!}u^{(3)}(t_{n+1/2})+ O(\tau^3)
\\
&-\frac{i}{\tau}u(t_{n+1/2}) + \frac{i}{2}u_t(t_{n+1/2})
-\frac{i\tau}{8}u_{tt}(t_{n+1/2}) + \frac{i\tau^2}{2^33!}u^{(3)}(t_{n+1/2})+ O(\tau^3)
\notag\\
= &\,iu_t(t_{n+1/2})+O(\tau^2).\notag\\
\label{C-N:3}
\rA_2
=&\frac 12 \Delta u(t_{n+1/2}) + \frac{\tau}{4}(\Delta u)_t (t_{n+1/2})
+\frac{\tau^2}{16}(\Delta u)_{tt} (t_{n+1/2}) + O(\tau^3)
\\
&+\frac 12 \Delta u(t_{n+1/2}) - \frac{\tau}{4}(\Delta u)_t (t_{n+1/2})
+\frac{\tau^2}{16}(\Delta u)_{tt}(t_{n+1/2})+ O(\tau^3) \notag\\
=&\Delta u(t_{n+1/2})+O(\tau^2).\notag
\end{align}
Note that
\begin{align*}
\bigl|u(t_n)\bigr|^2
=&\Bigl(u(t_{n+1/2})-\frac{\tau}{2}u_t(t_{n+1/2})
+ O(\tau^2)\Bigr)
\Bigl(\bar{u}(t_{n+1/2})-\frac{\tau}{2}\bar{u}_t(t_{n+1/2})
+ O(\tau^2)\Bigr)
\notag\\
=&\bigl|u(t_{n+1/2})\bigr|^2
-\frac{\tau}{2}\bar{u}(t_{n+1/2})u_t(t_{n+1/2})
-\frac{\tau}{2}\bar{u}_t(t_{n+1/2})u(t_{n+1/2})+O(\tau^2),
\notag\\
\bigl|u(t_{n+1})\bigr|^2
=&\Bigl(u(t_{n+1/2})+\frac{\tau}{2}u_t(t_{n+1/2})
+ O(\tau^2)\Bigr)
\Bigl(\bar{u}(t_{n+1/2})+\frac{\tau}{2}\bar{u}_t(t_{n+1/2})
+ O(\tau^2)\Bigr), 
\notag\\
=&\bigl|u(t_{n+1/2})\bigr|^2
+\frac{\tau}{2}\bar{u}(t_{n+1/2})u_t(t_{n+1/2})
+\frac{\tau}{2}\bar{u}_t(t_{n+1/2})u(t_{n+1/2})+O(\tau^2).
\notag
\end{align*}
which leads to 
\begin{align}\label{C-N:4}
\rA_3=&-\frac{\lambda}{4}\Bigl(2\bigl|u(t_{n+1/2})\bigr|^2
+O(\tau^2)\Bigr)
\bigl(2u(t_{n+1/2})+O(\tau^2)\bigr)\\
=&-\lambda |u(t_{n+1/2})|^2 u(t_{n+1/2}) +O(\tau^2).
\notag
\end{align}

Recall the definition of the true solution $u(t)$ 
and using \eqref{C-N:2}-\eqref{C-N:4} in \eqref{C-N:1} gives
\begin{align}\label{C-N:5}
TE^n=&\Bigl[iu_t(t_{n+1/2})
+\Delta u(t_{n+1/2})
-\lambda |u(t_{n+1/2})|^2 u(t_{n+1/2}) \Bigr]
+O(\tau^2)
=O(\tau^2).
\end{align}
Thus the local truncation error of above Crank-Nicolson scheme is $O(\tau^2)$. The proof is completed. 
\end{proof}

\subsection{\bf A Leapfrog scheme ($k=2$).} 
For $k=2$ we set $\beta_0=\beta_1=\frac 12 $ in (\ref{scheme_2}) so that 
$R^n=\frac12(u^n+u^{n-1})$,
which inserted into \eqref{scheme_1}) leads to the following  Leapfrog scheme:
\begin{align}\label{L-F}
&\,\frac{i}{2\tau}(u^{n+1}-u^{n-1})
=-\frac{1}{4} \bigl(\Delta u^{n+1}
+2\Delta u^n+\Delta u^{n-1} \bigr) \\
&\qquad+ \frac{\lambda}{8} \Bigl(\Bigl|\frac{u^n+u^{n-1}}{2}\Bigr|^2
+\Bigl|\frac{u^{n+1}+u^n}{2}\Bigr|^2\Bigr)
\bigl(u^{n+1} + 2u^n+ u^{n-1} \bigr).
\notag
\end{align}

The local truncation error of the leapfrog scheme  \eqref{L-F} is defined by 
\begin{align}\label{L-F:1}
TE^n = &\frac{i}{2\tau}\bigl(u(t_{n+1}) - u(t_{n-1})\bigr)
+ \frac{1}{4}\Delta \bigl( u(t_{n+1})
+ 2 u(t_n) + u(t_{n-1})\bigr)\notag \\
&-\frac{\lambda}{8} \Bigl[\Bigl|\frac{u(t_n) + u(t_{n-1})}{2}\Bigr|^2
+\Bigl|\frac{u(t_{n+1}) + u(t_n)}{2}\Bigr|^2\Bigr]
\bigl(u(t_{n+1}) + 2u(t_n)+u(t_{n-1})\bigr)\\
=:&\rB_1+\rB_2+\rB_3.
\notag
\end{align}

The next lemma establishes the local truncation error for the above Leapfrog scheme.
\begin{lemma}
The local truncation error of the Leapfrog scheme ($k=2$) \eqref{L-F} is $O(\tau^2)$.
\end{lemma}

\begin{proof}
We formally apply the Taylor series expansions of $u(t_{n+1})$ and $u(t_{n-1})$
about $u(t_n)$ for $\rB_1$ to get
\begin{align}\label{L-F:2}
\rB_1
=&\frac{i}{2\tau}u(t_n) + \frac{i}{2}u_t(t_n)
+\frac{i\tau}{4}u_{tt}(t_n) + \frac{i\tau^2}{12}u^{(3)}(t_n) + O(\tau^3)
\\
&-\frac{i}{2\tau}u(t_n) + \frac{i}{2}u_t(t_n)
-\frac{i\tau}{4}u_{tt}(t_n) + \frac{i\tau^2}{12}u^{(3)}(t_n)+ O(\tau^3) \notag\\
= &iu_t(t_n) + O(\tau^2).
\notag
\end{align}

Applying the Taylor series expansions of $u(t_{n+1})$ and $u(t_{n-1})$
about $u(t_n)$ again for $\rB_2$, we have
\begin{align}\label{L-F:3}
\rB_2
=&\frac{1}{4}\Delta u(t_n) + \frac{\tau}{4}(\Delta u)^{'}(t_n)
+\frac{\tau^2}{8}(\Delta u)^{''}(t_n)
+\frac{\tau^3}{24}(\Delta u)^{(3)}(t_n) + O(\tau^4)\\
&+ \frac 12 \Delta u(t_n)
+\frac{1}{4}\Delta u(t_n)
-\frac{\tau}{4}(\Delta u)_t(t_n)
+\frac{\tau^2}{8}(\Delta u)_{tt}(t_n)\notag\\
&-\frac{\tau^3}{24}(\Delta u)^{(3)}(t_n) + O(\tau^4)
\notag\\
=&\Delta u(t_n) + O(\tau^2).\notag
\end{align}

By using the following facts for $\rB_3$
\begin{align}\label{L-F:5}
\Bigl|\frac{u(t_n)+u(t_{n-1})}{2}\Bigr|^2
=&\frac{1}{4}\Bigl(u(t_n)+u(t_{n-1})\Bigr)\Bigl(\bar{u}(t_n)+\bar{u}(t_{n-1})\Bigr) \notag\\
=&\frac{1}{4}\Bigl(2u(t_n)-u_t(t_n)\tau
+ O(\tau^2)\Bigr)
\Bigl(2\bar{u}(t_n)-\bar{u}_t(t_n)\tau
+ O(\tau^2)\Bigr) 
\\
=&\bigl|u(t_n)\bigr|^2
-\frac 12 \tau\bar{u}(t_n)u_t(t_n)
-\frac 12 \tau\bar{u}_t(t_n)u(t_n) + O(\tau^2),
\notag\\
\label{L-F:6}
\Bigl|\frac{u(t_n)+u(t_{n+1})}{2}\Bigr|^2
=&\frac{1}{4}\Bigl(u(t_n)+u(t_{n+1})\Bigr)\Bigl(\bar{u}(t_n)+\bar{u}(t_{n+1})\Bigr) \notag\\
=&\bigl|u(t_n)\bigr|^2
+\frac 12 \tau\bar{u}(t_n)u_t(t_n)
+\frac 12 \tau\bar{u}_t(t_n)u(t_n) + O(\tau^2), 
\end{align}
the order of nonlinear term can be estimated as
\begin{align}\label{L-F:7}
\rB_3
=&-\frac{\lambda}{8}\Bigl(2\bigl|u(t_n)\bigr|^2
+O(\tau^2)\Bigr)
\bigl(4u(t_n) + O(\tau^2)\bigr)\\
\approx& -\lambda |u(t_n)|^2u(t_n) + O(\tau^2).
\notag
\end{align}

Combing \eqref{L-F:2}-\eqref{L-F:7} together in \eqref{L-F:1} leads to
\begin{align}\label{L-F:8}
TE^n = \Bigl[iu_t(t_n) +\Delta u(t_n) -\lambda |u(t_n)|^2 u(t_n)\Bigr] + O(\tau^2)
= O(\tau^2).
\end{align}
Thus the local truncation error of above Leapfrog scheme is $O(\tau^2)$. The proof is completed.
\end{proof}

\subsection{\bf Modified BDF schemes ($k=s$).} 
Let $s>1$. We recall that the s-step BDF scheme approximates the time derivative 
$u_t (t_{n+1})$ as follows: 
\begin{align}\label{BDF:1}
u_t(t_{n+1}) \approx& \frac{1}{\tau} \sum\limits_{j = 0}^s {\asj u^{n+1-j}}, 
\end{align}
where $\asj$ are given in Table~\ref{table_alpha_beta}.   
Our idea is to rewrite the above BDF expression as a first order backward difference, that is,
\begin{align}\label{BDF:2}
\frac{1}{\tau}\sum\limits_{j = 0}^s {\asj u^{n+1-j}} = \frac{1}{\tau}(R^{n+1}-R^{n}) 
\qquad \mbox{with} \quad
R^n=\sum\limits_{j = 0}^{s-1} \bsj u^{n-j},
\end{align}
which turns out is possible. Somehow this simple reformulation has not been seen 
in the literature before. 
  
In order to determine the coefficients $\bsj$, we solve them using the following identity:
\begin{align}\label{BDF:3}
\beta_{s,0}u^{n+1} + \sum\limits_{j = 1}^{s-1} (\bsj-\beta_{s,j-1})u^{n+1-j}
-\beta_{s,s-1}u^{n+1-s}
= \sum\limits_{j = 0}^s \asj u^{n+1-j}.
\end{align}
This holds true for all $u^j$ as long as the following matrix equation is satisfied:
\begin{align}\label{BDF:4}
\left[ 
\begin{array}{*{20}{c}}
1&\\
{-1}&1&\\
&\ddots&\ddots\\
&&{-1}&1&\\
&&&\ddots&\ddots&\\
&&&&{-1}&1\\
&&&&&{-1}
\end{array} 
\right]_{(s+1)\times s}
\left[ \begin{array}{*{20}{c}}
{\beta_{s,0}}\\
{\beta_{s,1}}\\
\vdots \\
{\bsj}\\
\vdots \\
{\beta_{s,s-1}}
\end{array} 
\right] 
= \left[ \begin{array}{*{20}{c}}
{\alpha_{s,0}}\\
{\alpha_{s,1}}\\
\vdots \\
{\asj}\\
\vdots \\
{\alpha_{s,s-1}}\\
{\alpha_{s,s}}
\end{array}\right].
\end{align}

It is easy to check that 
\begin{align}\label{BDF:5}
\bsj= \sum\limits_{\ell = 0}^j \alpha_{s,\ell}, \qquad j=0,1,2,\cdots, s-1.
\end{align}
Thus our modified BDF schemes are defined as
\begin{align}\label{BDF}
\frac{i}{\tau}\Bigl(\sum\limits_{j = 0}^{s-1} &\bsj u^{n+1-j}
-\sum\limits_{j = 0}^{s-1} \bsj u^{n-j}\Bigr)\\
=&-\Delta R^{n+1/2} 
+ \frac{\lambda}{2} \Bigl(\Bigl|\sum\limits_{j = 0}^{s-1} \bsj u^{n+1-j}\Bigr|^2
+\Bigl|\sum\limits_{j = 0}^{s-1} \bsj u^{n-j}\Bigr|^2 \Bigr)R^{n+1/2},
\notag
\end{align}
where 
\begin{align*}
R^{n+1/2}=
\frac 12 \beta_{s,0}  u^{n+1}
+\frac 12 \Bigl(\sum\limits_{j = 1}^{s-1}\bsj
+\sum\limits_{j = 1}^{s}\beta_{s,j-1} \Bigr) u^{n+1-j}, 
\end{align*}
and $\bsj$ are given in the following table (note that since BDF methods with $s>6$ 
are not zero-stable, so we only present s-step BDF with $s \le 6$ here).  
\begin{table}[htb]\small\centering\footnotesize
\caption{Parameters $\asj$ and $\bsj$.}
\renewcommand\arraystretch{2}
\resizebox{\textwidth}{!}{
\begin{tabular}{cccccccccccccc}\toprule
s& $\alpha_{s,0}$& $\alpha_{s,1}$& $\alpha_{s,2}$& $\alpha_{s,3}$& $\alpha_{s,4}$& $\alpha_{s,5}$& $\alpha_{s,6}$&
$\beta_{s,0}$& $\beta_{s,1}$& $\beta_{s,2}$& $\beta_{s,3}$& $\beta_{s,4}$& $\beta_{s,5}$ \\ 
\midrule
2&$\frac32$&$-2$& $\frac12$& & & & &  $\frac32$ &  $-\frac12$ & & & &   \\
3&$\frac{11}{6}$&$-3$& $\frac32$&$-\frac13$& & & &$\frac{11}6$&$-\frac76$&$\frac13$& & & \\
4&$\frac{25}{12}$&$-4$&3&$-\frac43$& $\frac14$& & &$\frac{25}{12}$& $-\frac{23}{12}$&$\frac{13}{12}$&$-\frac14$& &\\
5&$\frac{137}{60}$&$-5$&5&$-\frac{10}3$& $\frac54$&$-\frac15$&            
&$\frac{137}{60}$&$-\frac{163}{60}$&$\frac{137}{60}$& $-\frac{21}{20}$&$\frac15$ &\\
6&$\frac{147}{60}$&$-6$&$\frac{15}2$&$-\frac{20}{3}$&$\frac{15}{4}$&$-\frac65$&$\frac16$&
$\frac{147}{60}$&$-\frac{213}{60}$&$\frac{237}{60}$&$-\frac{163}{60}$&$\frac{31}{30}$&$-\frac16$\\
\bottomrule
\end{tabular}}
\label{table_alpha_beta}
\end{table}

Set $\Rn:=\sum\limits_{j = 0}^{s-1} \bsj u(t_{n-j})$ as a linear combination of exact solution values, 
{then the local truncation error of the modified BDF schemes \eqref{BDF} is defined by}
\begin{align}\label{BDFs:1}
TE^n=&\frac{i}{\tau}\bigl(\Rnn-\Rn)\bigr)
+\frac 12  \bigl(\Delta \Rnn+ \Delta \Rn\bigr)\\
&-\frac{\lambda}{4} \Bigl(\Bigl|\Rnn\Bigr|^2
+\Bigl|\Rn\Bigr|^2\Bigr)
\bigl(\Rnn+ \Rn\bigr)
\notag\\
=& \rT_1+\rT_2+\rT_3.
\notag
\end{align}

\begin{lemma}
The local truncation error of the modified BDF schemes \eqref{BDF} ($k=s$) is $O(\tau^2)$.
\end{lemma}

\begin{proof}
Recall the approximation of the derivative yields $\rT_1=iu_t(t_{n+1})+O(\tau^s)$.
Using the Taylor series expansion of $u(t_{n+1-j})$ about $u(t_n)$ we have  
\begin{align}\label{BDFs:2}
u(t_{n+1-j})
=&u(t_{n+1}) -j\tau u_t(t_{n+1}) + \frac{(j\tau)^2}{2} u_{tt}(t_{n+1}) + O(\tau^3),
\end{align}
which implies 
\begin{align}\label{BDFs:3}
\rT_2
=&\frac 12 \beta_{s,0} \Delta u(t_{n+1})
+\frac 12 \sum\limits_{j = 1}^{s-1} (\bsj+\beta_{s,j-1})
\Delta u(t_{n+1-j})
+\frac 12 \beta_{s,s-1} \Delta u(t_{n+1-s})
\\
=&\Bigl(\sum\limits_{j = 0}^{s-1}\bsj \Bigr)
\Delta u(t_{n+1})
-
\frac{\tau}{2}\Bigl(\sum\limits_{j = 1}^{s-1}j\bsj +\sum\limits_{j = 1}^{s} j\beta_{s,j-1} \Bigr)
(\Delta u)_{t}(t_{n+1})
\notag\\
&+\frac{\tau^2}{4}\Bigl(\sum\limits_{j = 1}^{s-1}j^2\bsj 
+ \sum\limits_{j = 1}^{s} j^2\beta_{s,j-1} \Bigr)
(\Delta u)_{tt}(t_{n+1})
+ O(\tau^3)
\notag \\
=&\Delta u(t_{n+1}) + O(\tau^2), \notag
\end{align}
where we have used the facts that
\begin{align*}
\sum\limits_{j = 0}^{s-1}\bsj = 1
\quad\mbox{and}\quad
\sum\limits_{j = 1}^{s-1}j\bsj+\sum\limits_{j = 1}^{s}j\beta_{s,j-1}
=\sum\limits_{j = 0}^{s-1}(2j+1)\bsj = 0.
\end{align*}

For the nonlinear term, we have the following estimates:
\begin{align}\label{BDFs:4}
\bigl|\Rn\bigr|^2
=&\Bigl| \sum\limits_{j = 1}^s \beta_{s,{j-1}} u(t_{n+1-j})\Bigr|^2\\
=&\Bigl|\sum\limits_{j = 1}^s \beta_{s,{j-1}} 
u(t_{n+1})
-\tau \sum\limits_{j = 1}^s j\beta_{s,{j-1}} 
u_t(t_{n+1}) 
+O(\tau^2)\Bigr|^2
\notag\\
=&\bigl|u(t_{n+1})\bigr|^2
-\tau \sum\limits_{j = 1}^s j\beta_{s,{j-1}} 
u_t(t_{n+1}) \bar{u}(t_{n+1}) \notag\\
&-\tau\sum\limits_{j = 1}^s j\beta_{s,{j-1}} 
u(t_{n+1}) \bar{u}_t(t_{n+1})
+O(\tau^2),
\notag
\end{align}
\begin{align}\label{BDFs:5}
\bigl|\Rnn\bigr|^2
=&\Bigl|\sum\limits_{j = 0}^{s-1} \bsj 
u(t_{n+1})
-\tau \sum\limits_{j = 1}^{s-1} j\bsj 
u_t(t_{n+1}) 
+O(\tau^2)\Bigr|^2
\\
=&\bigl|u(t_{n+1})\bigr|^2
- \tau \sum\limits_{j = 1}^{s-1} j\bsj 
u_t(t_{n+1}) \bar{u}(t_{n+1})\notag\\
&-\tau \sum\limits_{j = 1}^{s-1} j\bsj 
u(t_{n+1}) \bar{u}_t(t_{n+1})
+O(\tau^2),\notag
\end{align}
where we use $\sum\limits_{j = 0}^{s-1}\bsj=\sum\limits_{j = 1}^s\beta_{s,{j-1}}=1$, 
$\sum\limits_{j = 1}^{s-1} j\bsj \neq 0$  
and $\sum\limits_{j = 1}^s j\beta_{s,{j-1}} \neq 0$.
Since $\sum\limits_{j = 1}^{s-1} j\bsj + \sum\limits_{j = 1}^s j\beta_{s,{j-1}} = 0$, hence,
\begin{align}\label{BDFs:6}
\Bigl|\Rnn\Bigr|^2
+\Bigl|\Rn\Bigr|^2
=2\bigl|u(t_{n+1})\bigr|^2+O(\tau^2).
\end{align}

Similar to the estimates in \eqref{BDFs:3}, we have
$\Rnn+ \Rn = 2u(t_{n+1}) + O(\tau^2)$ and
\begin{align}\label{BDFs:7}
\rT_3
=&-\frac{\lambda}{4} \Bigl(\Bigl|\Rnn\Bigr|^2
+\Bigl|\Rn\Bigr|^2\Bigr)
\bigl(\Rnn+ \Rn\bigr)\\
=&-\frac{\lambda}{4}\Bigl(2\bigl|u(t_{n+1})\bigr|^2
+O(\tau^2) \Bigr) \bigl(2u(t_{n+1}) +O(\tau^2)\bigr) \notag\\
=&-\lambda |u(t_{n+1})|^2 u(t_{n+1}) +O(\tau^2). \notag
\end{align}

Combining \eqref{BDFs:2} and \eqref{BDFs:7} in \eqref{BDFs:1}, we obtain
\begin{align}\label{BDF:7}
TE^n=&\Bigl[iu_t(t_{n+1}) + \Delta u(t_{n+1}) - \lambda |u(t_{n+1})|^2 u(t_{n+1}) \Bigr] 
+O(\tau^2)
=O(\tau^2).
\end{align}
Thus the local truncation error is $O(\tau^2)$.
\end{proof}

\begin{remark}
As a by-product, the above construction also gives a (one-step) backward Euler reformulation 
for BDF schemes. Recall that the $s$-stage BDF scheme for  $u'(t)=f(t, u(t)) $ is defined as
\[
\sum_{j=0}^s \alpha_{s,j} u^{n+1-j} = f\bigl( t_{n+1}, u^{n+1} \bigr).
\]
Since
\[ 
\sum_{j=0}^s \alpha_{s,j} u^{n+1-j} = \frac{1}{\tau} \bigl( R^{n+1}- R^n \bigr)  \quad
\mbox{and} \quad 
R^{n+1} = \sum_{j=0}^{s-1}  \beta_j u^{n+1-j},
\] 
then we can rewrite the BDF scheme as
\begin{align}\label{euler_ref}
&\frac{1}{\tau} \bigl( R^{n+1}- R^n \bigr) = \widehat{f}\bigl( t_{n+1}, R^{n+1}\bigr),  
\end{align}
where
\begin{align*} 
\widehat{f}\bigl( t_{n+1}, R^{n+1}\bigr)
&= f\Bigl( t_{n+1},   \beta_0^{-1}  \bigl( R^{n+1}- \sum_{j=1}^{s-1} \beta_{s,j} u^{n+1-j}  \bigr)  \Bigr), \\
u^{n+1} &= \beta_0^{-1}  \bigl( R^{n+1}- \sum_{j=1}^{s-1} \beta_{s,j} u^{n+1-j} \bigr). 
\end{align*} 
Hence, each BDF scheme can be implemented as a one-step backward Euler scheme as \eqref{euler_ref} shows. 
\end{remark}

\subsection{\bf A four-step symmetric scheme ($k=4$).} 
To define this scheme, we set 
\begin{align}
R^{n+1}=\frac{1}{12}\bigl(-u^{n+1}+7u^n+7u^{n-1}-u^{n-2}\bigr), 
\end{align}
which fits (\ref{rn}) with $\beta=\frac{1}{12}(-1, 7, 7, -1)^\top$.   

\begin{lemma}
The local truncation error of the four-step symmetric scheme ($k=4$) is $O(\tau^2)$.
\end{lemma}

\begin{proof}
The derivative term is $O(\tau^4)$ as follows
\begin{align}\label{4SS:1}
\frac{1}{\tau}\bigl(R^{n+1}-R^n\bigr)
=&\frac{1}{12\tau}(-u^{n+1}+8u^n-8u^{n-2}+u^{n-3})\\
\approx & \,u_t(t_{n-1})+O(\tau^4). \notag
\end{align}

For the Laplace term, we obtain
\begin{align}\label{4SS:2}
\frac 12 \Delta \bigl(R^{n+1}+R^n\bigr)
=&\frac{1}{24}\Delta\bigl(- u^{n+1}+6u^n+14u^{n-1}+6u^{n-2}-u^{n-3}\bigr)\\
\approx & \Delta u(t_{n-1}) + O(\tau^2). \notag
\end{align}

For the nonlinear term, we have the following estimates:
\begin{align}\label{4SS:3}
-\frac{\lambda}{2} \bigl(|R^n|^2&+|R^{n+1}|^2 \bigr)R^{n+1/2}\\
=&-\frac{\lambda}{4}\Bigl(2\bigl|u(t_{n-1})\bigr|^2
+O(\tau^2)\Bigr)
\Bigl(2u(t_{n-1})+O(\tau^2)\Bigr)\notag\\
\approx&-\lambda |u(t_{n-1})|^2 u(t_{n-1}) + O(\tau^2).\notag
\end{align}
The proof is complete.
\end{proof}

\section{An efficient fixed-point nonlinear solver.} \label{sec-4}
To solve the nonlinear equation \eqref{scheme_1}, we adapt the fixed-point iterative algorithm of
\cite{Lu2015Mass} to the job. The proposed algorithm is defined below.

\medskip
\noindent
{\bf Algorithm 1}
\medskip

{\em Step 1}: Given $u^l$ for $l=0,1,...,n-k+1$, set
\begin{align}
R^n=\sum\limits_{j= 0}^{k-1} \beta_{j} u^{n-j}.
\end{align}

{\em Step 2}:
Update $R^{n+1}$ as follows:  define $\{ w^l\}_{l=0}^L$ iteratively by solving 
\begin{align}
\Bigl(iI+ \frac{\tau}2 \Delta \Bigr) w^{l+1} 
-\frac{\lambda\tau}{4} \bigl(|R^n|^2+|2w^l-R^n|^2 \bigr)w^{l+1}
=i R^n, \quad l=0, 1, \cdots, L, 
\end{align}
such that $\|w^L-w^{L-1}\|\leq \delta$ for some prescribed tolerance level $\delta$,  then set
\begin{align*}
R^{n+1}=2w-R^n.
\end{align*}

{\em Step 3}:
Update $u^{n+1}$ from $R^{n+1}$ by \eqref{scheme_2}, that is, 
\begin{align}
u^{n+1} = \begin{cases}
\beta_0^{-1} \Bigl(R^{n+1}-\sum\limits_{j=1}^{k-1} \beta_{j} u^{n+1-j} \Bigr) &\qquad k>1,  \\
\beta_0^{-1} R^{n+1} &\qquad k=1.
\end{cases}
\end{align}

We note that any spatial discretization method, such as finite element, spectral and discontinuous Galerkin methods,
can be employed in combination with the above algorithm to solve the nonlinear Schr\"odinegr equation (cf.  \cite{Lu2015Mass}).

\section{Extensions to Schr\"odinger equations with arbitrary nonlinearity.} \label{sec-5}
We consider the following initial-boundary value problem for the general nonlinear 
Schr\"odinger equation :   
\begin{alignat}{2}
\label{nls_eqn1a}
i u_t &= -\Delta u +\lambda  f(|u|^2) u &&\qquad \mbox{in } {\cD _T}: = \cD  \times (0,T),\\
\label{nls_eqn3a}
u (0) &= u _0 &&\qquad\mbox{in }\cD, 
\end{alignat}

We extend scheme \eqref{scheme_1}--\eqref{scheme_2} as follows for problem 
\eqref{nls_eqn1a}--\eqref{nls_eqn3a}:
Seeking $\{R^n,u^n\}$ for $n=k,k+1,\cdots, N$ such that 
\begin{align}\label{scheme_1a}
id_t R^{n+1} =&-\Delta R^{n+1/2} 
+ {\lambda G\bigl(|R^{n+1}|^2, |R^n|^2 \bigr) R^{n+1/2}},\\
\label{scheme_2a}
u^{n+1}=&\begin{cases}
\beta_0^{-1} (R^{n+1}-\sum\limits_{j=1}^{k-1} \beta_j u^{n+1-j}) &\qquad k>1,  \\
\beta_0^{-1} R^{n+1} &\qquad k=1,
\end{cases}
\end{align}
where $G(a,b)$ is the following two variable function: 
\begin{align}\label{Fp}
G(a,b)=\frac{F(a)-F(b)}{a-b},\qquad F(s):=\int_0^s f(\nu)\, d\nu. 
\end{align}

For example, $G(a,b)= \sum_{j=0}^p a^{p-j} b^j$ if $f(s)=s^{p+1}$.  
Again, from \eqref{scheme_2a} we have  
\begin{align*}
R^n=\sum\limits_{j= 0}^{ k-1}  \beta_{j} u^{n-j}.
\end{align*}
Define 
\begin{align}
\cM^n_g:=\|R^n\|_{L^2}^2,\qquad 
\cH^n_g:= \|\nabla R^n\|_{L^2}^2 + \lambda \| F(|R^n|^2)\|_{L^1}.
\end{align}
we have the following mass- and energy-conservation property of 
scheme \eqref{scheme_1a}--\eqref{scheme_2a}.
\begin{lemma} \label{conservation_law}
The solution to scheme \eqref{scheme_1a} and \eqref{scheme_2a} satisfies
$\cM^n_g = \cM^0_g$ and $\cH^n_g = \cH^0_g$ for all $n \ge 1$.
\end{lemma}

The nonlinear solver, Algorithm 1, now is replaced by the following modified algorithm.

\medskip
\noindent
{\bf Algorithm 2}

{\em Step 1}: Given $u^l$ for $l=0,1,...,n-k+1$, set
\begin{align}
R^n=\sum\limits_{j= 0}^{k-1} \beta_j u^{n-j}.
\end{align}

{\em Step 2}:
Update $R^{n+1}$ as follows:  
define $\{ w^l\}_{l=0}^L$ iteratively by solving 
\begin{align}
\Bigl(i I+ \frac{\tau}2 \Delta \Bigr) w^{l+1} 
-\frac{\lambda\tau}{2} G\bigl(|2w^l-R^n|^2, |R^n|^2\bigr)w^{l+1}
=i R^n, \quad l=0, 1, \cdots, L, 
\end{align}
such that $\|w^L-w^{L-1}\|\leq \delta$ for some prescribed tolerance level $\delta$, then set
\begin{align*}
R^{n+1}=2w-R^n.
\end{align*}

{\em Step 3}:
Update $u^{n+1}$ from $R^{n+1}$ by \eqref{scheme_2a}, that is, 
\begin{align}
u^{n+1}
=\begin{cases}
\beta_0^{-1} \Bigl(R^{n+1}-\sum\limits_{j=1}^{k-1} \beta_j u^{n+1-j} \Bigr) &\qquad k>1,  \\
\beta_0^{-1} R^{n+1} &\qquad k=1.
\end{cases}
\end{align}

Again, we remark that any spatial discretization method, such as finite element, spectral and discontinuous Galerkin methods,
can be employed in combination with the above algorithm to solve the nonlinear Schr\"odinegr equation.

\section{\bf Dispersion error analysis.} \label{sec-6}
In this section, we analyze the difference between the exact and
numerical dispersion relations for the nonlinear Schr\"odinger equation 
and investigate ways to reduce the dispersive error generated by 
our mass- and energy-conserved time-stepping scheme \eqref{scheme_1}--\eqref{scheme_2}. 
To minimize the numerical phase error while accurately solving the Schr\"odinger equation, 
the idea of preserving dispersion relation equation (DRE), which was proposed earlier   
in the area of computational aeroacoustics by Tam and Webb \cite{Christopher1993Dispersion}, 
is adopted. We refer the reader to \cite{Mclachlan2003Featured} for a discussion of other 
structure-preserving algorithms for solving ordinary differential equations. 

Consider the cubic nonlinear Schr\"odinger equation,
\begin{align}\label{cnls}
i u_t +\Delta u = \lambda|u|^2 u.
\end{align} 
Substituting the plane wave solution $u = \exp(i(kx-\omega t)) $ into  
equation \eqref{cnls}, the relation between the angular frequency
$\omega$ and the wave number $k$ is given by \cite{Sheu2015Dispersion}
\begin{align}\label{omega_1}
\omega = k^2 + \lambda.
\end{align} 

To derive the numerical dispersion relation equation for scheme 
\eqref{scheme_1}--\eqref{scheme_2}, the discrete plane wave solution
of the form $u^n = e^{i(kx-\tomega n \Dt)}$ is utilized,
where $\tomega$ is the numerical angular frequency.

\begin{lemma}
The numerical dispersion relation of the Crank-Nickson scheme 
for the cubic nonlinear Schr\"odinger equation \eqref{cnls} is given by 
\begin{align}\label{cn_0}
\tomega = \frac 2\Dt \arctan \bigl((k^2 + \lambda)\Dt/2\bigr).
\end{align}
\end{lemma} 
\begin{proof}
Substituting $u^n = e^{i(kx-\tomega n \Dt)}$ 
into the discrete Crank-Nicolson scheme \eqref{C-N} we get
\begin{align}\label{cn_1}
\frac i\Dt \bigl(e^{-i\tomega \Dt}- 1\bigr)
= \frac 12(k^2 + \lambda)\bigl(e^{-i\tomega \Dt}+ 1\bigr).
\end{align}

Multiplying the both sides of \eqref{cn_1} by $e^{i\tomega \Dt/2 }$ to obtain
\begin{align}\label{cn_2}
\frac i\Dt \bigl(e^{-i\tomega\Dt/2 }- e^{i\tomega \Dt/2 }\bigr)
= \frac 12 (k^2 + \lambda)\bigl(e^{-i\tomega \Dt/2}
+ e^{i\tomega \Dt/2 }\bigr).
\end{align}

By using the identities $e^{ix}-e^{-ix} = 2i\sin x$ and $e^{ix}+e^{-ix} = 2\cos x$,  
it follows from \eqref{cn_2} that 
\begin{align*} 
\frac 2\Dt \sin(\tomega \Dt/2) = (k^2 + \lambda)\cos(\tomega \Dt/2).
\end{align*}
Hence, \eqref{cn_0} holds. The proof is complete.
\end{proof}

To analyze the difference between the exact and
numerical dispersions for nonlinear Schr\"odinger equations,
we define the following dispersion error 
\begin{align}\label{disp_err}
\omega_{error} := \frac{|\omega-\tomega|}\omega.
\end{align}

Table~\ref{table_cn_disp_err} shows the computed dispersion errors and the convergence order 
for the modified Crank-Nicolson scheme. The numerical results indicate that 
this scheme has a second order dispersion error.

\begin{table}[htb]\small\centering\footnotesize
\caption{Dispersion error rates of Crank-Nickson scheme.}
\begin{tabular*}{\hsize}{@{}@{\extracolsep{\fill}}ccccc@{}}\toprule
$\lambda$         &               $k$& $ \Dt$& Dispersion errors& Error rates\\ 
\midrule
\multirow{4}{*}{2}&\multirow{4}{*}{1}&  1E-01&          0.003324& 	       --\\
				  &             	 &	1E-02& 		3.323244E-05&      1.9989\\
				  &     			 &	1E-03& 		3.333332E-07&      2.0000\\
				  &     			 &	1E-04& 		3.333333E-09&      2.0000\\
\bottomrule
\end{tabular*}
\label{table_cn_disp_err}
\end{table}

\begin{lemma}
The numerical dispersion relation of the Leapfrog scheme 
for the cubic nonlinear Schr\"odinger equation \eqref{cnls} is given by 
\begin{align}\label{lf_0}
\frac{2}{\Dt}\sin(\tomega \Dt/2)
= \bigl(k^2 + \lambda\cos^2(\tomega \Dt/2)\bigr)\cos(\tomega \Dt/2).
\end{align}
\end{lemma} 

\begin{proof}
Setting $u^n = e^{i(kx - \tomega n \Dt)}$ and using the identities $e^{ix}+e^{-ix} = 2\cos x$ and  
$\cos 2\theta = 2\cos^2\theta -1$ in $\bigl|\frac 12 \bigl(u^n+u^{n-1}\bigr)\bigr|^2$ 
and $\bigl|\frac 12 \bigl(u^n+u^{n+1}\bigr)\bigr|^2$ we get
\begin{align}\label{lf_1}
\Bigl|\frac 12 \bigl(u^n&+u^{n-1}\bigr)\Bigr|^2 + \Bigl|\frac 12 \bigl(u^n+u^{n+1}\bigr)\Bigr|^2\\
=&\Bigl|\frac 12 e^{ikx-i\tomega(n-\frac 12 ) \Dt}
\bigl(e^{-i\tomega \Dt/2} 
+ e^{i\tomega \Dt/2}\bigr)\Bigr|^2 \notag\\
&+\Bigl|\frac 12 e^{ikx-i\tomega(n+\frac 12) \Dt}
\bigl(e^{i\tomega \Dt/2}+ e^{-i\tomega \Dt/2} \bigr)\Bigr|^2 \notag\\
=&\bigl|e^{ikx-i\tomega(n-\frac 12) \Dt}\cos(\tomega \Dt/2)\bigr|^2
+\bigl|e^{ikx-i\tomega(n-\frac 12) \Dt}\cos(\tomega \Dt/2)\bigr|^2 \notag\\
=&2\cos^2(\tomega \Dt/2). \notag
\end{align}

Substituting $u^n = e^{i(kx - \tomega n \Dt)}$ into the Leapfrog scheme \eqref{L-F} 
and using \eqref{lf_1}, we get the following equality:
\begin{align}\label{lf_2}
\frac{i}{ \Dt} \bigl(e^{-i\tomega \Dt}- e^{i\tomega \Dt } \bigr)
=\bigl(k^2 + \lambda\cos^2(\tomega \Dt/2)\bigr)
\Bigl(\frac{e^{-i\tomega \Dt}+1}2 + \frac{e^{i\tomega \Dt}-1}2\Bigr).
\end{align}

Using the identities $ e^{ix}-e^{-ix} = 2i\sin x$, $e^{ix}+e^{-ix} = 2\cos x$, 
$\sin 2\theta = 2\sin \theta \cos \theta$, and $\cos 2\theta = 2\cos^2\theta -1$ in \eqref{lf_2}, 
we then obtain the following numerical dispersion relation equation of the Leapfrog scheme
for the cubic nonlinear Schr\"odinger equation:
\begin{align*}
\frac 2\Dt \sin(\tomega \Dt/2)
= \bigl(k^2 + \lambda\cos^2(\tomega \Dt/2)\bigr) \cos(\tomega \Dt/2).
\end{align*}
Hence, \eqref{lf_0} holds. The proof is complete.
\end{proof}

Table~\ref{table_lf_disp_err} shows the computed dispersion errors and convergence rates 
for the Leapfrog scheme. The numerical results indicate that the Leapfrog scheme also has a second 
order dispersion error.

\begin{table}[htb]\small\centering\footnotesize
\caption{Dispersion error rates of the Leapfrog scheme.}
\begin{tabular*}{\hsize}{@{}@{\extracolsep{\fill}}ccccc@{}}\toprule
$\lambda$         &               $k$& $ \Dt$& Dispersion errors& Error rates\\ 
\midrule
\multirow{4}{*}{2}&\multirow{4}{*}{1}&  1E-01&          0.021354& 	       --\\
				  &             	 &  1E-02& 		2.253333E-04&      1.9767\\
			      &     			 &	1E-03& 		2.266667E-06&      1.9974\\
				  &     			 &	1E-04& 		2.250000E-08&      2.0032\\
\bottomrule
\end{tabular*}
\label{table_lf_disp_err}
\end{table}

\begin{lemma}
The numerical dispersion relations of the modified BDF schemes  
for the cubic nonlinear Schr\"odinger equation \eqref{cnls} are given by 
\begin{align}\label{bdf_0}
\frac 2\Dt \sin(\tomega \Dt/2)
= \bigl(k^2 + \lambda H(\tomega,\Dt)\bigr)\cos(\tomega \Dt/2).
\end{align}
where $H(\tomega,\Dt)$ is given in Table~\ref{table_bdfs_H} below.
\begin{table}[htb]\small\centering\footnotesize
\caption{$H(\tomega,\Dt)$ of the numerical dispersion relation equations of the modified BDF schemes.}
\renewcommand\arraystretch{2}
\resizebox{\textwidth}{!}{
\begin{tabular*}{\hsize}{@{}@{\extracolsep{\fill}}rl@{}}\toprule
M--BDFs&$H(\tomega, \Dt)$ \\ 
\midrule
M--BDF2&$\frac{-1}2 \bigl[3\cos(\tomega\Dt)-5\bigr]$\\
M--BDF3&$\frac{-1}{36}\bigl[47\cos(2\tomega\Dt)-83\bigr]$\\
M--BDF4&$\frac{-1}{72}\bigl[913\cos(\tomega\Dt) - 394\cos(2\tomega\Dt) + 75\cos(3\tomega\Dt) - 666\bigr] $\\
M--BDF5&$\frac{-1}{1800}\bigl[54049 \cos(\tomega\Dt) - 30682\cos(2\tomega\Dt) + 10587\cos(3\tomega\Dt) - 1644\cos(4\tomega\Dt) + 34110\bigr]$ \\
M--BDF6&$\frac{-1}{1800}\bigl[131149\cos(\tomega\Dt) - 85882\cos(2\tomega\Dt) + 39537\cos(3\tomega\Dt) - 11244\cos(4\tomega\Dt) $ \\
	   &\qquad \quad$ + 1470\cos(5\tomega\Dt) - 76830\bigr]$\\
\bottomrule
\end{tabular*}}
\label{table_bdfs_H}
\end{table}
\end{lemma} 

\begin{proof}
We only present the proofs of the numerical dispersion relation equations 
for the modified BDF2 and  BDF3 schemes because the proofs for the remaining ones are similar.

We first consider the modified BDF2 scheme. Setting $u^n = e^{i(kx - \tomega n \Dt)}$ 
and using the identities 
$e^{ix}+e^{-ix} = 2\cos x$ and $\cos2\theta = 2\cos^2\theta -1$ 
in $\bigl|\frac 12 \bigl(3u^n-u^{n-1}\bigr)\bigr|^2$, we get 
\begin{align}\label{bdf2_1}
\Bigl| \frac{3u^n-u^{n-1}}2 \Bigr|^2
=& \Bigl|\frac 12 e^{ikx-i\tomega(n-\frac 12 ) \Dt}
\bigl(3e^{-i\tomega \Dt/2} -e^{i\tomega \Dt/2} \bigr)\Bigr|^2\\
=&\Bigl|e^{ikx-i\tomega(n-\frac 12) \Dt}
\bigl(\cos(\tomega \Dt/2) -2i\sin(\tomega \Dt/2)\bigr)\Bigr|^2 \notag\\
=&\cos^2(\tomega \Dt/2) + 4\sin^2(\tomega \Dt/2) \notag\\
=&\frac 12 \bigl(5 - 3\cos(\tomega\Dt)\bigr).\notag
\end{align}
A similar calculation applies to $\bigl|\frac 12 \bigl(3u^{n+1}-u^n \bigr)\bigr|^2$. Substituting 
$u^n = e^{i(kx - \tomega n \Dt)}$ into the modified BDF2 scheme \eqref{BDF}
and using \eqref{bdf2_1} yield 
\begin{align}\label{bdf2}
\frac i\Dt &\Bigl(\frac{3u^{n+1}-4u^n+u^{n-1}}2 \Bigr)
+\frac{3\Delta u^{n+1}+ 2\Delta u^n-\Delta u^{n-1}}4 \\
&=\frac \lambda 2 \Bigl(\Bigl|\frac{3u^{n+1}-u^n}2\Bigr|^2 + \Bigl|\dfrac{3u^n-u^{n-1}}2\Bigr|^2\Bigr)
\Bigl(\frac{3\Delta u^{n+1}+ 2\Delta u^n-\Delta u^{n-1}}4\Bigr), \notag \\
&=\frac{i}{ \Dt}\Bigl[3(e^{-i\tomega \Dt}-1) + (e^{i\tomega \Dt}-1)\Bigr] \notag \\
&=\frac 12 \bigl(k^2 + \lambda H_1(\omega, \Dt)\bigr)
\Bigl[3 (e^{-i\tomega \Dt}+1) - (e^{i\tomega \Dt}+1)\Bigr],\notag
\end{align}
where $H_1(\omega,\Dt) =\frac 12 \bigl(5-3 \cos(\tomega \Dt)\bigr)$.

It is easy to obtain the following equation from \eqref{bdf2_1}:  
\begin{align}\label{bdf2_3}
\frac{2i}{\Dt}\Bigl[3 (&e^{-i\tomega \Dt/2} - e^{i\tomega \Dt/2})e^{-i\tomega \Dt/2}
+(e^{i\tomega \Dt/2} - e^{-i\tomega \Dt/2})e^{i\tomega \Dt/2}\Bigr]\\
=&3(k^2 + \lambda H_1(\omega, \Dt))(e^{-i\tomega \Dt/2}
+ e^{i\tomega \Dt/2})e^{-i\tomega \Dt/2}
\notag \\
&-(k^2 + \lambda H_1(\omega, \Dt))(e^{i\tomega \Dt/2}
+e^{-i\tomega \Dt/2})e^{i\tomega \Dt/2}.
\notag
\end{align}
Using the identities $ e^{ix}-e^{-ix} = 2i\sin x$ and 
$ e^{ix}+e^{-ix} = 2\cos x $ in \eqref{bdf2_3}, we get 
\begin{align}\label{bdf2_4}
\frac 2\Dt&\sin(\tomega \Dt/2)
\bigl(3e^{-i\tomega \Dt/2} - e^{i\tomega \Dt/2}\bigr)\\
&=\bigl(k^2 + \lambda H_1(\omega,  \Dt)\bigr) \cos(\tomega \Dt/2)
(3e^{-i\tomega \Dt/2} - e^{i\tomega \Dt/2}).\notag
\end{align}

Since $3e^{-i\tomega \Dt/2} -e^{i\tomega \Dt/2} \ne 0$ in \eqref{bdf2_4}, 
the numerical dispersion relation  of the modified BDF2 scheme can be written as
\begin{align}\label{bdf2_5}
\frac 2\Dt \sin(\tomega \Dt/2)
=& \bigl(k^2 + \lambda H_1(\omega, \Dt)\bigr)\cos(\tomega \Dt/2)\\
=&\bigl(k^2 +\frac \lambda 2 (5 - 3\cos(\tomega \Dt))\bigr)\cos(\tomega \Dt/2).\notag
\end{align}
The desired equation \eqref{bdf_0} holds by letting $H(\omega, \Dt) = H_1(\omega, \Dt)$. 

Next, we consider the modified BDF3 scheme. Using the identities 
 $e^{ix}+e^{-ix} = 2\cos x $ and 
  $\cos{2\theta} = 2\cos^2{\theta} -1 $ 
we get
\begin{align}\label{bdf3_1}
&\Bigl|\frac{11u^n-7u^{n-1} +2u^{n-2}}6 \Bigr|^2
= \Bigl|\frac 16 e^{ikx-i\tomega(n-1)\Dt}\bigl(11e^{-i\tomega \Dt}- 7 + 2e^{i\tomega \Dt}\bigr)\Bigr|^2 \\
&\hskip 1in = \Bigl|\frac 16 e^{ikx-i\tomega(n-1)\Dt}\bigl(13\cos(\tomega \Dt)-
7- 9i\sin(\tomega \Dt)\bigr)\Bigr|^2 \notag \\
&\hskip 1in =\frac{1}{18}\bigl(44\cos^2(\tomega \Dt)-91\cos^2(\tomega \Dt) +65\bigr) \notag\\
&\hskip 1in =\frac{1}{36}\bigl(83 - 47\cos(2\tomega \Dt)\bigr).\notag
\end{align}

Substituting $u^n = e^{i(kx - \tomega n \Dt)}$ into the modified BDF3 scheme, we obtain 
\begin{align}\label{bdf3_2}
&\frac i\Dt \Bigl[
\frac{11 (e^{-i\tomega \Dt}-1)+9(e^{i\tomega \Dt}-1)-2(e^{2i\tomega \Dt}-1)}{6}\Bigr]\\
&\quad =\frac 12 \bigl(k^2 + \lambda H_2(\omega, \Dt)\bigr)
\frac{11 (e^{-i\tomega \Dt}-1)-5(e^{i\tomega \Dt}-1)+2(e^{2i\tomega \Dt}-1)-4}{6}, \notag 
\end{align}
where $H_2(\omega, \Dt) = \frac{1}{36}\bigl(83 - 47 \cos(2\tomega \Dt)\bigr)$.

Similar to the derivation of \eqref{bdf2_4},    
using identities $e^{ix}-e^{-ix} = 2i\sin x$, $e^{ix}+e^{-ix} = 2\cos x$, $e^{ix}-e^{-ix} = 2i\sin x$, 
and $e^{ix}+e^{-ix} = 2\cos x$ in \eqref{bdf3_2}, we get 
\begin{align}\label{bdf3_3}
&\sin(\tomega \Dt/2)(11e^{-i\tomega \Dt/2} - 9e^{i\tomega \Dt/2})
+2\sin(\tomega \Dt)e^{i\tomega \Dt}\\
&\qquad = \cos(\tomega \Dt/2)(11e^{-i\tomega \Dt/2} - 5e^{i\tomega \Dt/2})
+2\cos(\tomega \Dt)e^{i\tomega \Dt} - 2\notag\\
&\sin(\tomega \Dt/2)(11e^{-i\tomega \Dt/2} - 9e^{i\tomega \Dt/2}
+ 4\cos(\tomega \Dt/2)e^{i\tomega \Dt}) \label{bdf3_4} \\
&\qquad =\cos(\tomega \Dt/2)(11e^{-i\tomega \Dt/2} - 9e^{i\tomega \Dt/2}
+ 4\cos(\tomega \Dt/2)e^{i\tomega \Dt}),\notag
\end{align}
where we have used the following identity in \eqref{bdf3_4}:
\begin{align*}
2\cos(\tomega \Dt)e^{i\tomega \Dt}-2
&=2(2\cos^2(\tomega \Dt/2)-1)e^{i\tomega \Dt}\notag\\
&=4\cos^2(\tomega \Dt/2)e^{i\tomega \Dt} - 2(e^{i\tomega \Dt}+1)\\
&=4\cos^2(\tomega \Dt/2)e^{i\tomega \Dt} - 4\cos(\tomega \Dt/2)e^{i\tomega \Dt/2}.
\end{align*}

Since $(11e^{-i\tomega \Dt/2} - 9e^{i\tomega \Dt/2} + 4\cos{(\tomega \Dt/2)}e^{i\tomega \Dt})\ne 0$
in \eqref{bdf3_4}, the numerical dispersion relation for the modified BDF3 can be written as
\begin{align}\label{bdf3_5}
\frac{2}{ \Dt}\sin(\tomega \Dt/2)
=& \bigl(k^2 + \lambda H_2(\omega,  \Dt)\bigr)
\cos(\tomega \Dt/2)\\
=& \bigl(k^2 + \frac{\lambda}{36}(83 - 47\cos(2\tomega \Dt))\bigr)
\cos(\tomega \Dt/2), \notag
\end{align}
which gives \eqref{bdf_0} after setting $H(\omega, \Dt) = H_2(\omega,\Dt)$. 
The proof is complete.
\end{proof}

\begin{table}[htb]\small\centering\footnotesize
\caption{Dispersion error rates of M-BDF schemes.}
\begin{tabular*}{\hsize}{@{}@{\extracolsep{\fill}}cccccc@{}}\toprule
Modified BDFs           &$\lambda$          &                $k$&$ \Dt$&Dispersion errors&Error rates\\ 
\midrule
\multirow{3}{*}{M--BDF2}&\multirow{19}{*}{2}&\multirow{19}{*}{1}& 1E-02& 	 3.752127E-04&         --\\
	   		      	    &				    &     			    & 1E-03& 	 3.750001E-06&     2.0002\\
	   		      	    &				    &     			    & 1E-04&  	 3.750000E-08&     2.0000\\
\\			  
\multirow{3}{*}{M--BDF3}&			        &                   & 1E-02&     2.467750E-05& 	       --\\
	   		      	    &				    &             	    & 1E-03&     2.499677E-05&     1.9944\\
	   		      	    &				    &     			    & 1E-04& 	 2.499996E-07&     2.0000\\
\\ 	
\multirow{3}{*}{M--BDF4}&				    &				    & 1E-02& 	 2.521704E-05&         --\\
	   		      	    &				    &     			    & 1E-03& 	 2.500217E-07&     2.0037\\
	   		      	    &				    &     			    & 1E-04& 	 2.500002E-09&     2.0000\\
\\
\multirow{3}{*}{M--BDF5}&				    &				    & 1E-02& 	 2.500172E-05&         --\\
	   		      	    &				    &     			    & 1E-03& 	 2.500001E-07&     2.0000\\
	   		      	    &				    &     			    & 1E-04& 	 2.500000E-09&     2.0000\\
\\ 	
\multirow{3}{*}{M--BDF6}&				    &				    & 1E-02& 	 2.500124E-05&         --\\
	   		      	    &				    &     			    & 1E-03& 	 2.500001E-07&     2.0000\\
	   		      	    &				    &     			    & 1E-04& 	 2.500005E-09&     2.0000\\
\bottomrule
\end{tabular*}
\label{table_bdfs_disp_err}
\end{table}

Table~\ref{table_bdfs_disp_err} shows the computed dispersion errors and convergence rates 
for the modified BDF schemes. The numerical results indicate that these modified BDF schemes have a second 
order dispersion error.

\section{Numerical experiments: validating the convergence rates.} \label{sec-7}
In this section, we present several 1D numerical tests to illustrate our theoretical results, in particular, 
to verify the rates of convergence of the proposed time-stepping schemes. Our computations are done using the 
software package FEniCS and the linear finite element method is employed for the spatial discretization 
in all our numerical tests.

We consider the cubic nonlinear Schr\"odinger equation \cite{Taghizadeh2011Exact, Xu2005Local,  bao2013numerical}
 (i.e., $f(s) = s$, $\lambda = -2$)
\begin{alignat}{2}\label{nls_sol}
i u_t +\Delta u +2|u|^2 u &=0, &&\qquad t>0,\\
u (0) &= u _0, && \label{nls_solb}
\end{alignat}
where initial data  $u_0 = \sech(x)\exp(2ix)$ is chosen so that the
exact solution is given by \cite{Lin2017Numerical}
\begin{align}\label{exact_sol}
u(x,t) = \sech(x-4t)\exp(i(2x-3t)).
\end{align}

\begin{figure}[htb]	
\centering
\subfigure[]{
\label{fig_cn_m_e:a} 
\includegraphics[height=1.8in,width=2.4in]{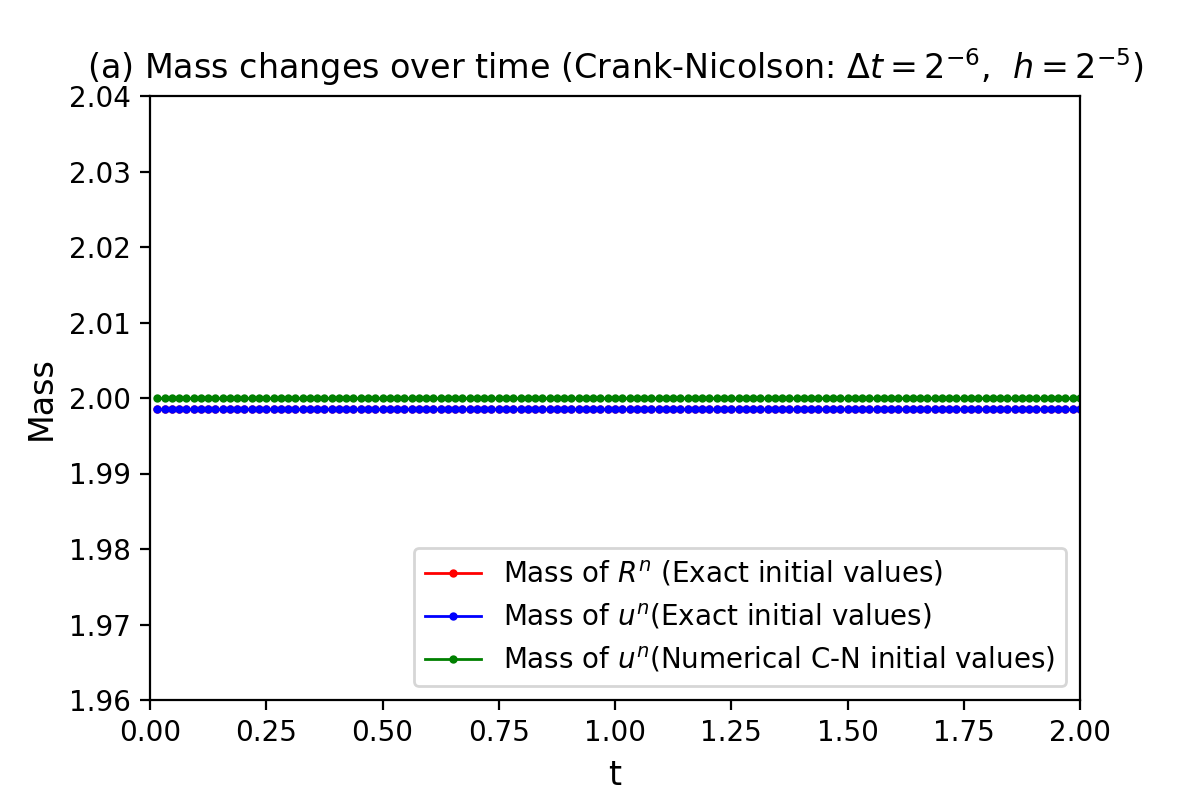}}
\subfigure[]{
\label{fig_cn_m_e:b} 
\includegraphics[height=1.8in,width=2.4in]{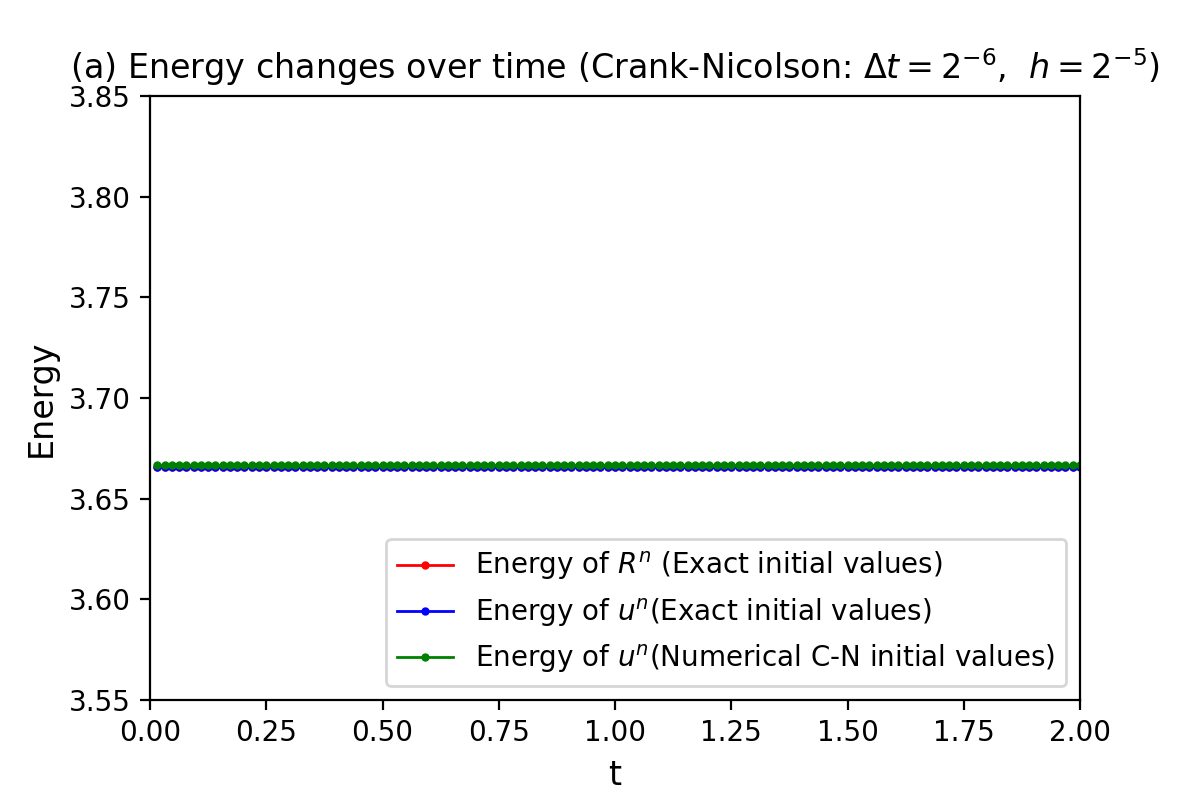}}
\subfigure[]{
\label{fig_cn_m_e:c} 
\includegraphics[height=1.8in,width=2.4in]{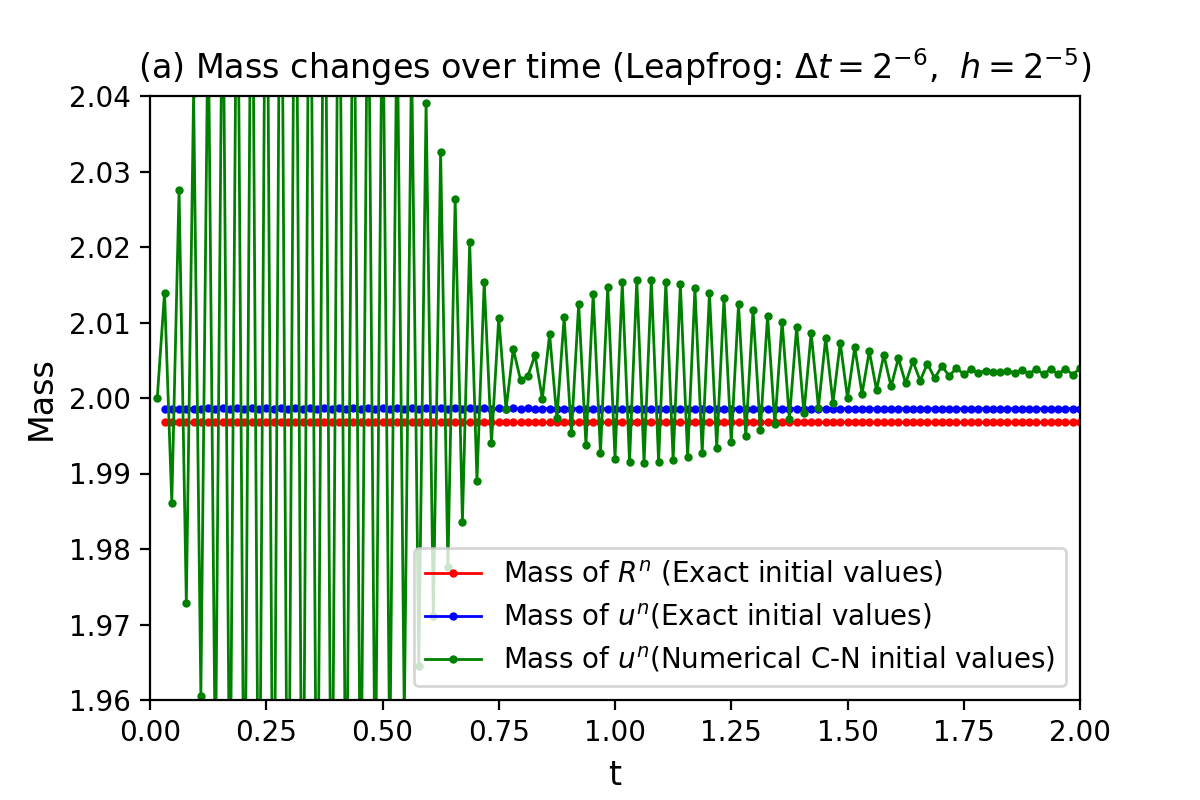}}
\subfigure[]{
\label{fig_cn_m_e:d} 
\includegraphics[height=1.75in,width=2.4in]{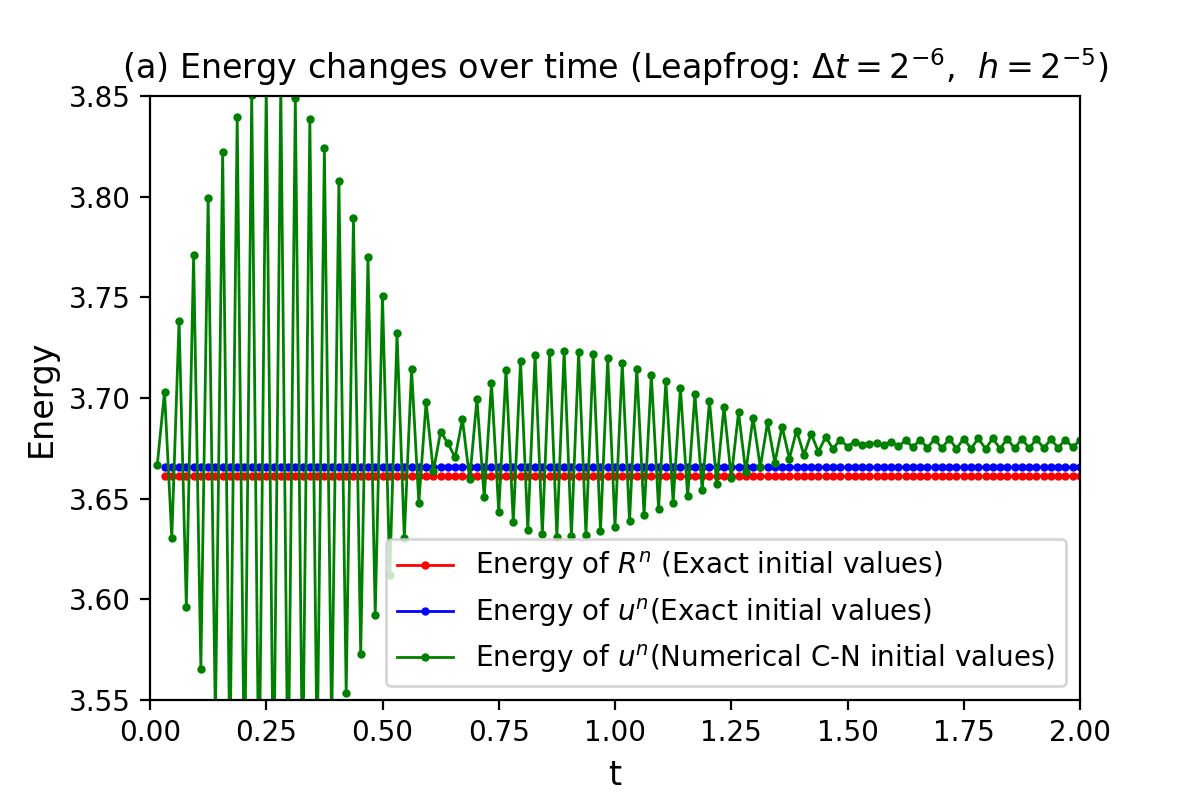}}
\subfigure[]{
\label{fig_bdf2_m_e:a} 
\includegraphics[height=1.75in,width=2.4in]{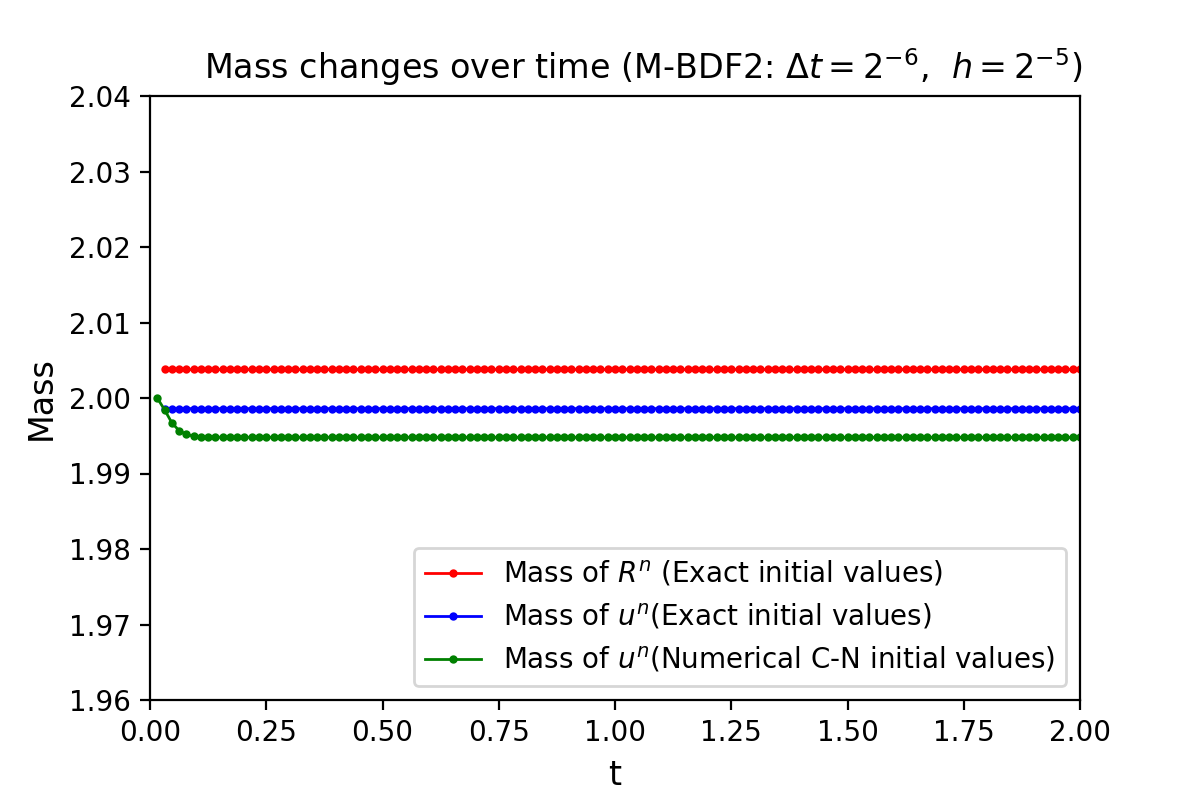}}
\subfigure[]{
\label{fig_bdf2_m_e:b} 
\includegraphics[height=1.75in,width=2.4in]{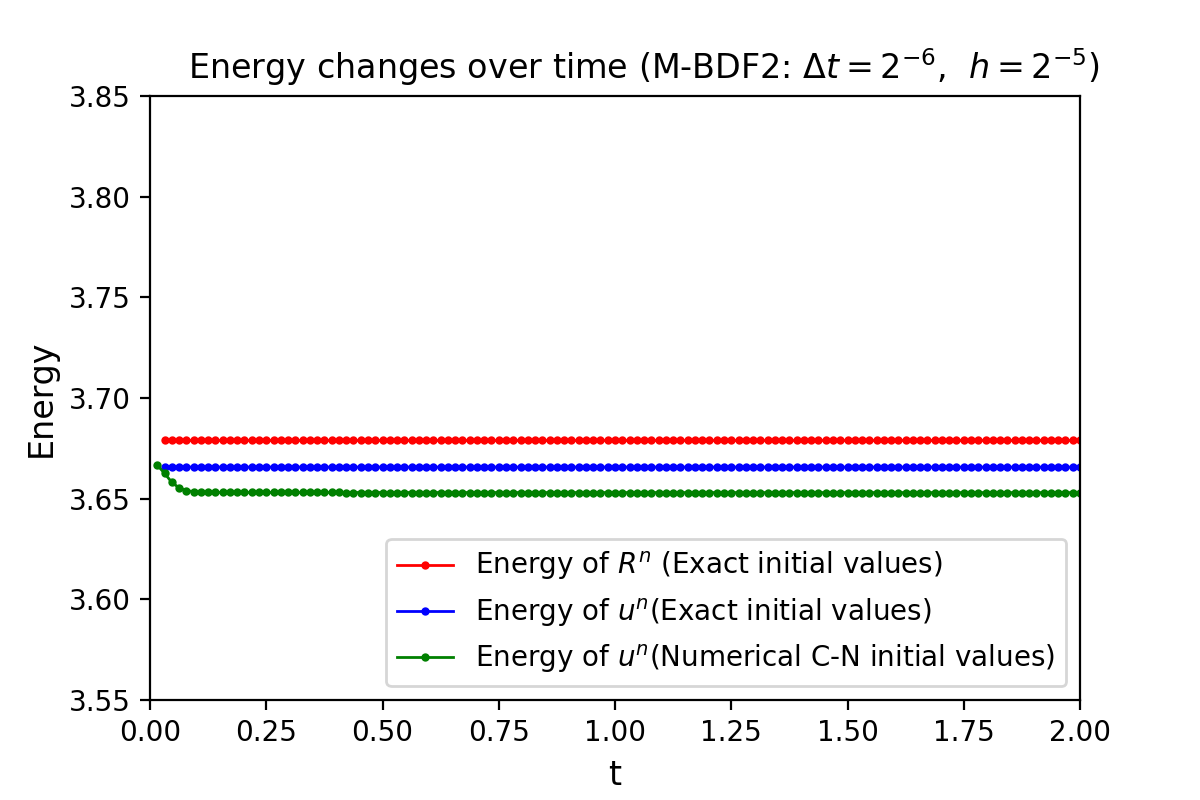}}
\subfigure[]{
\label{fig_bdf3_m_e:a} 
\includegraphics[height=1.75in,width=2.4in]{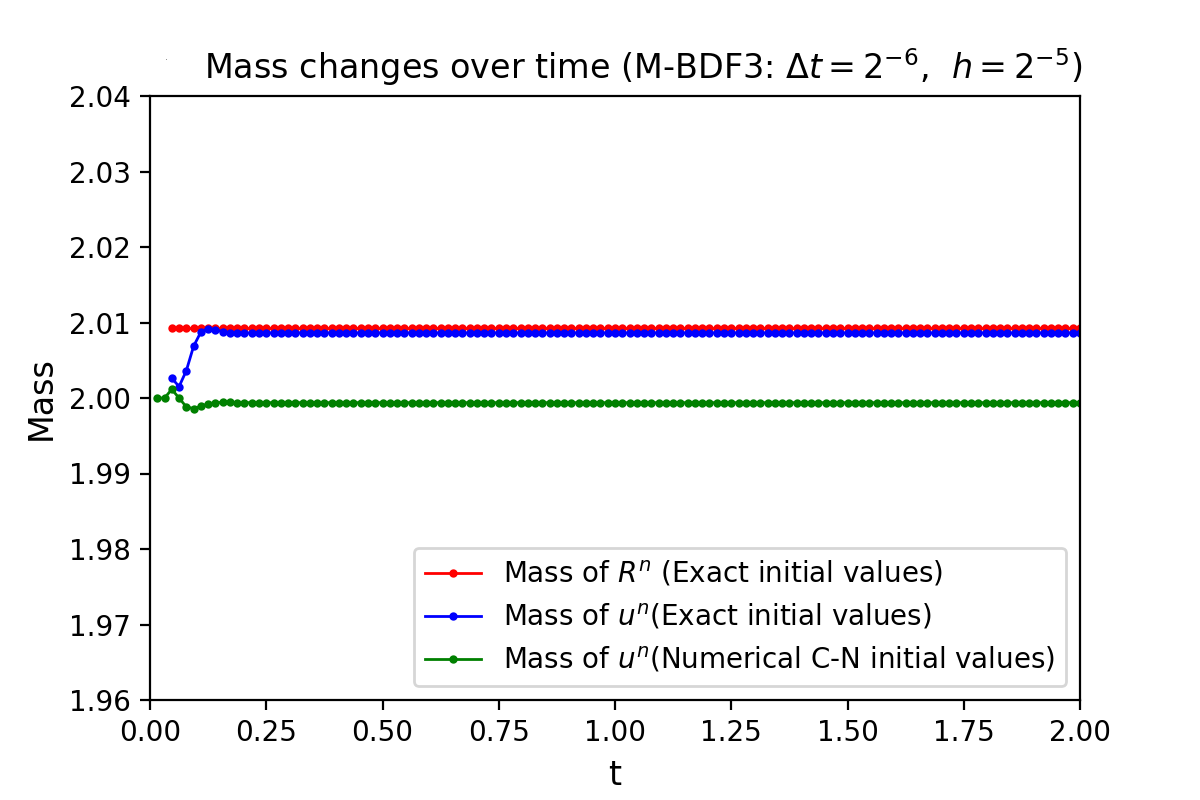}}
\subfigure[]{
\label{fig_bdf3_m_e:b} 
\includegraphics[height=1.75in,width=2.4in]{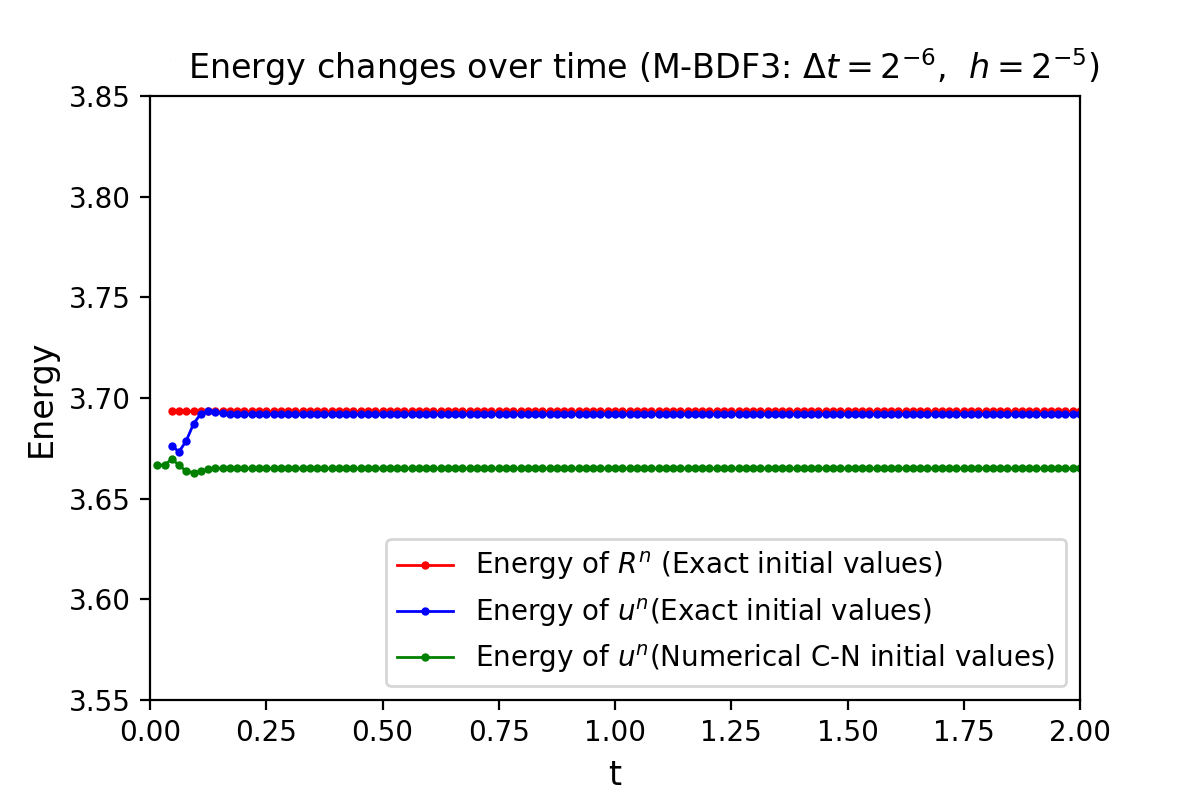}}
\caption{The computed mass  (left) and energy (right) of the right propagation problem by 
the Crank-Nicolson scheme, Leapfrog scheme, M-BDF2 scheme and M-BDF3 scheme  
with $(h=2^{-5}$ and $ \Dt =2^{-6})$. }
\label{fig_cn_m_e}
\end{figure}

\begin{figure}[htb]
\centering
\subfigure[]{
\label{fig_bdf4_m_e:a} 
\includegraphics[height=1.75in,width=2.4in]{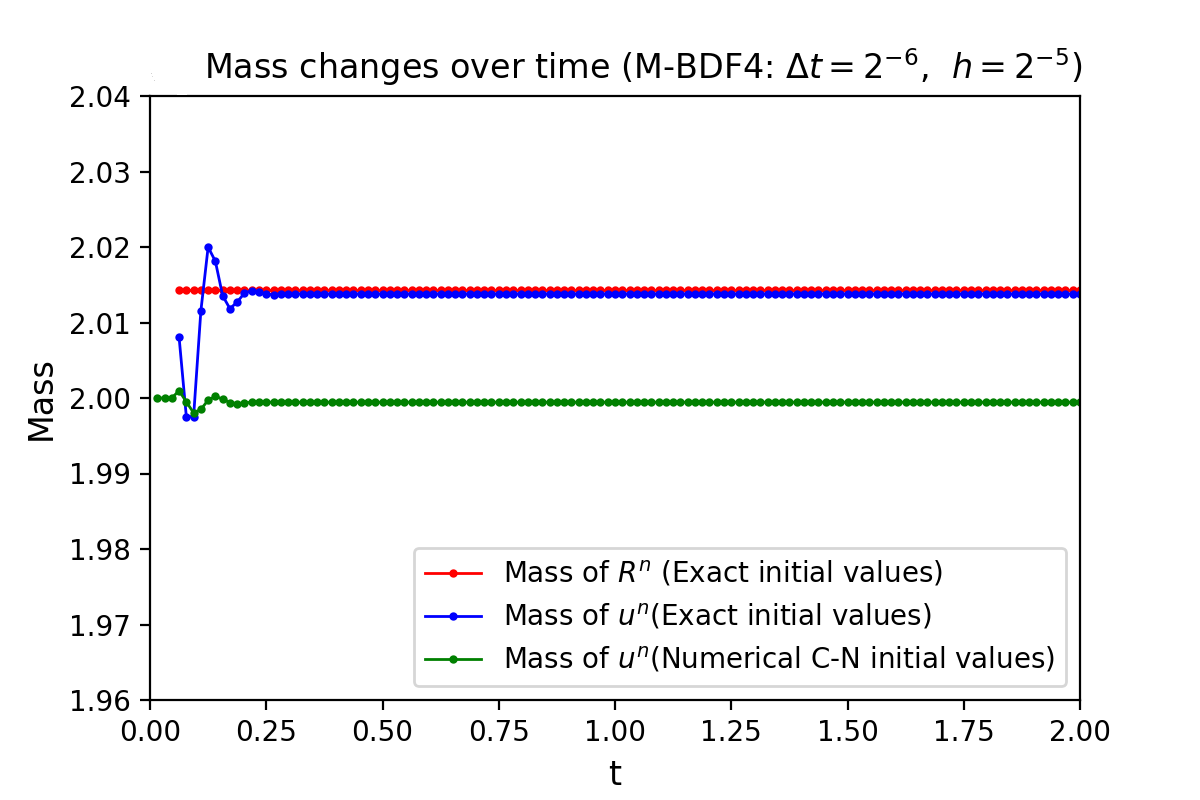}}
\subfigure[]{
\label{fig_bdf4_m_e:b} 
\includegraphics[height=1.75in,width=2.4in]{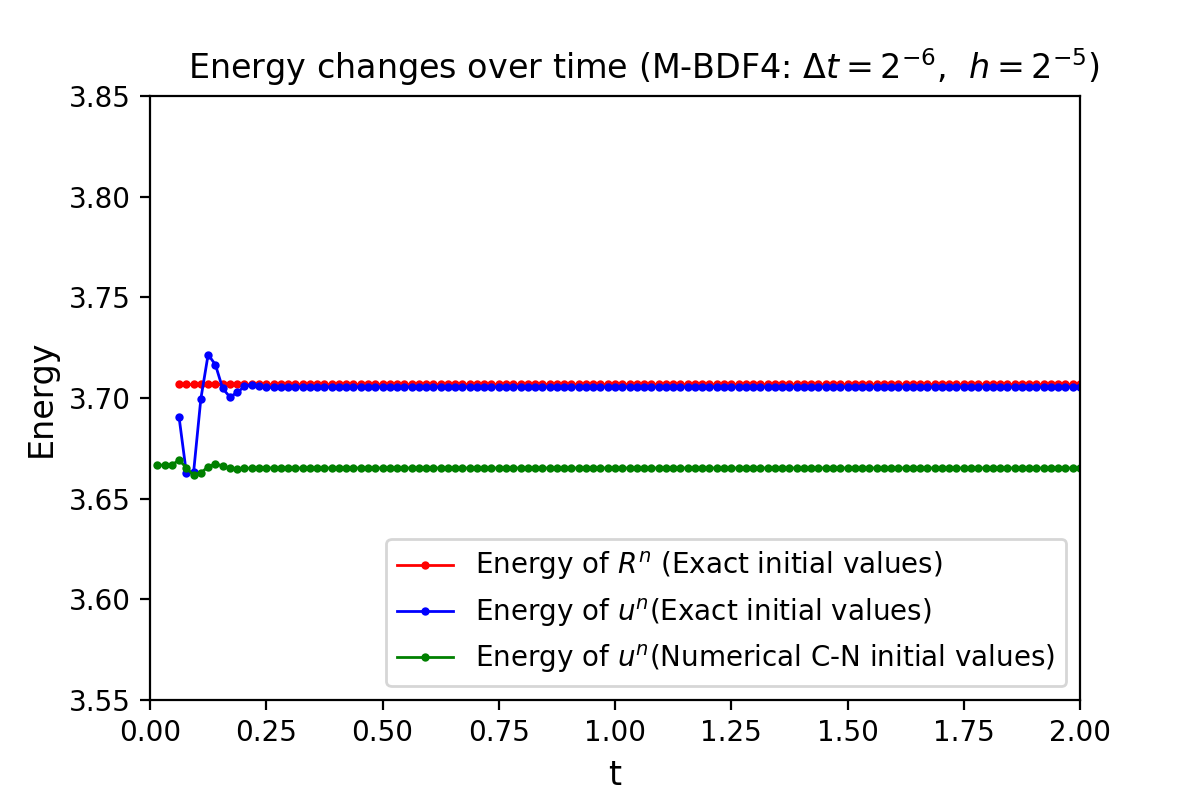}}
\subfigure[]{
\label{fig_bdf5_m_e:a} 
\includegraphics[height=1.75in,width=2.4in]{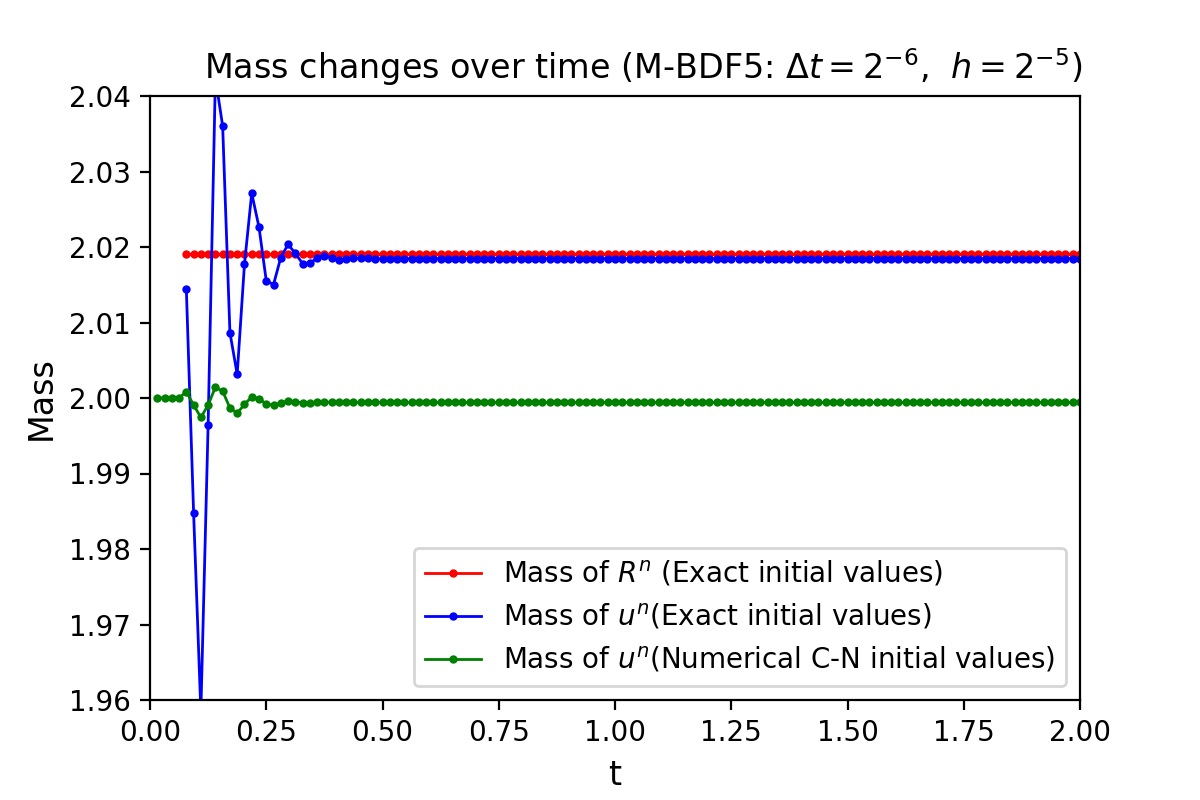}}
\subfigure[]{
\label{fig_bdf5_m_e:b} 
\includegraphics[height=1.75in,width=2.4in]{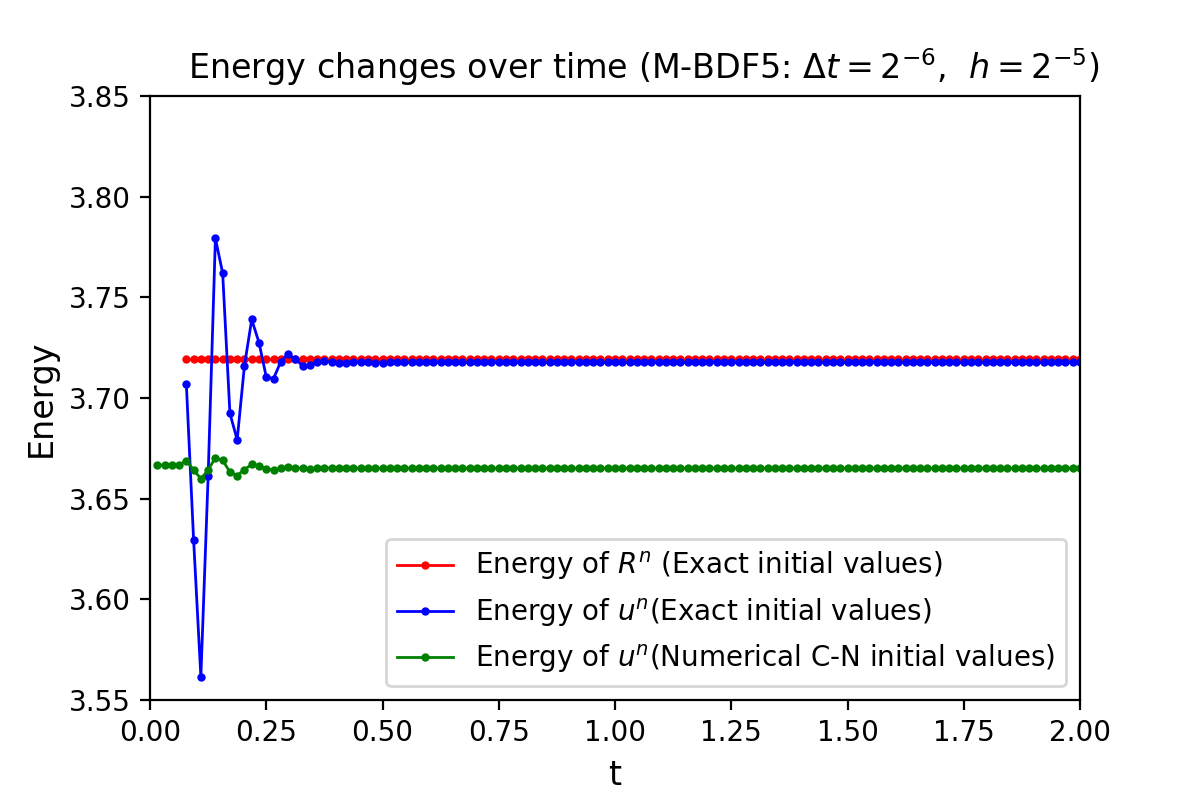}}
\subfigure[]{
\label{fig_bdf6_m_e:a} 
\includegraphics[height=1.75in,width=2.4in]{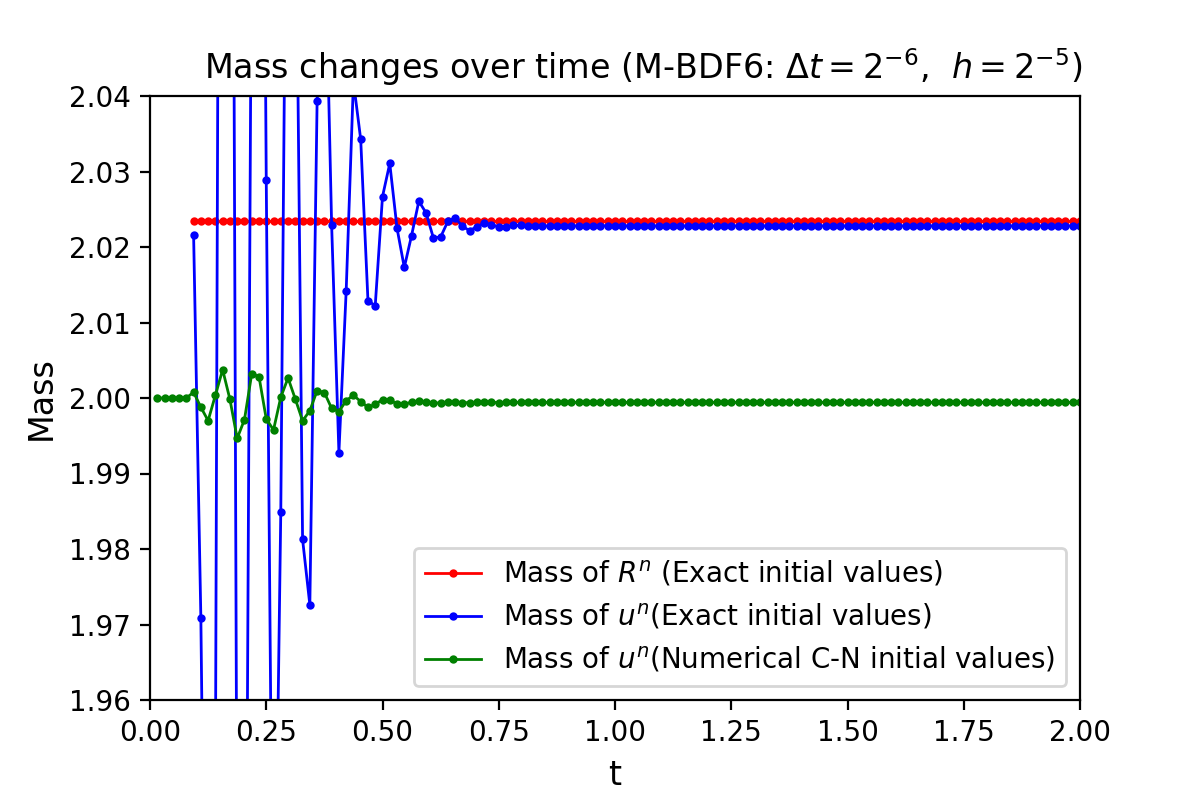}}
\subfigure[]{
\label{fig_bdf6_m_e:b} 
\includegraphics[height=1.75in,width=2.4in]{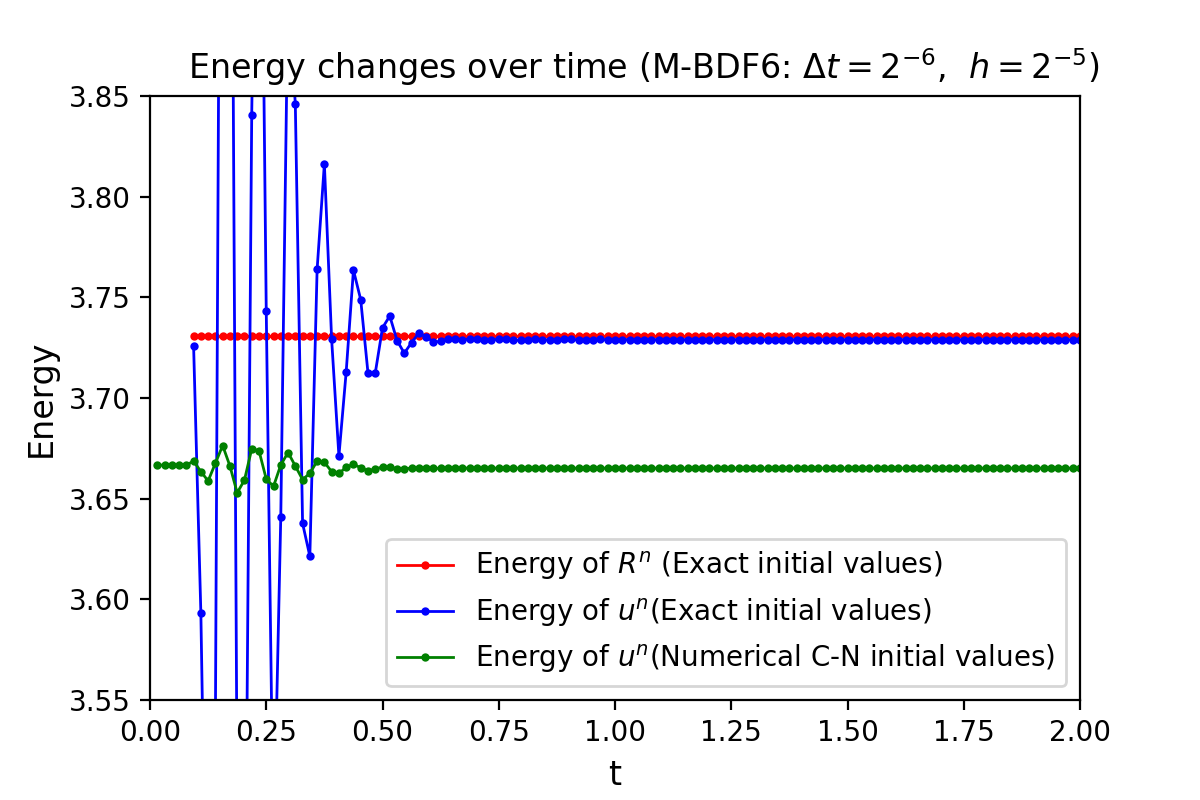}}
\subfigure[]{
\label{fig_bdfs_m_e:a} 
\includegraphics[height=1.75in,width=2.4in]{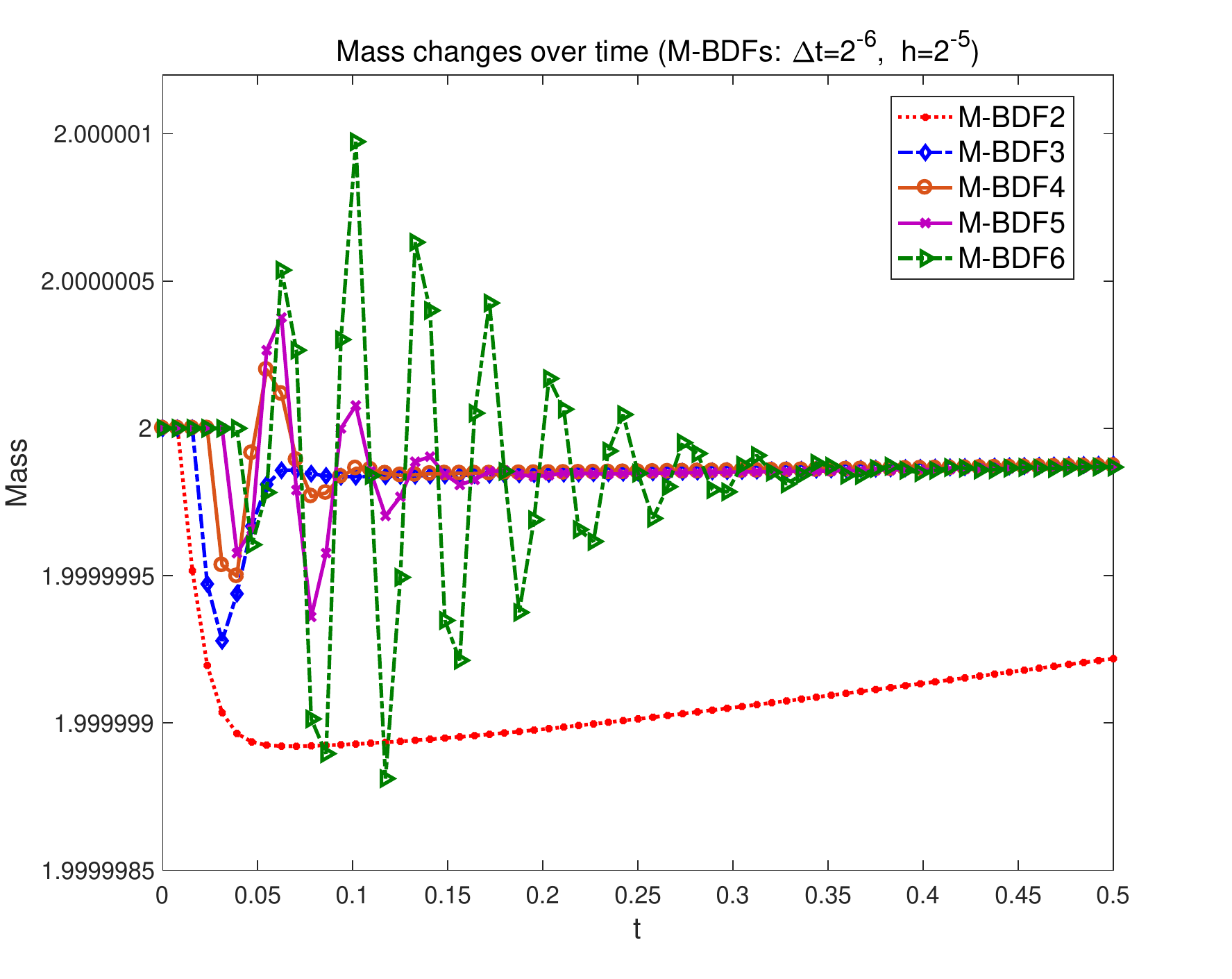}}
\subfigure[]{
\label{fig_bdfs_m_e:b} 
\includegraphics[height=1.75in,width=2.4in]{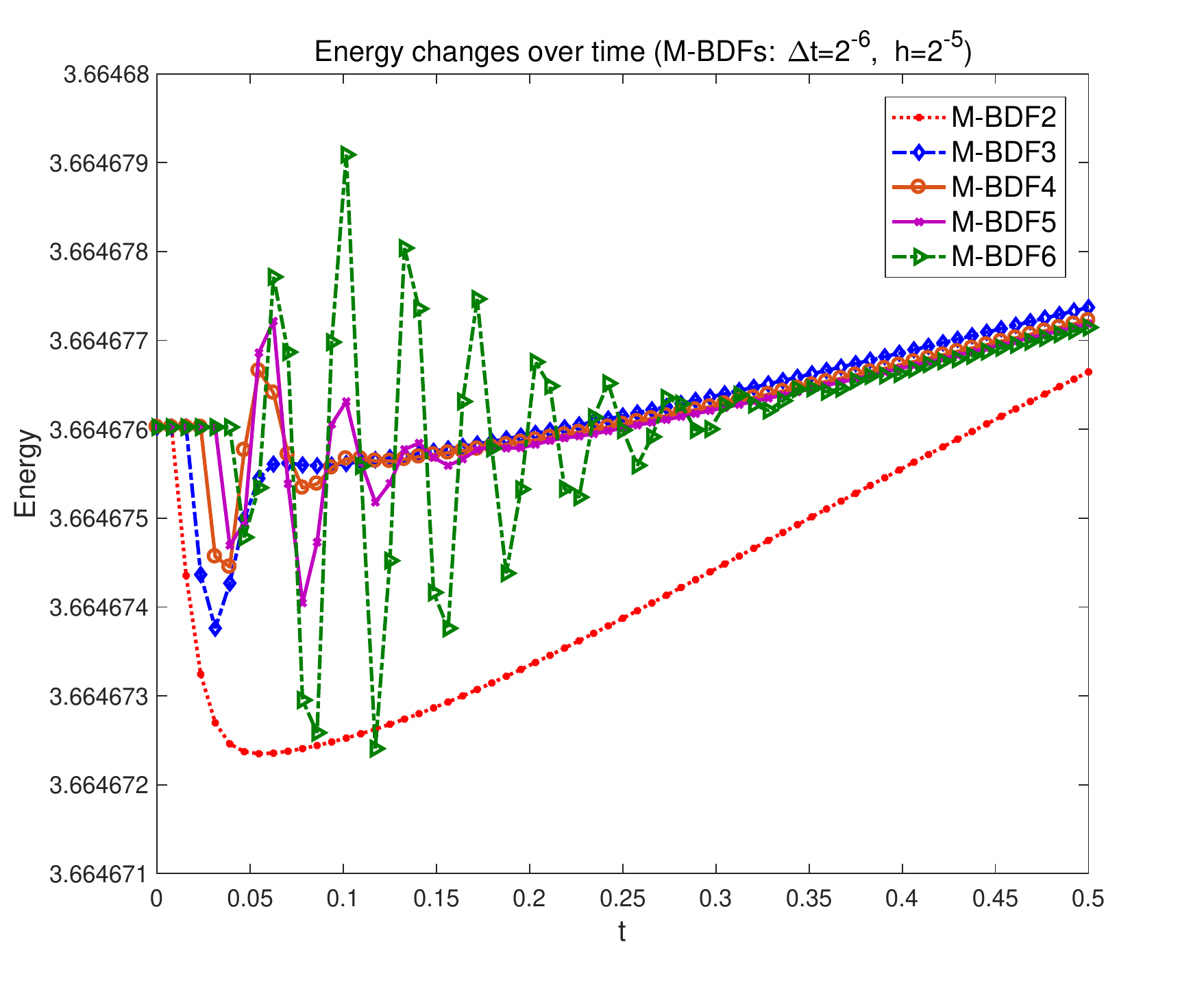}}
\caption{The computed mass (left) and energy (right) of the right propagation problem by 
the modified BDF schemes (s=3,4,5) with $(h=2^{-5}$ and  $\Dt =2^{-6})$,  
and the comparison of the solutions obtained 
by different modified BDF schemes (the zoom-in figures from $t=0$ to $t=0.5$). }
\label{fig_bdf_m_e}
\end{figure}

We solve problem  \eqref{nls_sol}--\eqref{nls_solb} with the periodic boundary condition in $[-20, 20]$
by a few selected schemes from the family of  the time--stepping schemes  proposed in Section~\ref{sec-3}.  
In  Theorem~\ref{Co:conservation}, we proved the mass- 
and energy-conservation properties of scheme \eqref{scheme_1}--\eqref{scheme_2} 
without any restrictions on $h$ and $\Dt$. 
To show the conservation properties of problem \eqref{nls_sol}--\eqref{nls_solb},
uniform spatial and temporal  meshes are used with $h = 2^{-5}$ and $\Dt = 2^{-6}$ 
and the discrete mass and energy are defined in \eqref{dis_mh}.

The evolution of the $L^2$-norm of $R^n$ and $u^n$ for one trajectory of 
the modified Crank-Nicolson scheme and Leapfrog scheme are shown in Figure~\ref{fig_cn_m_e}.
For the Crank-Nicolson scheme, we also show in Figure~\ref{fig_cn_m_e:a} that the evolution  
of the $L^2$-norm is exactly conserved although at each time step, the nonlinear equation 
is not exactly solved. Meanwhile, the energy is also exactly conserved  as shown in Figure~\ref{fig_cn_m_e:b}. 

It should be noted that the Leapfrog scheme is a multi-step method,  it requires two starting values, 
which are usually generated by a one-step method. 
In order to choose a suitable starting one-step scheme, we present computational results using different 
starting values generated by the exact solution and the Crank-Nicolson scheme.
The computed evolution of $L^2$-norm and the energy of  the Leapfrog scheme are shown in 
Figure~\ref{fig_cn_m_e:c} and Figure~\ref{fig_cn_m_e:d}, from which we observe that the  
evolution has a large oscillation at the early stage, and gradually becomes stable.

The evolution of the $L^2$-norm for one trajectory of the  modified BDF schemes 
are show in Figure~\ref{fig_bdf2_m_e:a},\subref{fig_bdf3_m_e:a} and Figure~\ref{fig_bdf4_m_e:a},
\subref{fig_bdf5_m_e:a},\subref{fig_bdf6_m_e:a}.
As expected, all of them gradually become stable and are exactly conserved. 
The differences of these oscillations (the zoom-in graphics at the beginning stage of the evolution) 
are obtained in Figure~\ref{fig_bdfs_m_e:a}. We observe that all modified BDF schemes start with some 
oscillations and the amplitudes of the oscillations quickly diminish with time. 
In addition,  similar phenomenon for the energy are also seen in  
Figure~\ref{fig_bdf2_m_e:b},\subref{fig_bdf3_m_e:b}, Figure~\ref{fig_bdf4_m_e:b},\subref{fig_bdf5_m_e:b},
\subref{fig_bdf6_m_e:b} and Figure~\ref{fig_bdfs_m_e:b}.

The accuracy of a proposed method is examined numerically by comparing the solution obtained on 
a sequence of coarse (time) meshes with the exact solution given in \eqref{exact_sol}.  
The computed errors 
and rates of the Crank-Nicolson scheme are shown in Table~\ref{table_cn_errors}. We observe 
that the $L^2$-norm errors decrease by a factor $2$ when the step-size $\Dt$ 
is halved. Hence,  a second order convergence rate  is verified.  
\begin{table}[htb]\small\centering\footnotesize
\caption{Accuracy test of the Crank-Nicolson scheme for NLS equation \eqref{nls_sol}
with the exact solution \eqref{exact_sol}.  $(h=40/N)$ and  $t = 2$. }
\resizebox{\textwidth}{!}{
\begin{tabular}{crcccccccc}
\toprule
\multicolumn{2}{l}{Crank-Nicolson}&\multicolumn{4}{l}{Real Part}&\multicolumn{4}{l}{Imaginary Part}\\
\cmidrule(lr){1-2} \cmidrule(lr){3-6} \cmidrule(lr){7-10}
$N$ &$\Dt$&$L^2$ error&Order&$\Li$ error&Order&$L^2$ error&Order&$\Li$ error&Order \\
\midrule
4000&  1/8&     0.5194&   --&     0.5136&   --&     0.5290&   --&     0.4351&    --\\
4000& 1/16&     0.1337& 1.96&     0.1073& 2.26&     0.1299& 2.03&     0.1057&  2.04\\
4000& 1/32&     0.0319& 2.07&     0.0250& 2.10&     0.0311& 2.06&     0.0249&  2.09\\
\bottomrule
\end{tabular}}
\label{table_cn_errors}
\end{table}

Recall that the Leapfrog scheme is a multi-step numerical method. 
In order to choose a suitable starting (one-step) scheme,   
we present some convergence results using difference starting values generated by the exact 
solution and by the Crank-Nickson scheme in Table~\ref{table_lf_errors}. 
It is clear to see that the $L^2$-norm error rate is $2$, which confirms our theoretical result.

\begin{table}[htb]\small\centering \footnotesize
\caption{Accuracy test of the Leapfrog scheme for NLS equation
\eqref{nls_sol} with the exact solution \eqref{exact_sol}.  $(h=40/N)$ and  $t = 2$.}
\resizebox{\textwidth}{!}{
\begin{tabular}{lcrcccccccc}
\toprule
\multicolumn{3}{l}{Leapfrog}&\multicolumn{4}{l}{Real Part}&\multicolumn{4}{l}{Imaginary Part} \\
\cmidrule(lr){1-3} \cmidrule(lr){4-7} \cmidrule(lr){8-11}
Initial values&  $N$&$\Dt$&$L^2$ error&Order&$\Li$ error&Order&$L^2$ error&Order&$\Li$ error&Order\\
\midrule
Exact         & 4000&  1/8&     0.5194&   --&     0.5136&   --&     0.5290&   --&     0.4351&   --\\
solutions     & 4000& 1/16&     0.1337& 1.96&     0.1073& 2.26&     0.1299& 2.03&     0.1057& 2.04\\
	          & 4000& 1/32&     0.0319& 2.07&     0.0250& 2.10&     0.0311& 2.06&     0.0249& 2.09\\
\\
Numerical     & 4000&  1/8&     0.4606&   --&     0.3583&	--&     0.5208&   --&     0.5005&   --\\
solution      & 4000& 1/16&     0.1163&	1.99&     0.0971& 1.88&  	0.1165&	2.16&     0.0969& 2.37\\
(C-N)         & 4000& 1/32& 	0.1594&	0.74&     0.1228& 0.75& 	0.1596&	0.74&     0.1192& 0.71\\
\bottomrule
\end{tabular}}
\label{table_lf_errors}
\end{table}

We run the same tests for the proposed modified BDF schemes, the computed results are shown  
in Table~\ref{table_bdfs_errors}, we again observe that the $L^2$-norm error rate is $2$.

\begin{table}[htb]\small\centering\footnotesize
\caption{Accuracy test of modified BDF2 scheme for NLS equation
\eqref{nls_sol} with the exact solution \eqref{exact_sol}. $(h=40/N)$ and  $t = 2$.}
\resizebox{\textwidth}{!}{
\begin{tabular}{clcrcccccccc}
\toprule
\multicolumn{4}{l}{M--BDFs}&\multicolumn{4}{l}{Real Part}&\multicolumn{4}{l}{Imaginary Part} \\
\cmidrule(lr){2-4} \cmidrule(lr){5-8} \cmidrule(lr){9-12}
&Initial values&  $N$&$\Dt$&$L^2$ error&Order&$\Li$ error&Order&$L^2$ error&Order&$\Li$ error&Order \\
\midrule
M--BDF2 & Exact   & 4000& 1/16& 0.2557&	  --& 0.2318&   --&	0.2627&	  --& 0.2536&   --\\
		&solutions&     & 1/32& 0.0624& 2.04& 0.0558& 2.06&	0.0640& 2.04& 0.0618& 2.04\\
		&		  &     & 1/64& 0.0149& 2.07& 0.0133& 2.07& 0.0152& 2.07& 0.0148& 2.07\\
\\
		&Numerical& 4000&  1/4& 0.7278&	  --& 0.5425&   --&	0.7598&	  --& 0.6963&	--\\
		&solution &     &  1/8& 0.3144& 1.21& 0.2751& 0.98&	0.3132& 1.28& 0.2664& 1.39\\
		&(C-N)    &     & 1/16& 0.0535& 2.55& 0.0345& 3.00&	0.0506& 2.63& 0.0397& 2.75\\
		\rowcolor{mygray}&&&&&&&&&&&\\
		\rowcolor{mygray}
M--BDF3 &Exact    & 4000& 1/16& 0.3848&   --& 0.3833&   --& 0.3833&   --& 0.3910&   --\\
		\rowcolor{mygray}
		&solutions&     & 1/32& 0.0799& 2.27& 0.0770& 2.32& 0.0770& 2.32& 0.0799& 2.29\\
		\rowcolor{mygray}
		&		  &     & 1/64& 0.0168& 2.25& 0.0153& 2.33& 0.0153& 2.33& 0.0161& 2.31\\
		\rowcolor{mygray}&&&&&&&&&&&\\
		\rowcolor{mygray}
		&Numerical& 4000&  1/4& 0.7491&	  --& 0.5313&   --& 0.7680&	  --& 0.6889&   --\\
		\rowcolor{mygray}
		&solution &     &  1/8& 0.3456& 1.12& 0.2821& 0.91& 0.3470& 1.15& 0.2892& 1.25\\
		\rowcolor{mygray}
		&(C-N)    &     & 1/16& 0.0555& 2.64& 0.0373& 2.92& 0.0505& 2.78& 0.0407& 2.83\\
\\
M--BDF4 &Exact    & 4000& 1/16& 0.5192&   --& 0.5376&   --& 0.5443&   --& 0.4896&   --\\
		&solutions&     & 1/32& 0.1073& 2.27& 0.1043& 2.37& 0.1204& 2.18& 0.0999& 2.29\\
		&		  &     & 1/64& 0.0267& 2.01& 0.0234& 2.16& 0.0310& 1.96& 0.0251& 1.99\\
		\\
		&Numerical& 4000&  1/4& 0.7736&	  --& 0.5937&   --& 0.7849&   --& 0.6807&   --\\
		&solution &     &  1/8& 0.3651& 1.08& 0.3245& 0.87& 0.3637& 1.12& 0.2969& 1.20\\
		&(C-N)    &     & 1/16& 0.0584& 2.64& 0.0398& 3.03& 0.0534& 2.77& 0.0401& 2.89\\
		\rowcolor{mygray}&&&&&&&&&&&\\
		\rowcolor{mygray}
M--BDF5 &Exact    & 4000& 1/16& 0.6498& 1.47& 0.6805& 1.54& 0.6567& 1.57& 0.5517& 1.89\\
		\rowcolor{mygray}
		&solutions&     & 1/32& 0.1357& 2.26& 0.1321& 2.37& 0.1511& 2.12& 0.1199& 2.20\\
		\rowcolor{mygray}
		&         &     & 1/64& 0.0361& 1.91& 0.0313& 2.08& 0.0413& 1.87& 0.0336& 1.83\\
		\rowcolor{mygray}&&&&&&&&&&&\\
		\rowcolor{mygray}
		&Numerical& 4000&  1/4& 0.7491&   --& 0.5313&   --& 0.7680&	  --& 0.6889&   --\\
		\rowcolor{mygray}
		&solution &     &  1/8& 0.3456& 1.12& 0.2821& 0.91& 0.3470& 1.15& 0.2891& 1.25\\
		\rowcolor{mygray}
		&(C-N)    &     & 1/16& 0.0555& 2.64& 0.0373& 2.92& 0.0505& 2.78& 0.0407& 2.83\\
\\
M--BDF6 &Exact    & 4000& 1/16& 0.7683&   --& 0.8078&   --& 0.7492&   --& 0.5869&   --\\
		&solutions&     & 1/32& 0.1630& 2.24& 0.1589& 2.35& 0.1794& 2.06& 0.1389& 2.08\\
		&		  &     & 1/64& 0.0444& 1.88& 0.0384& 2.05& 0.0505& 1.83& 0.0409& 1.76\\
		\\
		&Numerical& 4000&  1/4& 0.8018&   --& 0.6350&   --& 0.8089&	  --& 0.6751&   --\\
		&solution &     &  1/8& 0.3808& 1.07& 0.3515& 0.85& 0.3748& 1.11& 0.3017& 1.16\\
		&(C-N)    &     & 1/16& 0.0579& 2.72& 0.0404& 3.12& 0.0529& 2.83& 0.0399& 2.92\\
\bottomrule
\end{tabular}}
\label{table_bdfs_errors}
\end{table}
 
We conclude this section by presenting a convergence and performance comparison of the 
Leapfrog scheme and the modified BDF schemes in Figure~\ref{fig_bdfs_rates}.
 
\begin{figure}[htb]
\centering
\subfigure[]{
\label{fig_bdfs_rates:a} 
\includegraphics[height=1.8in,width=2.4in]{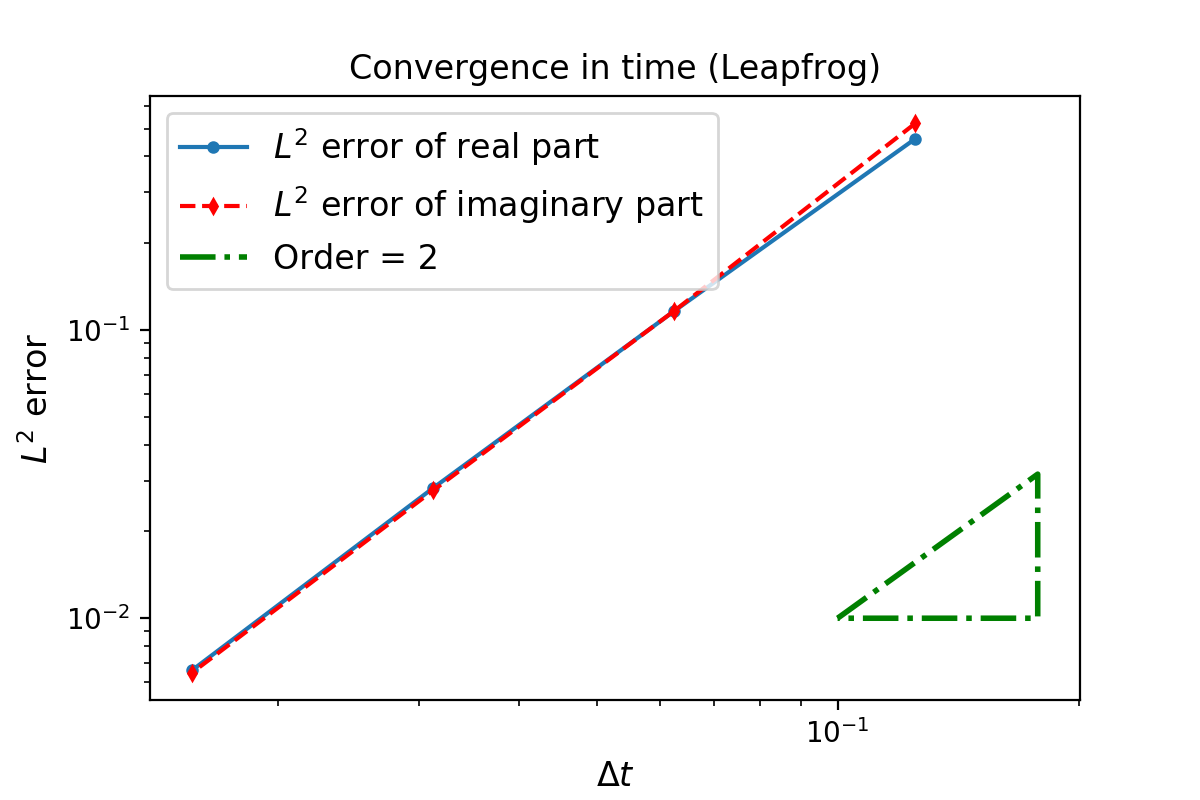}}
\subfigure[]{
\label{fig_bdfs_rates:b} 
\includegraphics[height=1.8in,width=2.4in]{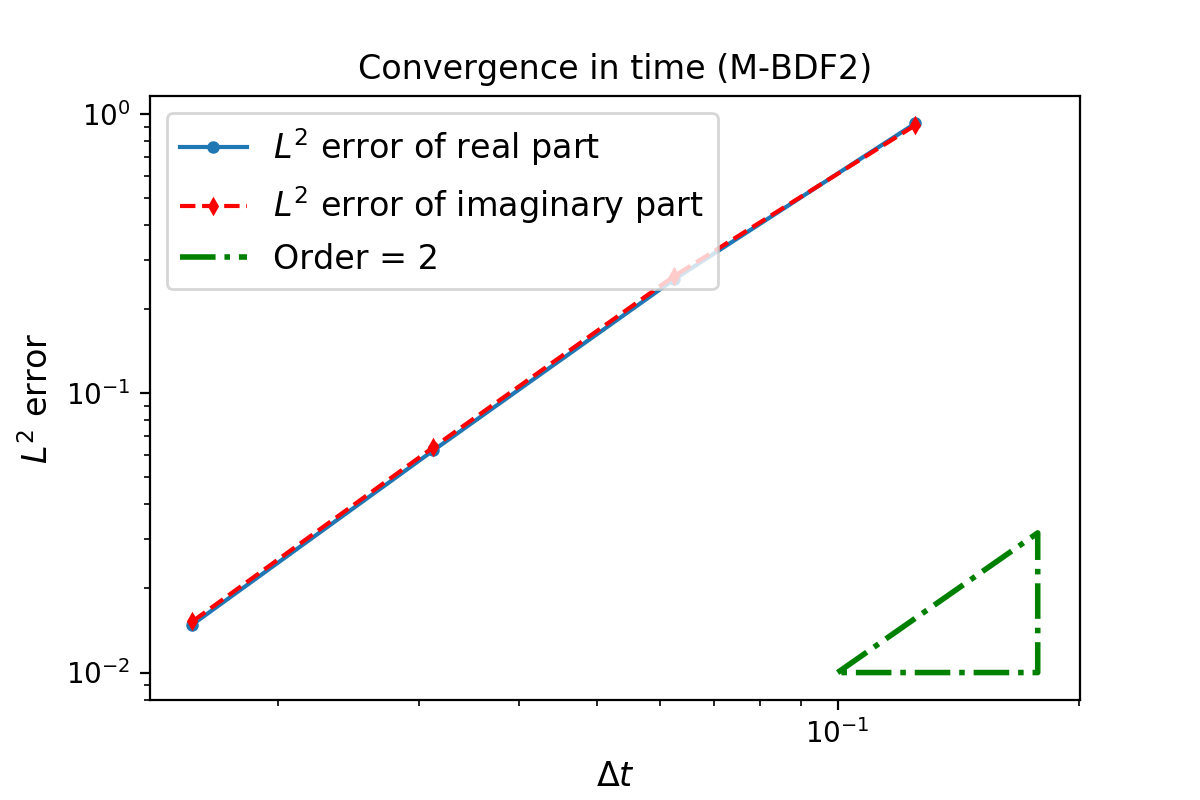}}
\subfigure[]{
\label{fig_bdfs_rates:c} 
\includegraphics[height=1.8in,width=2.4in]{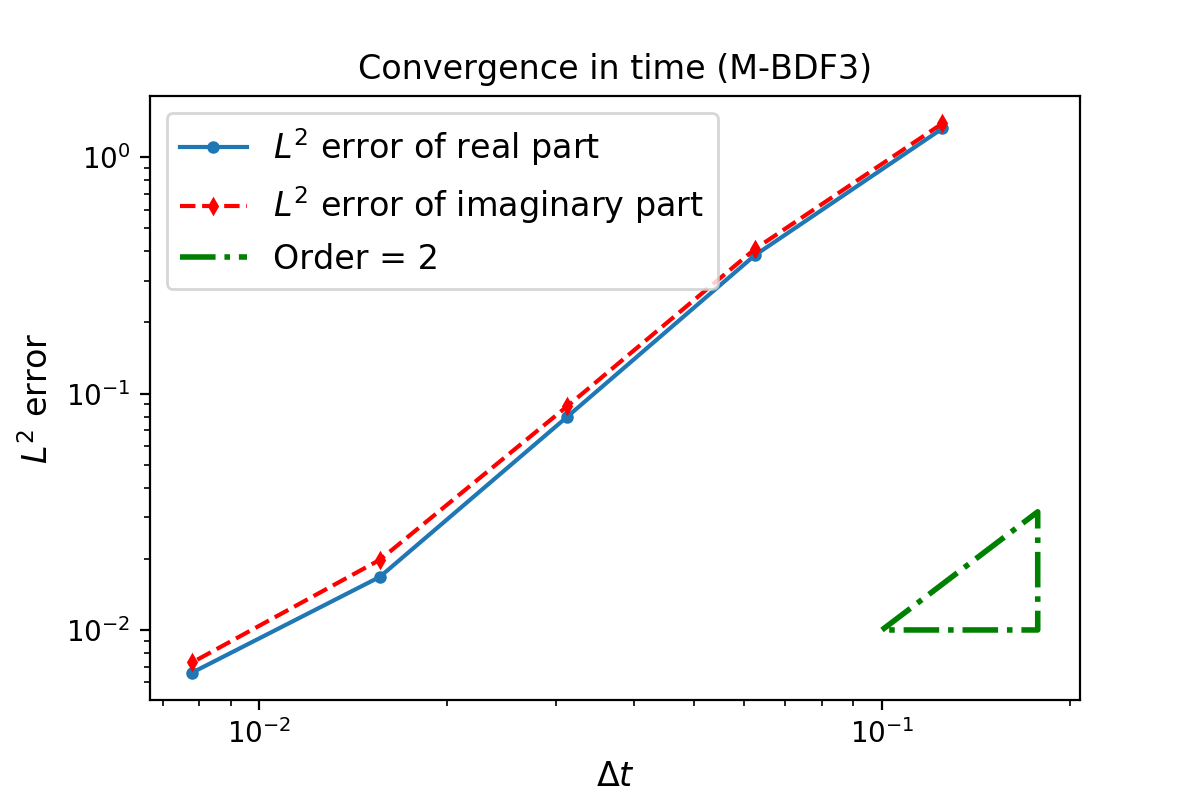}}
\subfigure[]{
\label{fig_bdfs_rates:d} 
\includegraphics[height=1.8in,width=2.4in]{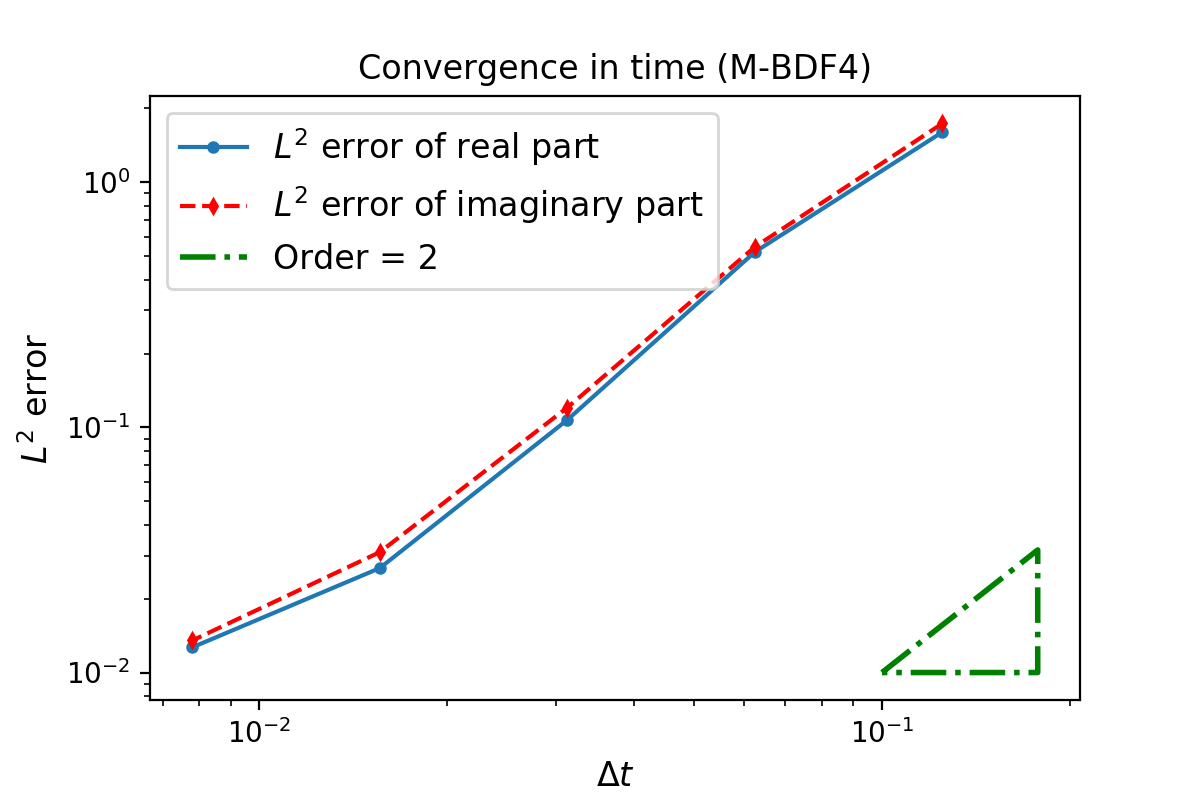}}
\subfigure[]{
\label{fig_bdfs_rates:e} 
\includegraphics[height=1.8in,width=2.4in]{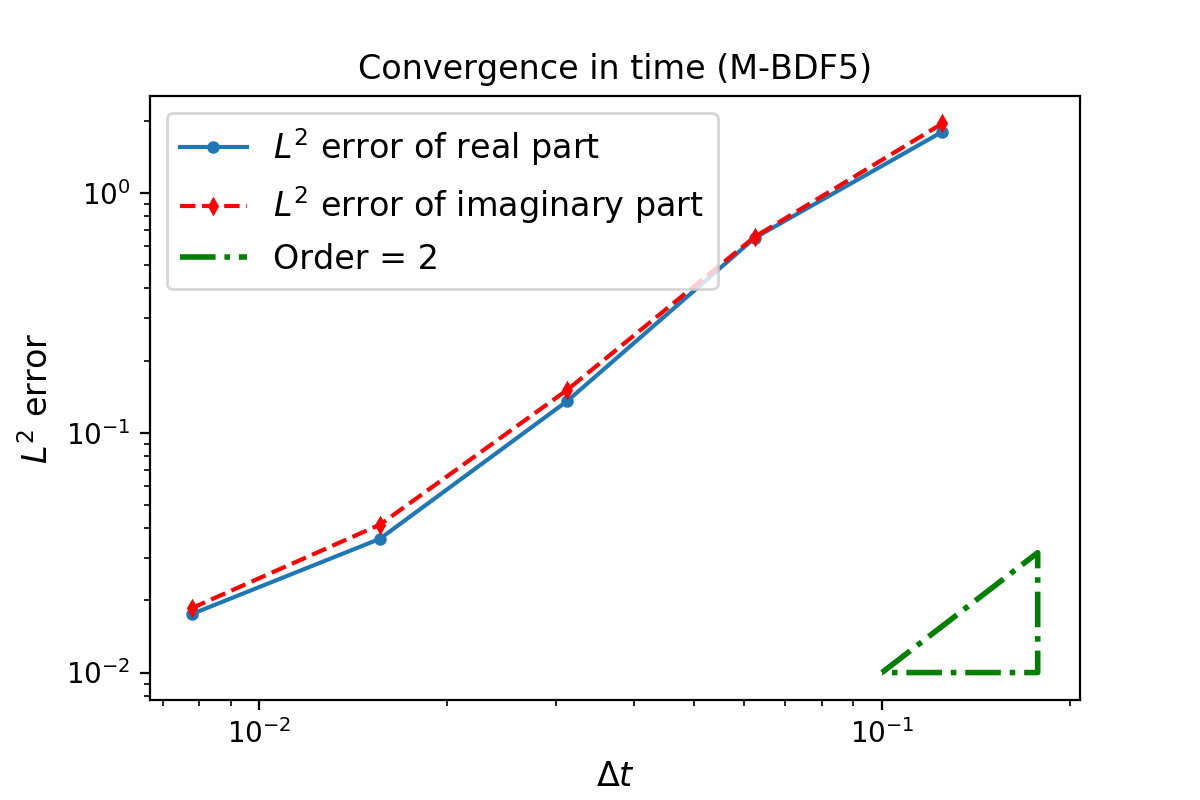}}
\subfigure[]{
\label{fig_bdfs_rates:f} 
\includegraphics[height=1.8in,width=2.4in]{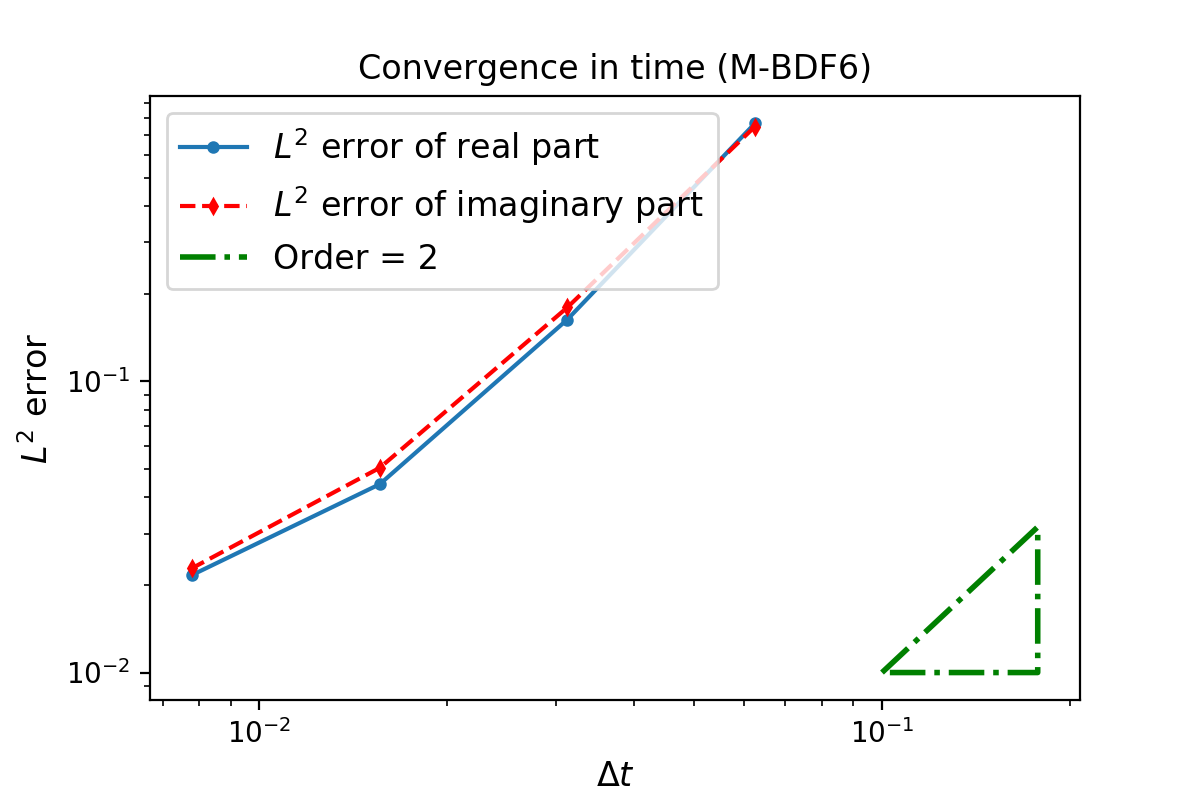}}
\caption{Rates of convergence in the norm $\|u(T)-u^{[T/\Dt]}\|_{L^2}$. $T=2$, 
$N = 4000$, $h=40/N$, $\Dt=2^{-i} (i=3,4,5,6)$.}
\label{fig_bdfs_rates}
\end{figure}

\section{Numerical experiments: capturing the blow-up time.}\label{sec-8}
Our aim in this section is to present a numerical study of the blow-up phenomenon for the 
quintic nonlinear Schr\"odinger equation, which is known to be very delicate to simulate  
in order to have an accurate understanding of this behavior. In particular, we focus 
on predicating and capturing the blow-up time using several proposed schemes. 
Existing numerical results have shown that starting with an initial condition of a given amplitude, 
one can claim that the solution has a singularity as soon as its amplitude becomes three (or more) 
times bigger than the initial amplitude \cite{Debussche2002Numerical}. 

We consider the following quintic nonlinear Schr\"odinger problem:
\begin{alignat}{2}\label{nls_sol_2}
i u_t +\Delta u +|u|^4 u &=0, &&\qquad t>0,\quad  -10 \le x \le 10,\\
u (0) = u &_0, && 
\end{alignat}
with periodic boundary conditions \cite{Lu2015Mass}. The initial condition is chosen as 
$u_0 = 1.6 e^{-x^2}$. Since the initial energy is negative,  it is known that the
a blow-up in the solution must occur in finite time \cite{Debussche2002Numerical}.

\subsection{Comparison of $\Li$-norm profiles of the computed solutions by different schemes.}
In this subsection, we want to test whether all or which of our proposed time-stepping 
schemes will be able to capture the blow-up phenomenon.

Figure~\ref{fig_max} shows the simulation results of various schemes. 
We clearly see the formation of a singularity and that the used mesh size is small enough to
capture the essential feature of the blow-up by all but the Leapfrog scheme in 
Figure~\ref{fig_max:a}, As a result, we conclude that the Leapfrog scheme is not, 
but all other proposed schemes are, capable of capturing the blow-up phenomenon of the quintic 
nonlinear Schr\"odinger equation \eqref{nls_sol_2}. 

\begin{figure}[htb]
\centering
\subfigure[]{
\label{fig_max:a} 
\includegraphics[height=1.75in,width=2.4in]{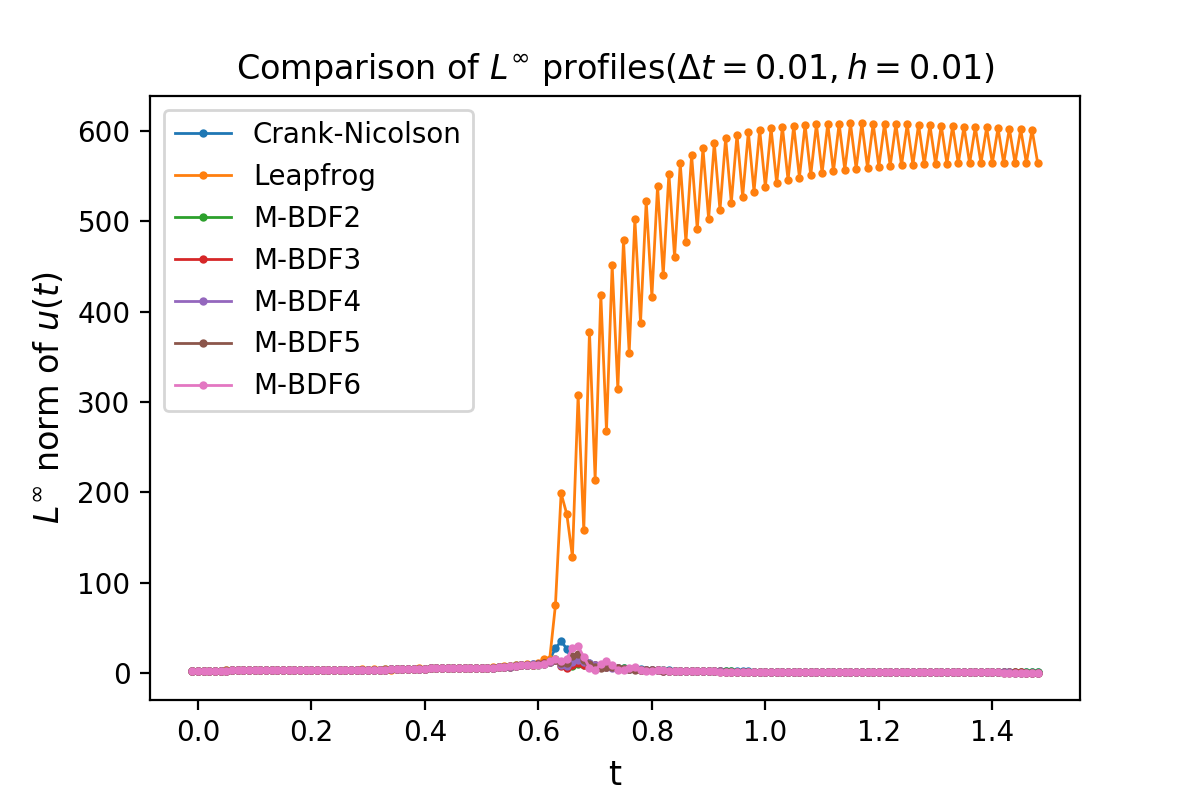}}
\subfigure[]{
\label{fig_max:b} 
\includegraphics[height=1.75in,width=2.4in]{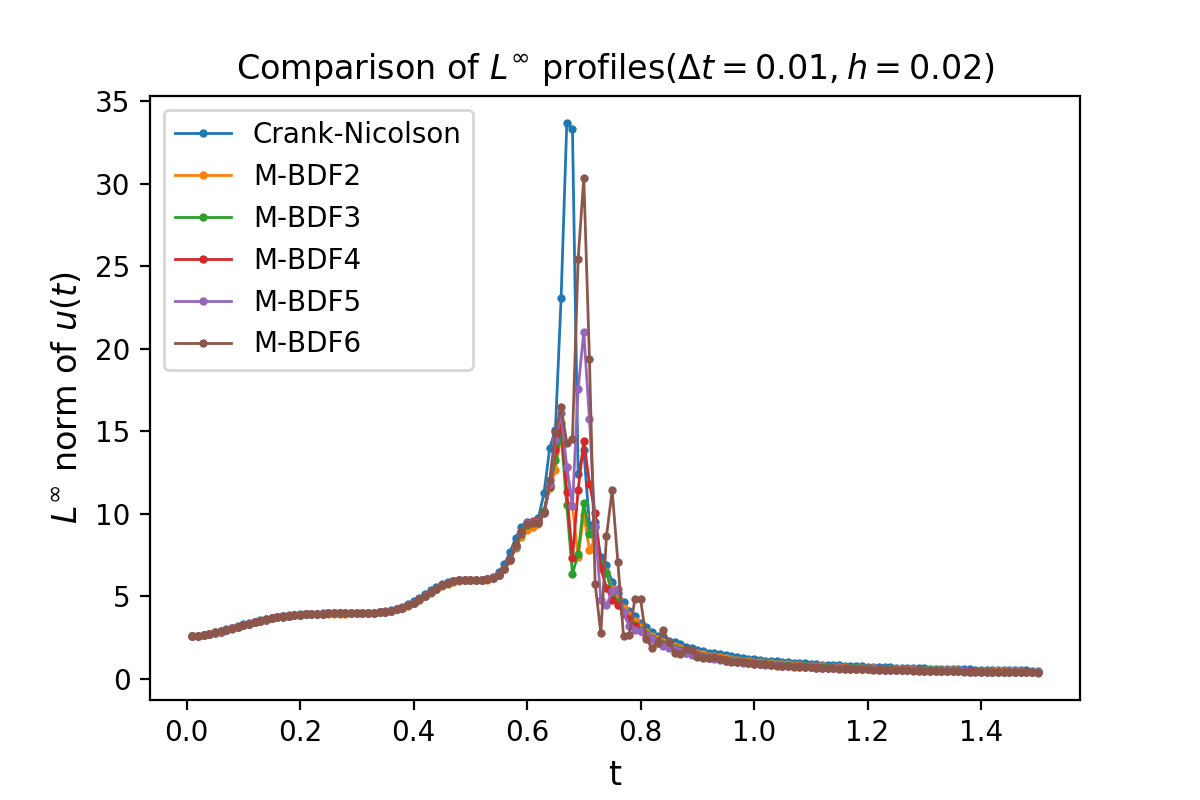}}
\subfigure[]{
\label{fig_max:c} 
\includegraphics[height=1.75in,width=2.4in]{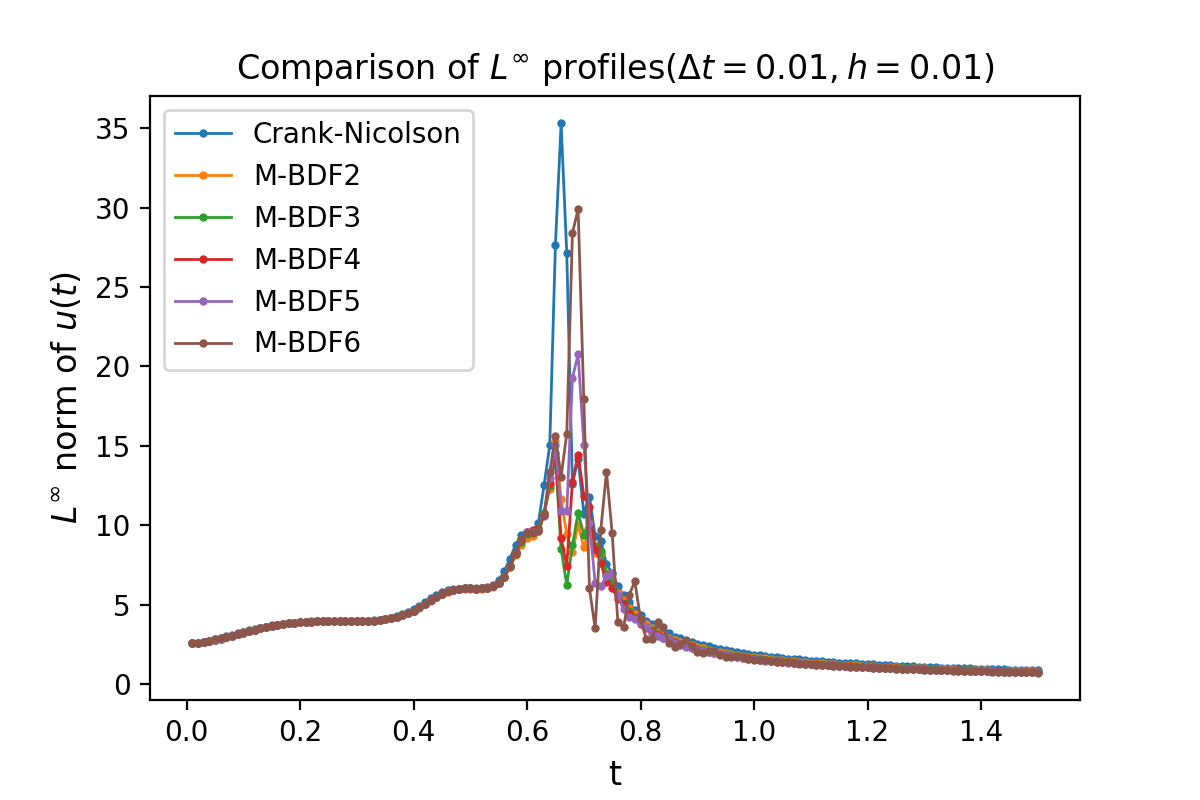}}
\subfigure[]{
\label{fig_max:d} 
\includegraphics[height=1.75in,width=2.4in]{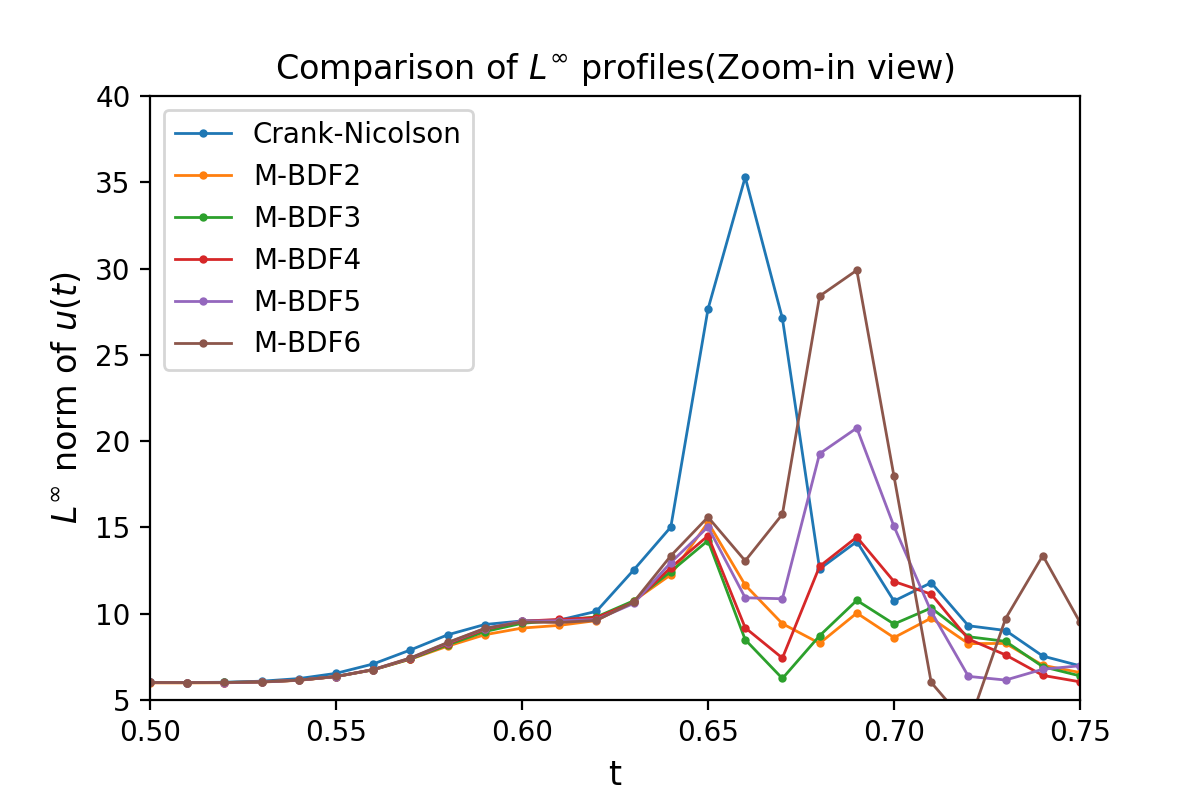}}
\caption{The comparison of $\Li$-norm profiles obtained by different schemes: 
\subref{fig_max:a} (including the Leapfrog scheme);  \subref{fig_max:b}--\subref{fig_max:d}
(excluding the Leapfrog scheme \subref{fig_max:b}--\subref{fig_max:d}).}
\label{fig_max}
\end{figure}

The further comparison of $\Li$-norm profiles obtained by other schemes (excluding the Leapfrog scheme) 
are shown in Figure~\ref{fig_max:b}--\ref{fig_max:d}. The tests in Figure~\ref{fig_max} 
indicate that the Crank-Nickson scheme and the modified BDF schemes are capable of 
capturing the blow-up phenomenon. 

In order to study whether the blow-up phenomenon will affect the mass- and 
energy-conservation results of Theorem~\ref{Co:conservation}, we present the time evolution 
of the mass and energy of $R^n$ in Figure~\ref{fig_m_e:a}--\ref{fig_m_e:b}. 
As expected, the mass of $R^n$ is exactly conserved and the energy 
is also conserved before and after the blow-up time, in spite of a sharply 
increase in energy at the blow-up time. The similar behaviors of $u^n$ are observed 
from Figure~\ref{fig_m_e:c}--\ref{fig_m_e:d}.

\begin{figure}[htb]
\centering
\subfigure[]{
\label{fig_m_e:a} 
\includegraphics[height=1.75in,width=2.4in]{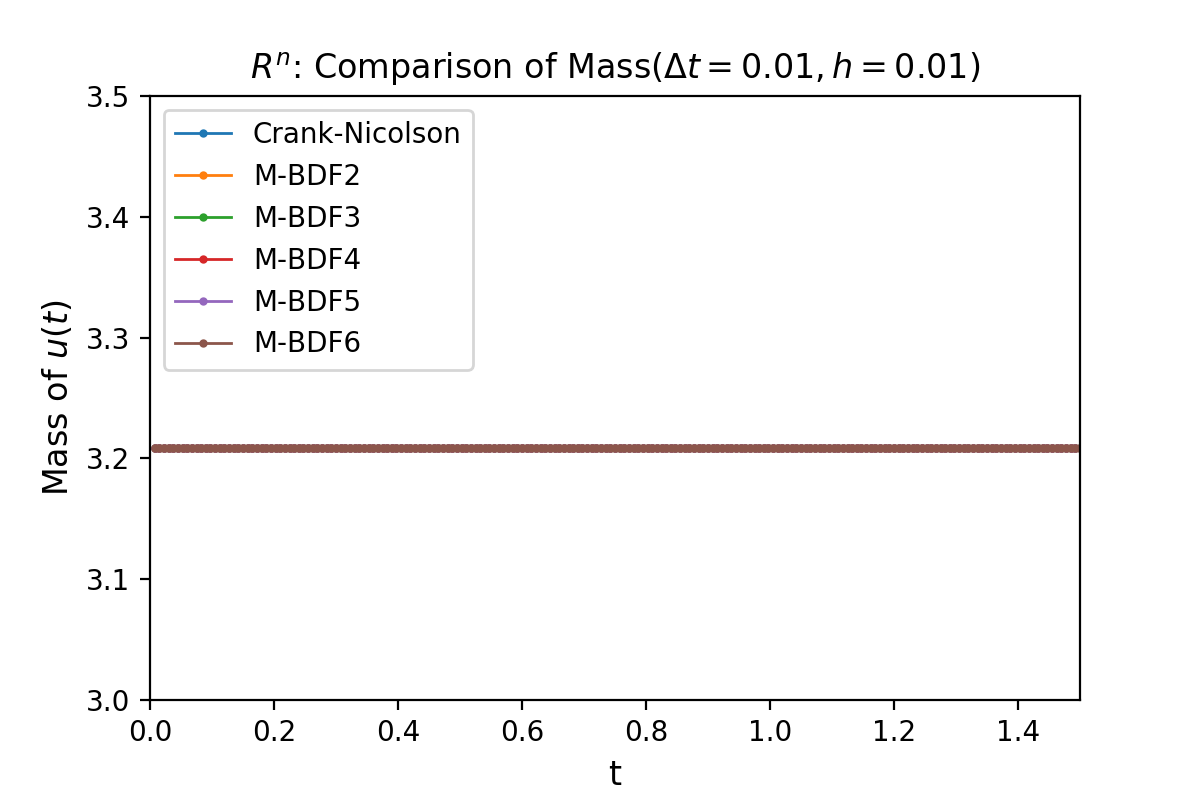}}
\subfigure[]{
\label{fig_m_e:b} 
\includegraphics[height=1.75in,width=2.4in]{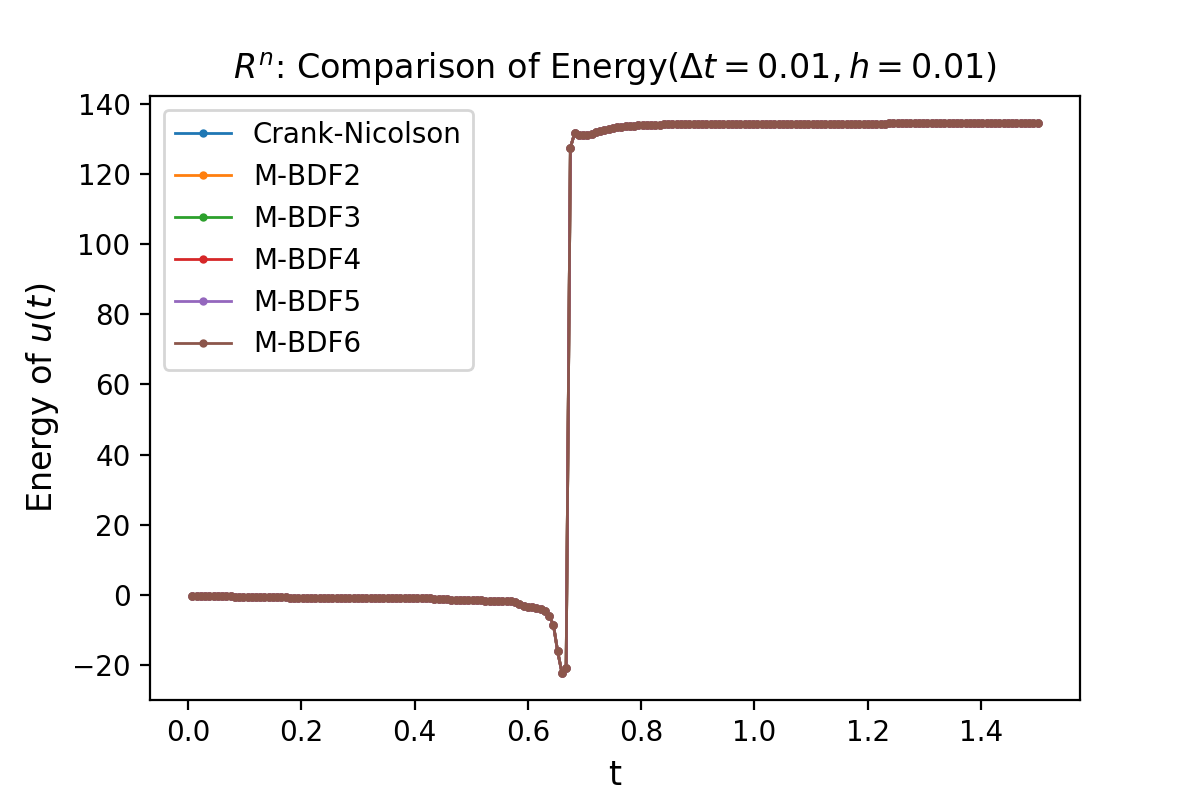}}
\subfigure[]{
\label{fig_m_e:c} 
\includegraphics[height=1.75in,width=2.4in]{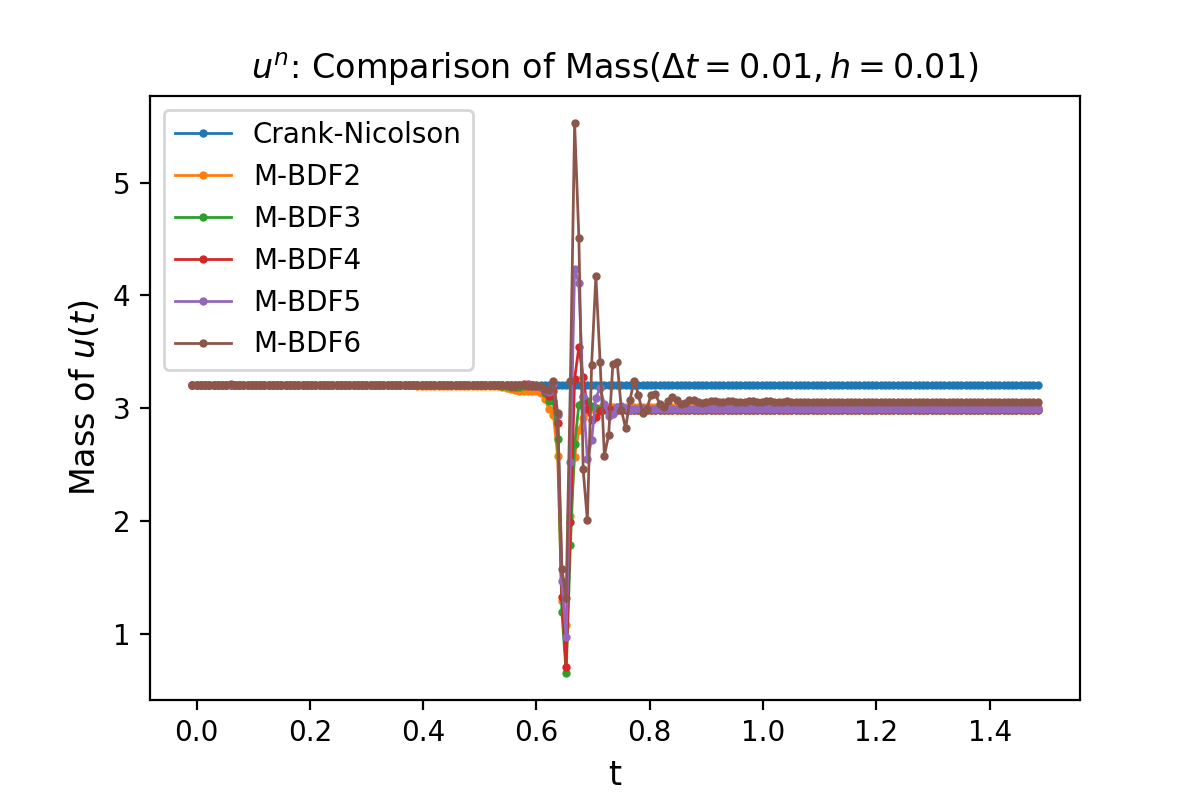}}
\subfigure[]{
\label{fig_m_e:d} 
\includegraphics[height=1.75in,width=2.4in]{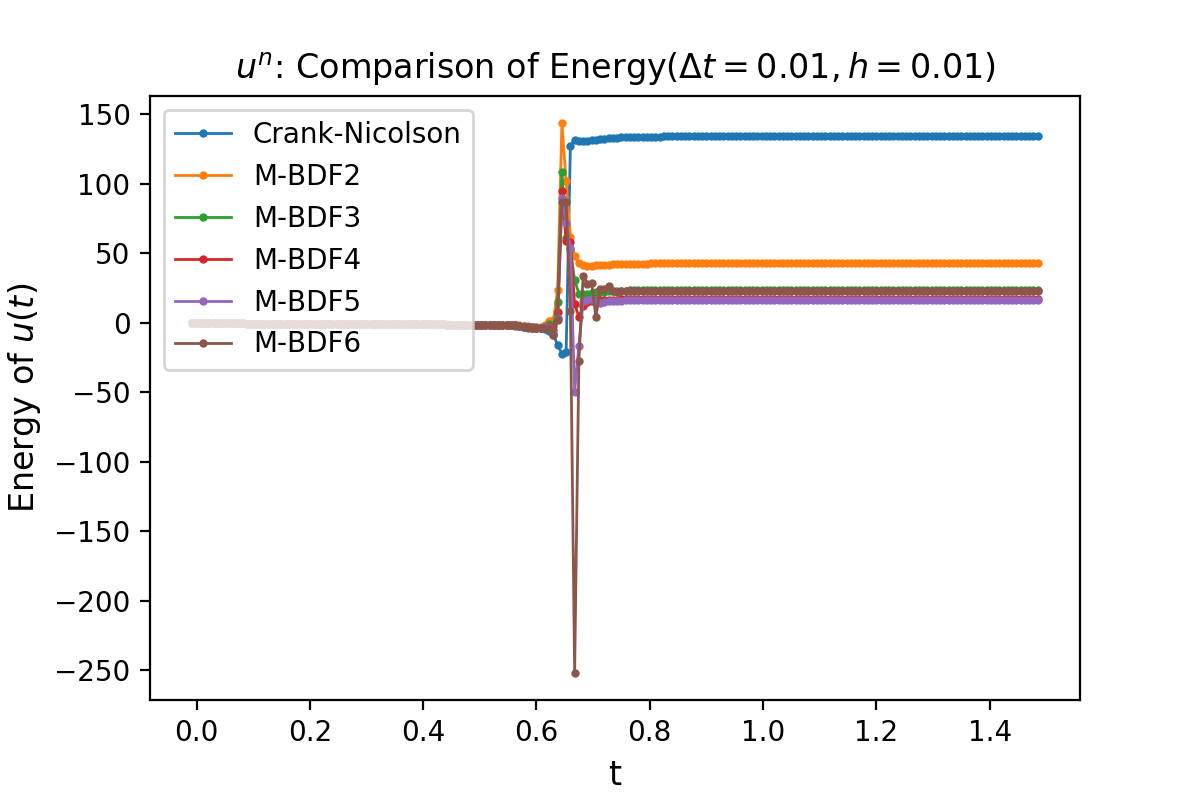}}
\caption{The time evolution of the computed mass and energy of $R^n$ and $u^n$ obtained by different 
schemes (excluding the Leapfog scheme).}
\label{fig_m_e}
\end{figure}

\subsection{Comparison of $\Li$--norm profiles obtained using different time-step sizes.}
To better understand the sensitivity of the blow-up simulations to the time-step size,
we analyze the capability of different schemes for capturing the blow-up phenomenon 
and provide three criteria for the blow-up time in this subsection.

The comparison of $\Li$--norm profiles obtained by different time-step sizes 
for the proposed time-stepping schemes are shown in Figure~\ref{fig_schemes_max}.
We observe similar behavior  for most schemes. 
Different simulation results for the Leapfrog scheme are shown in Figure~\ref{fig_schemes_max:b} 
and for the linearized scheme of \cite{Wang201New} in Figure~\ref{fig_schemes_max:h}, 
although small enough time-step sizes are used.
As already mentioned earlier, our numerical tests show that the 
Crank-Nickson scheme (see Figure~\ref{fig_schemes_max:a}) and the modified BDF schemes 
(see Figure~\ref{fig_schemes_max:c}--\subref{fig_schemes_max:g}) are capable of capturing the 
blow-up phenomenon. These results also consist with our previous results. 

\begin{figure}[htb]
\centering
\subfigure[]{
\label{fig_schemes_max:a} 
\includegraphics[height=1.75in,width=2.4in]{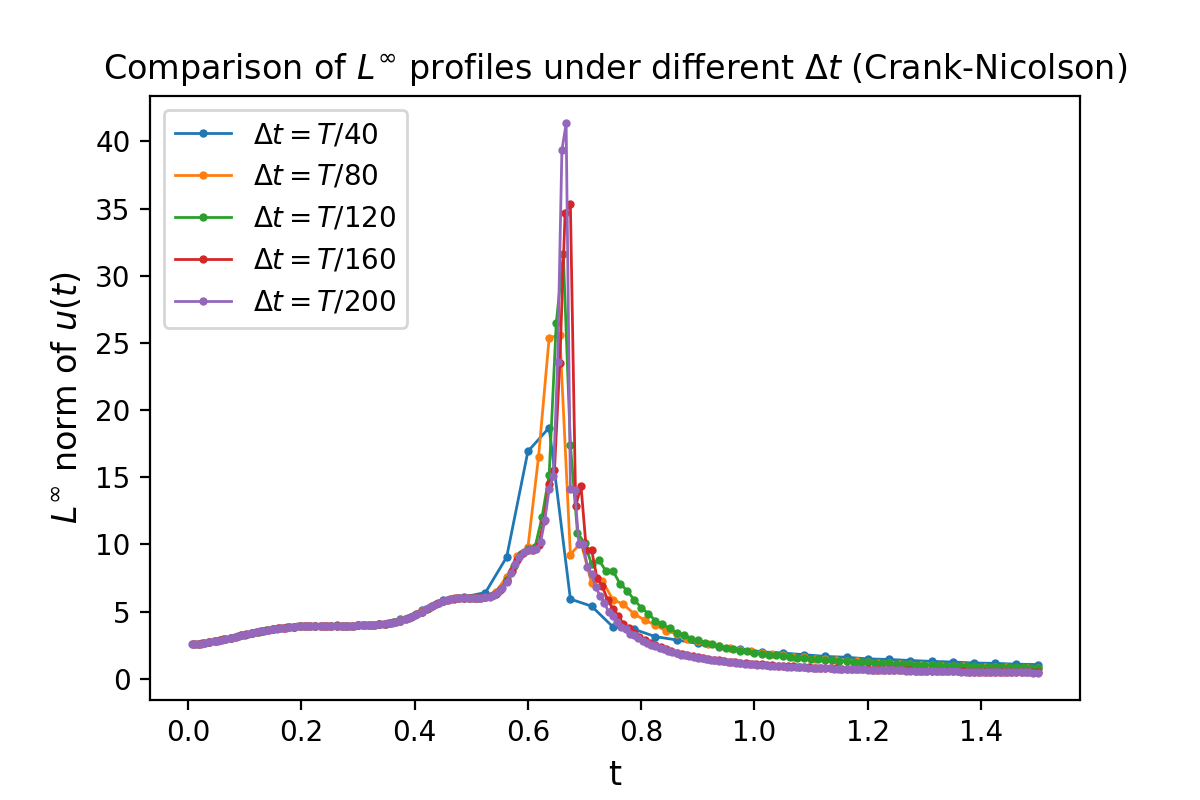}}
\subfigure[]{
\label{fig_schemes_max:b} 
\includegraphics[height=1.75in,width=2.4in]{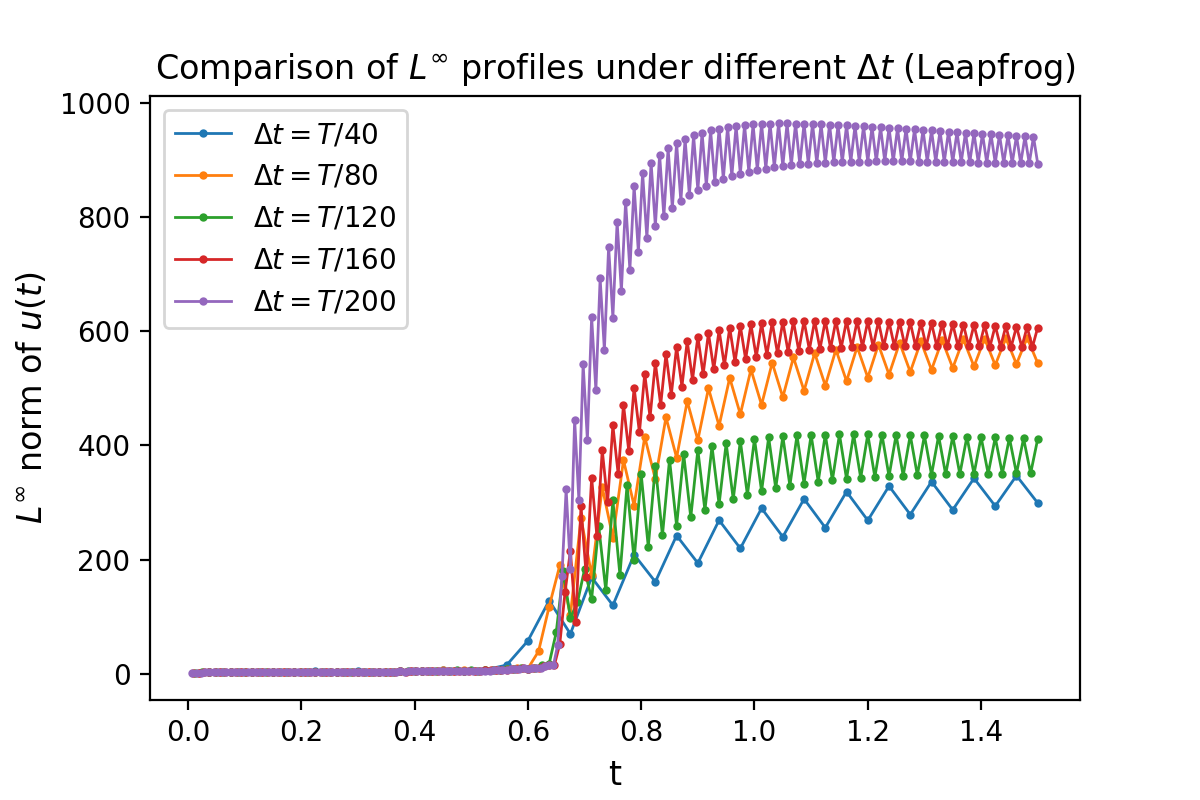}}
\subfigure[]{
\label{fig_schemes_max:c} 
\includegraphics[height=1.75in,width=2.4in]{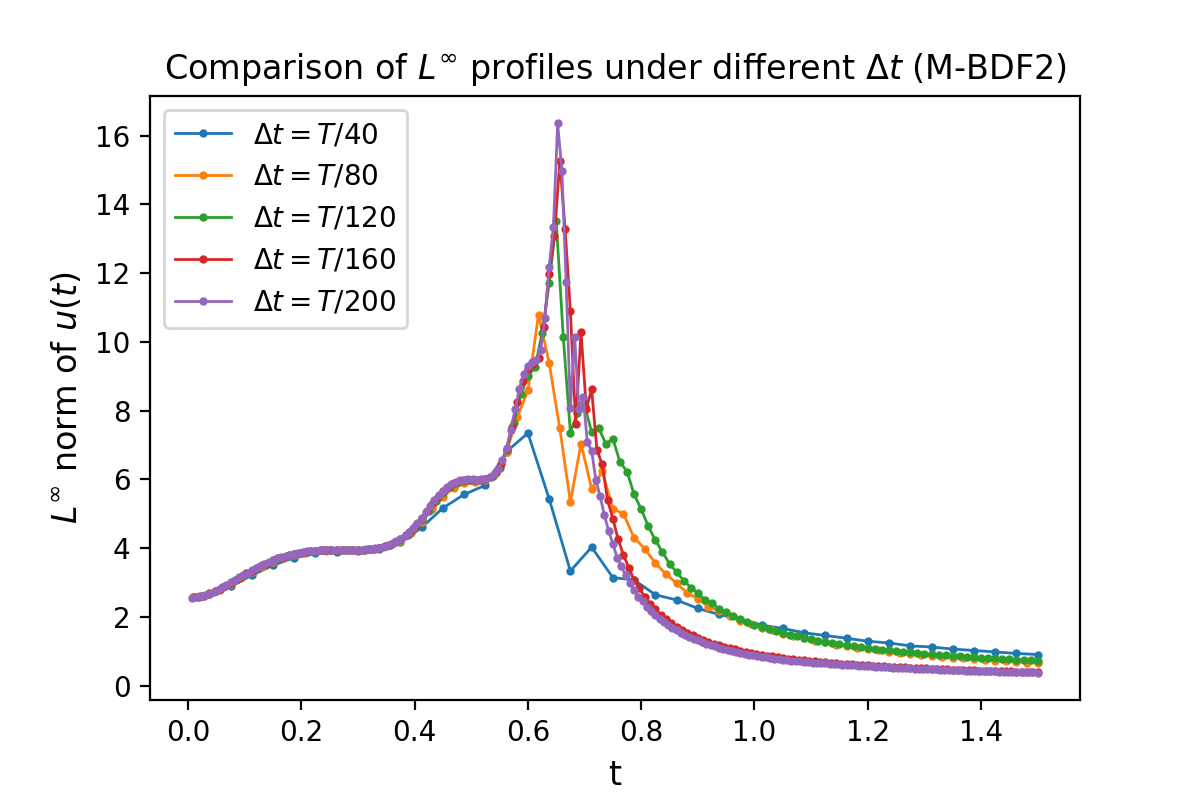}}
\subfigure[]{
\label{fig_schemes_max:d} 
\includegraphics[height=1.75in,width=2.4in]{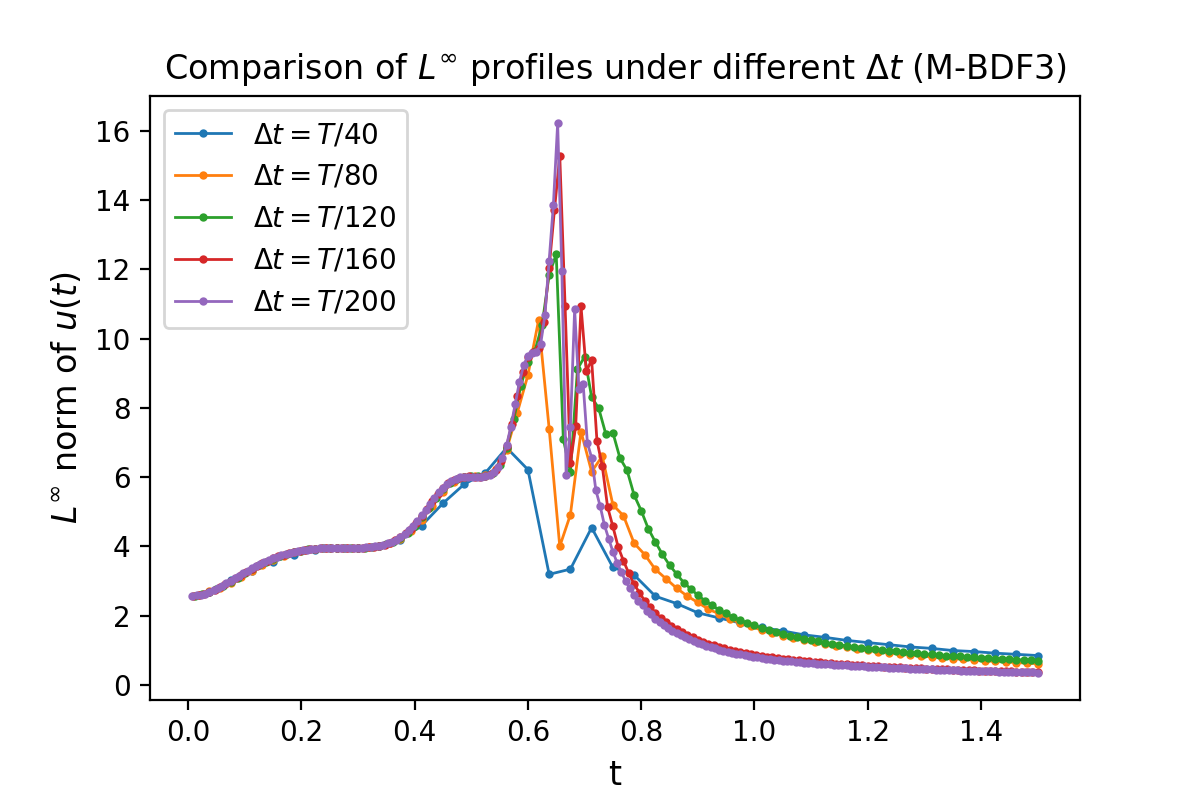}}
\subfigure[]{
\label{fig_schemes_max:e} 
\includegraphics[height=1.75in,width=2.4in]{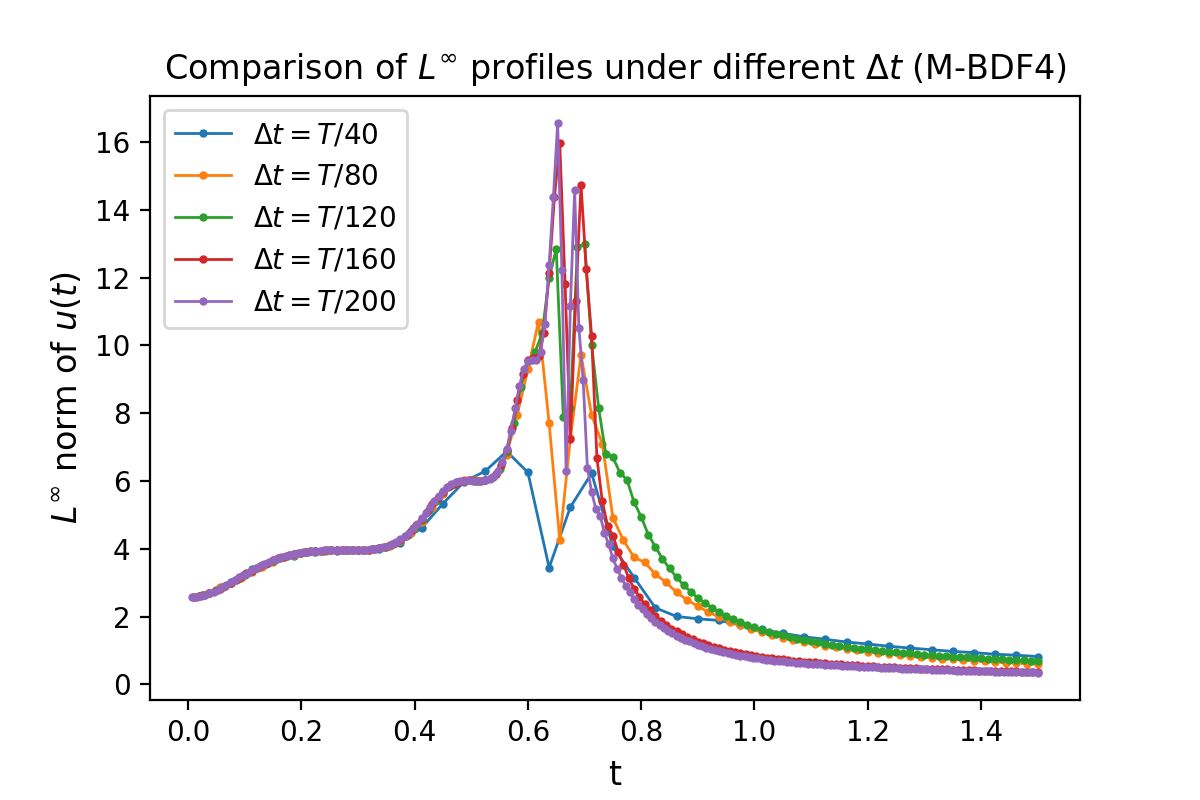}}
\subfigure[]{
\label{fig_schemes_max:f} 
\includegraphics[height=1.75in,width=2.4in]{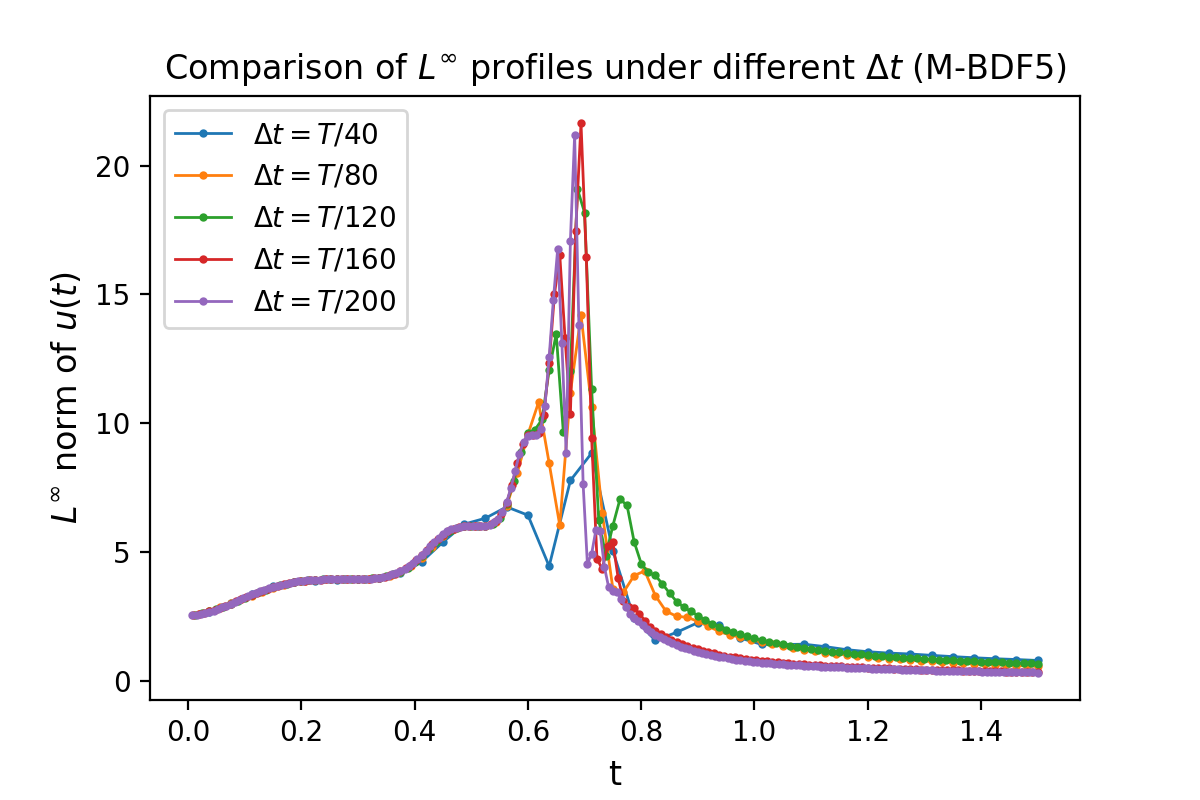}}
\subfigure[]{
\label{fig_schemes_max:g} 
\includegraphics[height=1.75in,width=2.4in]{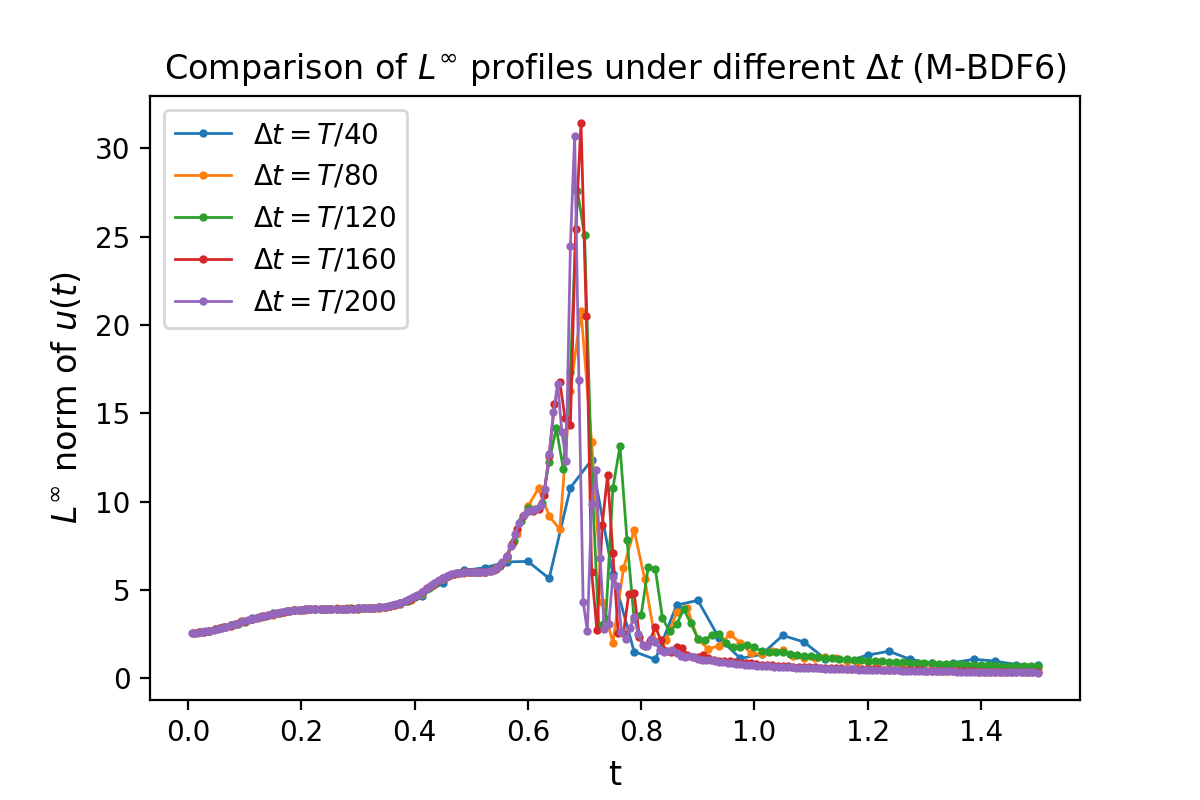}}
\subfigure[]{
\label{fig_schemes_max:h} 
\includegraphics[height=1.75in,width=2.4in]{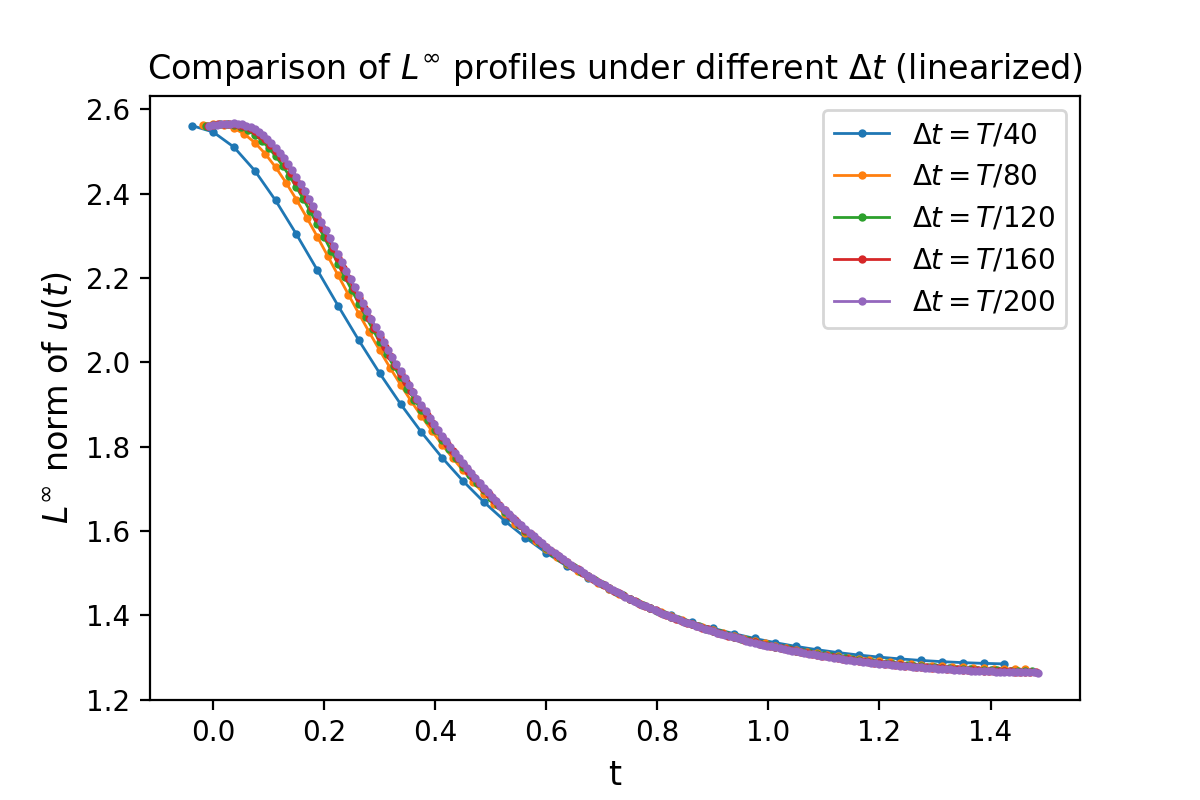}}
\caption{The comparison of $\Li$-norm profiles obtained by different $\Dt$, $h = 0.01$, 
where we also include the linearized scheme of \cite{Wang201New}.}
\label{fig_schemes_max}
\end{figure}

In order to provide some criterion for identifying the blow-up time, we first propose three different 
criteria and then to present a comparison of them on different schemes in Table~\ref{table_blowup}. 
The first criterion is to identify $t^{\max}$ corresponding to the time at which $\|u^n\|_{\Li}$ takes 
its maximum as shown in Figure~\ref{fig_blowup_point:a}, where $u^{\max}$ represents the maximum value.  
One difficulty with this criterion is that we may get different $t^{\max}$ with different schemes. 
To be specific, the modified lower order BDF schemes (i.e., M--BDF2, M--BDF3 and M--BDF4) identify 
the same earliest $t^{\max}$, while the modified higher order BDF schemes (i.e., M--BDF5 and M--BDF6) 
capture the same latest $t^{\max}$. In addition, the $t^{\max}$ found by the Crank-Nickson scheme 
is in the middle of the above two values, so it is inconclusive that which $t^{\max}$ is the most accurate.

However, it should be noted that the Crank-Nickson scheme is a preferable scheme for capturing the blow-up 
phenomenon because it finds the largest $\|u^{\max}\|_{\Li}$ as shown in Table~\ref{table_blowup}.

The other two criteria identify $t_1^n$ and $t_2^n$
in Figure~\ref{fig_blowup_point:b}, where $t_1^n$ represents the time  
point at which the energy of $R^n$ is the smallest and 
$t_2^n$ denotes the time point at which the energy of $R^n$ has the maximum increase. 
From Table~\ref{table_blowup} we observe that all time-stepping schemes identify the same blow-up time using 
these two criteria, which shows the robustness of both criteria.

\begin{figure}[htb]
\centering
\subfigure[]{
\label{fig_blowup_point:a}
\includegraphics[height=1.75in,width=2.4in]{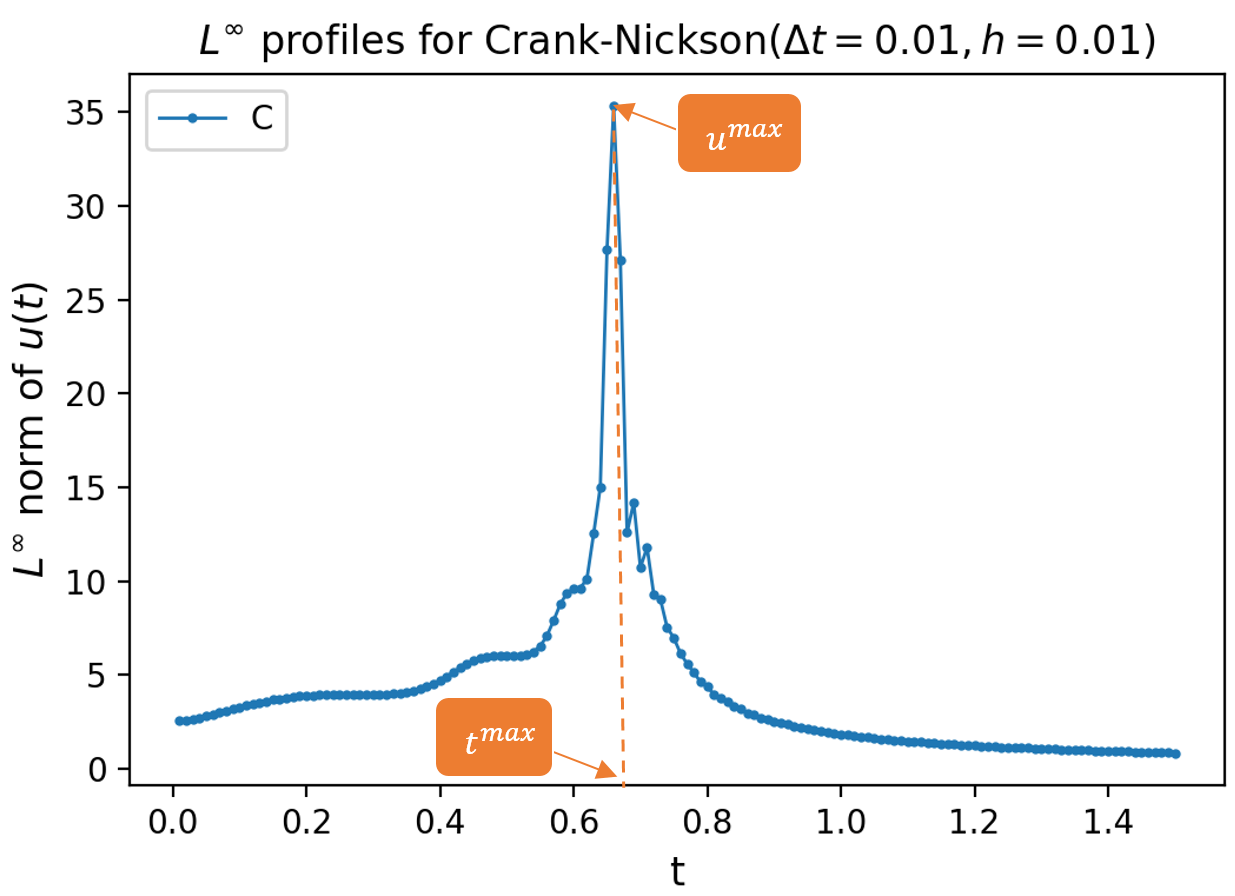}}
\subfigure[]{
\label{fig_blowup_point:b}
\includegraphics[height=1.75in,width=2.4in]{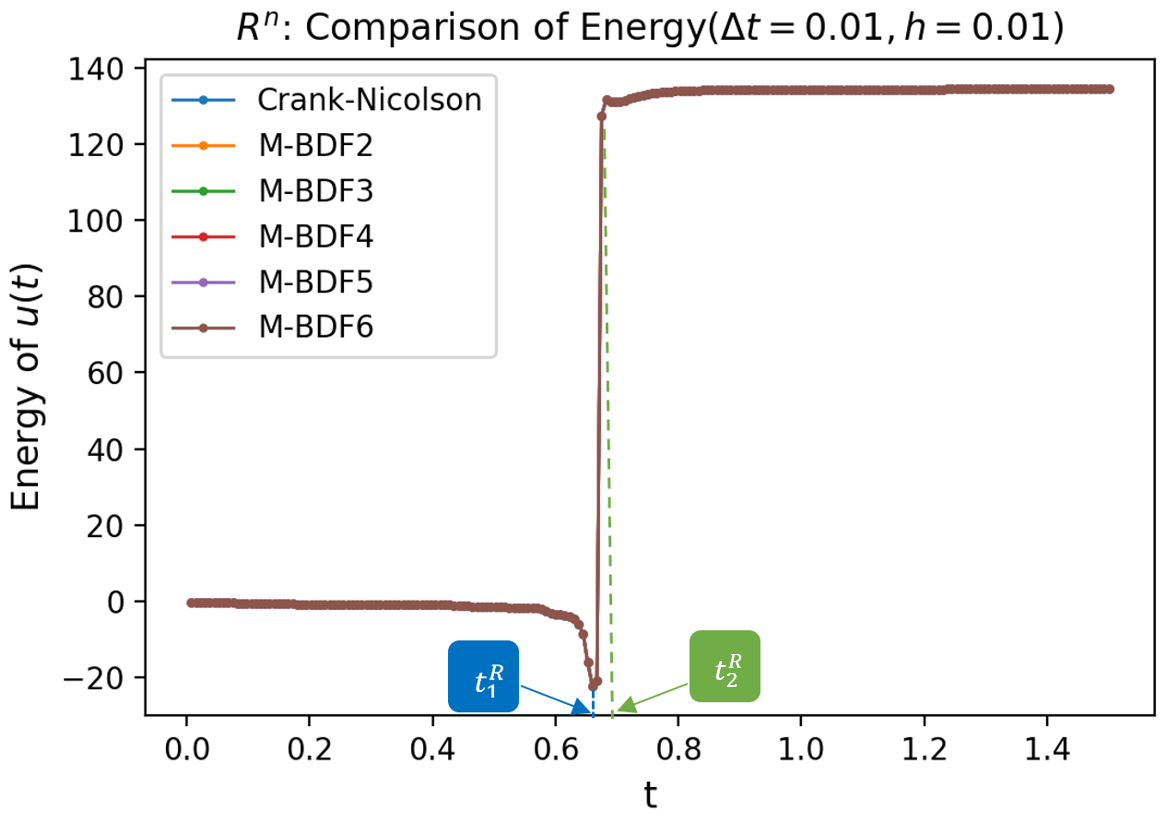}}
\caption{Three computed $t^{\max}$ identified from the data in Table~\ref{table_blowup}.}
\label{fig_blowup_point}
\end{figure}
\begin{table}[htb]\small\centering\footnotesize
\caption{Comparison of the blow-up times found by different methods 
(with $N_1 = 2000$ and $N_2 = 4000$).}
\begin{tabular*}{\hsize}{@{}@{\extracolsep{\fill}}clcccccccc@{}}
\toprule
\multirow{2}{*}{$\Dt$}&\multirow{2}{*}{Schemes}
&\multicolumn{2}{c}{$t^{max}$}&\multicolumn{2}{c}{$t_1^R$}
&\multicolumn{2}{c}{$t_2^R$}&\multicolumn{2}{c}{$u^{max}$ values}\\
\cmidrule(lr){3-4} \cmidrule(lr){5-6} \cmidrule(lr){7-8} \cmidrule(lr){9-10}
& & $N_1$&$N_2$& $N_1$&$N_2$& $N_1$&$N_2$& $N_1$&$N_2$ \\
\midrule
\multirow{6}{*}{0.02} &    C-N&                                0.62& 0.64& 
\multicolumn{2}{c}{\multirow{6}{*}{0.62}}& \multicolumn{2}{c}{\multirow{6}{*}{0.66}}& 24.6002& 24.6537\\
					  & M-BDF2& \multicolumn{2}{c}{\multirow{3}{*}{0.60}}&    & & & & 10.6163& 10.6164\\
					  & M-BDF3&                                    &     &    & & & & 10.3393& 10.3550\\
					  & M-BDF4&                                    &     &    & & & & 10.5676& 10.5936\\
					  & M-BDF5& \multicolumn{2}{c}{\multirow{2}{*}{0.68}}&    & & & & 14.3712& 14.5215\\
					  & M-BDF6&                                    &     &    & & & & 21.1538& 21.4550\\
\\
\multirow{6}{*}{0.01} &    C-N&                  \multicolumn{2}{c}{0.65}& 
\multicolumn{2}{c}{\multirow{6}{*}{0.65}}& \multicolumn{2}{c}{\multirow{6}{*}{0.67}}& 35.2965& 35.2392\\
					  & M-BDF2& \multicolumn{2}{c}{\multirow{3}{*}{0.64}}&    & & & & 15.3288& 15.3288\\
					  & M-BDF3&                                    &     &    & & & & 14.2396& 14.4062\\
				  	  & M-BDF4&                                    &     &    & & & & 14.5168& 14.7713\\
					  & M-BDF5& \multicolumn{2}{c}{\multirow{2}{*}{0.68}}&    & & & & 20.7615& 21.1336\\
					  & M-BDF6&                                    &     &    & & & & 29.9041& 30.7954\\
\\
\multirow{6}{*}{0.005}&    C-N&                               0.67& 0.675& 
\multicolumn{2}{c}{\multirow{6}{*}{0.67}}& \multicolumn{2}{c}{\multirow{6}{*}{0.68}}& 48.2457&48.2666\\
					  & M-BDF2&\multicolumn{2}{c}{\multirow{3}{*}{0.665}}&    & & & & 20.4022&20.4023\\
					  & M-BDF3&                                    &     &    & & & & 20.6082&20.8447\\
					  & M-BDF4&                                    &     &    & & & & 21.5725&22.2511\\
					  & M-BDF5&\multicolumn{2}{c}{\multirow{2}{*}{0.685}}&    & & & & 29.6170&30.2689\\
					  & M-BDF6&                                    &     &    & & & & 43.8004&44.7018\\
\bottomrule
\end{tabular*}
\label{table_blowup}
\end{table}

\section{Conclusion.}\label{sec-9}
In this paper we present a family of mass- and energy-conserved time-stepping 
schemes for general nonlinear Schr\"odinger equations. This includes the modified 
Crank-Nicolson scheme, Leapfrog scheme and modified BDF schemes as well  as 
a four-step symmetric scheme. We have shown that the proposed schemes have second-order 
convergence while preserving both mass and energy in the discrete setting without 
any mesh constraint. We also derive the dispersion relation equation for 
each of the proposed schemes and numerically show the convergence orders  
for the numerical dispersions. Extensive numerical experiments have been presented 
to illustrate the performance of the proposed schemes and to validate the theoretical 
results of the paper. Additional numerical experiments have also be provided 
to test the capability of the proposed schemes for capturing the blow-up phenomenon of the quintic 
nonlinear Schr\"odinger equation. Various criteria are proposed for identifying
the blow-up time and their effectiveness is also extensively examined. 
It is a bit disappointing that all proposed time-stepping schemes of this paper only have
second order accuracy and second order truncation errors. A very interesting question
is whether it is possible to improve these schemes into higher order schemes while 
still conserving both mass and (a modified) energy. Another challenging question
is whether it is possible to construct mass- and energy-conserved linear schemes (that is,  
only a linear problem needs to be solved at each time step). These 
open questions are worthy of further investigation and will be addressed in a further work.

\bigskip
\noindent {\bf  Acknowledgments.} 
The work of the first author was partially supported by the NSF Grant 
DMS-1620168, and the work of the second author was partially supported by NSF Grant
DMS-1812666. Part of the third author's work was done during a recent visit to the 
University of Tennessee at Knoxville (UTK), the author would like to thank Department 
of Mathematics of UTK for the support and hospitality. 
The visit was financially supported by a scholarship from the author's home institution,
Northwestern Polytechnical University of China.

\bibliographystyle{abrv}

\begin{thebibliography}{10}

\bibitem{Antoine2013Computational}
{\sc X.~Antoine, W.~Bao, and C.~Besse}, {\em Computational methods for the
	dynamics of the nonlinear {S}chr{\"o}dinger/Gross-Pitaevskii equations},
Computer Physics Communications, 184 (2013), pp.~2621--2633.

\bibitem{bao2013numerical}
{\sc W.~Bao, Q.~Tang and Z.~Xu},
{\em Numerical methods and comparison for computing dark and bright solitons in the nonlinear Schr{\"o}dinger equation},
Journal of Computational Physics,
235(2013), pp.~423--445.

\bibitem{bao2013Mathematical}
{\sc W.~Bao, Y.~Cai}, {\em Mathematical theory and numerical methods for Bose-Einstein condensation}, Kinetic $\&$ Related Models,
6(2013), pp.~1937-5093.
 
\bibitem{Bona2013Conservative}
{\sc J.~L. Bona, H.~Chen, O.~Karakashian, and Y.~Xing}, {\em Conservative,
	discontinuous {G}alerkin-methods for the generalized {K}orteweg-de {V}ries
	equation}, Mathematics of Computation, 82 (2013), pp.~1401--1432.

\bibitem{Bourgain1999Global}
{\sc J.~Bourgain}, {\em Global {S}olutions of {N}onlinear {S}chr{\"o}dinger
	{E}quations}, vol.~46, American Mathematical Society, 1999.

\bibitem{Calogero1988Spectral}
{\sc F.~Calogero and A.~Degasperis}, {\em Spectral {T}ransform and {S}olitons:
	{H}ow to {S}olve and {I}nvestigate {SN}nlinear {E}volution {E}quations},
Springer, New York, 1988.

\bibitem{Debussche2002Numerical}
{\sc A.~Debussche and L.~D. Menza}, {\em Numerical simulation of focusing
	stochastic nonlinear {S}chr\"odinger equations}, Physica D Nonlinear
Phenomena, 162 (2002), pp.~131--154.


\bibitem{Lax2010Integrals}
{\sc P.~D. Lax}, {\em Integrals of nonlinear equations of evolution and
	solitary waves}, Communications on Pure and Applied Mathematics, 21 (2010),
pp.~467--490.

\bibitem{Lin2017Numerical}
{\sc J.~Lin, Y.~Hong, L.-H. Kuo, and C.-S. Liu}, {\em Numerical simulation of
	3D nonlinear {S}chr{\"o}dinger equations by using the localized method of
	approximate particular solutions}, Engineering Analysis with Boundary
Elements, 78 (2017), pp.~20--25.

\bibitem{Liu2018On}
{\sc H.~Liu, Y.~Huang, W.~Lu, and N.~Yi}, {\em On accuracy of the
	mass-preserving DG method to multi-dimensional {S}chr{\"o}dinger equations},
IMA Journal of Numerical Analysis,  (2018).

\bibitem{Liu2016A}
{\sc H.~Liu and N.~Yi}, {\em A Hamiltonian preserving discontinuous {G}alerkin
	method for the generalized Korteweg-de Vries equation}, Journal of
Computational Physics, 321 (2016), pp.~776--796.

\bibitem{Lu2015Mass}
{\sc W.~Lu, Y.~Huang, and H.~Liu}, {\em Mass preserving discontinuous
	{G}alerkin methods for {S}chr{\"o}dinger equations}, Journal of Computational
Physics, 282 (2015), pp.~210--226.

\bibitem{Mclachlan2003Featured}
{\sc R.~Mclachlan}, {\em Featured review: Geometric numerical integration:
	Structure-preserving algorithms for ordinary differential equations}, Siam
Review, 45 (2003), pp.~817--821.

\bibitem{Miura1968Korteweg}
{\sc R.~M. Miura}, {\em Korteweg-de {V}ries equation and generalizations.
	{II}. existence of conservation laws and constants of motion.}, Journal of
Mathematical Physics, 9 (1968), pp.~1204--1209.

\bibitem{Pelinovsky1996Nonlinear}
{\sc D.~E. Pelinovsky, V.~V. Afanasjev, and Y.~S. Kivshar}, {\em Nonlinear
	theory of oscillating, decaying, and collapsing solitons in the generalized
	nonlinear {S}chr{\"o}dinger equation}, Physical Review E Statistical Physics
Plasmas Fluids and Related Interdisciplinary Topics, 53 (1996), p.~1940.

\bibitem{Sch1996Traveling}
{\sc H.~W. Sch{\"u}rmann}, {\em Traveling-wave solutions of the cubic-quintic
	nonlinear {S}chr{\"o}dinger equation}, Physical Review E Stat Phys Plasmas
Fluids Relat Interdiscip Topics, 54 (1996), pp.~4312--4320.

\bibitem{Sheu2015Dispersion}
{\sc T.~W. Sheu and L.~Lin}, {\em Dispersion relation equation preserving
	{FDTD} method for nonlinear cubic {S}chr{\"o}dinger equation}, Journal of
Computational Physics, 299 (2015), pp.~1--21.

\bibitem{Taghizadeh2011Exact}
{\sc N.~Taghizadeh, M.~Mirzazadeh, and F.~Farahrooz}, {\em Exact solutions of
	the nonlinear {S}chr{\"o}dinger equation by the first integral method},
Journal of Mathematical Analysis and Applications, 374 (2011), pp.~549--553.

\bibitem{Christopher1993Dispersion}
{\sc C.~K. Tam and J.~C. Webb}, {\em Dispersion-relation-preserving finite
	difference schemes for computational acoustics}, Journal of Computational
Physics, 107 (1993), pp.~262--281.

\bibitem{Tao2006Nonlinear}
{\sc T.~Tao}, {\em Nonlinear {D}ispersive {E}quations: {L}ocal and {G}lobal
	{A}alysis}, American Mathematical Society, 2006.

\bibitem{Wang201New}
{\sc J.~Wang}, {\em A new error analysis of {C}rank--{N}icolson {G}alerkin
	{FEM}s for a generalized nonlinear {S}chr{\"o}dinger equation}, Journal of
Scientific Computing, 60 (2014), pp.~390--407.

\bibitem{Xu2005Local}
{\sc Y.~Xu and C.~W. Shu}, {\em Local discontinuous {G}alerkin methods for
	nonlinear {S}chr{\"o}dinger equations}, Journal of Computational Physics, 205
(2005), pp.~72--97.


\bibitem{Zabusky1965Interaction}
{\sc N.~J. Zabusky and M.~D. Kruskal}, {\em Interaction of "solitons" in a
	collisionless plasma and the recurrence of initial states}, Physical Review
Letters, 15 (1965), pp.~240--243.

\bibitem{Zwillinger1989Handbook}
{\sc D.~Zwillinger}, {\em Handbook of {D}ifferential {D}quations, 3rd ed},
Boston, MA: Academic Press, 1997.

\end{thebibliography}
\bibliographystyle{plain}

\end{document}